\documentclass[a4paper,10pt,twoside,english]{amsart}
\usepackage[final]{microtype}
\usepackage[utf8]{inputenc}
\usepackage{amsthm}
\usepackage{amsmath}
\usepackage{eucal}
\usepackage[toc]{appendix}
\usepackage{enumitem}
\usepackage{tikz}
\usepackage[sort,nocompress]{cite}
\usetikzlibrary{automata,matrix,arrows,decorations.pathmorphing,calc}
\usepackage{tikz-cd}
\usepackage{tensor}
\usepackage{enumitem}
\usepackage{mathtools, mathdots}
\usepackage{amssymb,graphicx}
\usetikzlibrary{shapes}
\usepackage[colorinlistoftodos]{todonotes}
\usepackage{hyperref}
\usepackage{lipsum}
\usepackage{adjustbox}

\DeclareMathOperator{\defect}{\mathsf{defect}}
\DeclareMathOperator{\Kpac}{\catK_{\pac}}

\DeclareMathOperator{\PProj}{\mathsf{PProj}}
\DeclareMathOperator{\PInj}{\mathsf{PInj}}
\DeclareMathOperator{\Dpure}{\mathsf{D_{pure}}}

\DeclareMathOperator{\Qcoh}{\mathsf{Qcoh}}

\DeclareMathOperator{\Coloc}{\mathsf{Coloc}}

\DeclareMathOperator{\compacts}{\mathsf{c}}
\DeclareMathOperator{\bounded}{\mathsf{b}}

\DeclareMathOperator{\pac}{\mathsf{pac}}
\DeclareMathOperator{\id}{\mathsf{id}}

\DeclareMathOperator{\ev}{\mathsf{ev}}
\DeclareMathOperator{\op}{\mathsf{op}}
\DeclareMathOperator{\Kb}{\mathsf K^{\bounded}}

\DeclareMathOperator{\Kac}{\mathsf K_{\mathsf{ac}}}

\DeclareMathOperator{\catA}{\mathsf A}
\DeclareMathOperator{\catX}{\mathsf X}

\DeclareMathOperator{\catC}{\mathsf C}
\DeclareMathOperator{\catD}{\mathsf D}
\DeclareMathOperator{\catT}{\mathsf T}
\DeclareMathOperator{\catS}{\mathsf S}

\DeclareMathOperator{\Z}{\mathsf Z}
\DeclareMathOperator{\HH}{\mathsf H}
\DeclareMathOperator{\catK}{\mathsf K}
\DeclareMathOperator{\Ab}{\mathsf{Ab}}
\DeclareMathOperator{\Tot}{\mathsf{Tot}}

\DeclareMathOperator{\Loc}{\mathsf{Loc}}
\DeclareMathOperator{\Ker}{\mathsf{Ker}}
\DeclareMathOperator{\Imm}{\mathsf{Im}}
\DeclareMathOperator{\Cok}{\mathsf{Cok}}
\DeclareMathOperator{\modcat}{\mathsf{mod}}
\DeclareMathOperator{\Modcat}{\mathsf{Mod}}
\DeclareMathOperator{\Hom}{\mathsf{Hom}}
\DeclareMathOperator{\End}{\mathsf{End}}

\DeclareMathOperator{\HOM}{\mathcal{H}\mathsf{om}}

\DeclareMathOperator{\colim}{\mathsf{colim}}
\DeclareMathOperator{\limit}{\mathsf{lim}}
\DeclareMathOperator{\holim}{\mathsf{holim}}
\DeclareMathOperator{\hocolim}{\mathsf{hocolim}}
\DeclareMathOperator{\Gstable}{\underline{\mathsf{Gproj}}}
\DeclareMathOperator{\GStable}{\underline{\mathsf{GProj}}}

\DeclareMathOperator{\stablemodules}{\underline{\mathsf{\modcat}}}
\DeclareMathOperator{\stableModules}{\underline{\mathsf{\Modcat}}}
\DeclareMathOperator{\Ktac}{\mathsf K_{\mathsf{tac}}}

\DeclareMathOperator{\proj}{\mathsf{proj}}
\DeclareMathOperator{\Proj}{\mathsf{Proj}}
\DeclareMathOperator{\inj}{\mathsf{inj}}
\DeclareMathOperator{\Inj}{\mathsf{Inj}}

\DeclareMathOperator{\Db}{\mathsf D^{\mathsf b}}
\DeclareMathOperator{\Dsg}{\mathsf D_{\mathsf{sg}}}

\DeclareMathOperator{\perf}{\mathsf{perf}}

\DeclareMathOperator{\rad}{\mathsf{rad}}
\DeclareMathOperator{\thick}{\mathsf{thick}}

\DeclareMathOperator{\Cb}{\mathsf C^{\mathsf b}}

\DeclareMathOperator{\Gproj}{\mathsf{Gproj}}
\DeclareMathOperator{\GProj}{\mathsf{GProj}}

\DeclareMathOperator{\Cone}{\mathsf{Cone}}
\DeclareMathOperator{\coCone}{\mathsf{coCone}}

\DeclareMathOperator{\row}{\mathsf{row}}
\DeclareMathOperator{\bigR}{\mathsf{R}}
\DeclareMathOperator{\bigL}{\mathsf{L}}

\DeclareMathOperator{\GP}{\mathsf{GP}}
\DeclareMathOperator{\boundary}{\mathsf{B}}
\DeclareMathOperator{\CR}{\mathsf{CR}}

\DeclareMathOperator{\tac}{\mathsf{tac}}

\DeclareMathOperator{\MaxSpec}{\mathsf{MaxSpec}}

\renewcommand{\to}{\longrightarrow}
\renewcommand{\hookrightarrow}{\lhook\joinrel\longrightarrow}
\newcommand{\xto}{\xrightarrow}
\newcommand{\doublerightarrow}{\longrightarrow\mathrel{\mkern-14mu}\rightarrow}

\numberwithin{equation}{section}
\newtheorem{theorem}[equation]{Theorem}
\newtheorem{lemma}[equation]{Lemma}
\newtheorem{corollary}[equation]{Corollary}
\newtheorem{proposition}[equation]{Proposition}
\newtheorem*{theorem*}{Theorem}
\newtheorem*{introtheorem1}{Theorem~\ref{theorem partial Serre implies X compact and Y 0-cocompact}}
\newtheorem*{introtheorem2}{Theorem~\ref{theorem partial serre for Kb(mod R)}}
\newtheorem*{introtheorem3}{Theorem~\ref{theorem serre functors in K(Inj) and K(Proj)}}
\newtheorem*{introtheorem4}{Theorem~\ref{theorem our triangle summand of AR triangle}; Theorem~\ref{theorem.S triangle almost split}}
\newtheorem*{theorem1}{Theorem I}
\newtheorem*{theorem2}{Theorem II}
\newtheorem*{theorem3}{Theorem III}

\newtheorem*{1weakdual}{Theorem~\ref{thm:neemanthreeadjoints}${}^{\op}$}
\newtheorem*{2weakdual}{Theorem~\ref{theorem Neeman trick}${}^{\op}$}

\newtheorem*{theoremappendix}{Theorem~\ref{theorem main general result on partial serre}}
\newtheorem*{proposition*}{Proposition}
\newtheorem*{definition*}{Definition}
\newtheorem*{example*}{Example}
\newtheorem*{lemma*}{Lemma}
\newtheorem*{remark*}{Remark}
\theoremstyle{definition}
\newtheorem{observation}[equation]{Observation}
\newtheorem{definition}[equation]{Definition}
\newtheorem{example}[equation]{Example}
\newtheorem{construction}[equation]{Construction}
\newtheorem*{construction*}{Construction}
\newtheorem{remark}[equation]{Remark}

\newtheorem*{notation*}{Notation}
\newtheorem*{ackn}{Acknowledgments}

\let\amsamp=&

\begingroup
\catcode`\&=13
\gdef\pampmatrix{%
  \begingroup
  \let&=\amsamp
  \begin{smallmatrix}%
}
\gdef\endpampmatrix{\end{smallmatrix}\endgroup}
\endgroup

\begin{document}

\title[Partial Serre duality and cocompact objects]{Partial Serre duality and cocompact objects}
\author[Oppermann, Psaroudakis and Stai]{Steffen Oppermann, Chrysostomos Psaroudakis and Torkil Stai}
\address{Institutt for matematiske fag, NTNU, 7491 Trondheim, Norway}
\email{steffen.oppermann@ntnu.no}
\email{torkil.stai@ntnu.no}
\address{Department of Mathematics, Aristotle University of Thessaloniki, Thessaloniki 54124, Greece}
\email{chpsaroud@math.auth.gr}
\subjclass[2020]{18G80; 18E40; 16G50; 16G70; 16E35; 16E30; 16E65.}
\keywords{Cocompact object, Cocompactly cogenerated triangulated category, Partial Serre functor, Brown representability, Almost split triangle, Pure resolution, Gorenstein projective module, Singularity category, Homotopy category of injectives and of projectives.}

\dedicatory{Dedicated to Henning Krause on the occasion of his sixtieth birthday}

\thanks{The first and third named authors are supported by Norges forskningsr{\aa}d (grants 250056 and 302223).}

\begin{abstract}
A successful theme in the development of triangulated categories has been the study of compact objects. A weak dual notion called $0$-cocompact objects was introduced in \cite{MR3946864}, motivated by the fact that sets of such objects cogenerate co-t-structures, dual to the t-structures generated by sets of compact objects. In the present paper, we show that the notion of $0$-cocompact objects also appears naturally in the presence of certain dualities.

We introduce ``partial Serre duality'', which is shown to link compact to 0-cocompact objects. We show that partial Serre duality gives rise to an Auslander--Reiten theory, which in turn implies a weaker notion of duality which we call ``non-degenerate composition'', and throughout this entire hierarchy of dualities the objects involved are \( 0 \)-(co)compact.

Furthermore, we produce explicit partial Serre functors for multiple flavors of homotopy categories, thus illustrating that this type of duality, as well as the resulting \( 0 \)-cocompact objects, are abundant in prevalent triangulated categories.
\end{abstract}

\maketitle

\setcounter{tocdepth}{1}
\tableofcontents
\section{Introduction}
The concept of a triangulated category, introduced independently by Verdier \cite{MR1453167} and Puppe \cite{MR0150183}, appears naturally in various branches of mathematics. It is omnipresent in areas like algebraic geometry, stable homotopy theory, and representation theory; triangulated categories give a common framework for doing modern homological algebra in very different contexts.

Now, it might happen that we want to examine a triangulated category $\catT$ which is somehow `too big', making it hopless to really understand certain facets off the bat. For instance, $\catT$ might have coproducts. In this scenario compact objects can be helpful: Following Beligiannis--Reiten~\cite{MR2327478}, any set $\catX$ of compact objects gives rise to a decomposition of the ambient category, in the form of a (stable) t-structure
\begin{align} \label{align the classical t-structure}
\left({}^\perp(\catX^\perp), \catX^\perp\right). \tag{$\text{t}_1$}
\end{align}

If $\catT$ is even generated by compact objects, then more tools become available: Neeman~\cite{MR1308405} proved that $\catT$ satisfies Brown representatibility, which in turn unveils a fairly constructive localization theory. This covers many categories that occur in nature, for instance derived categories of rings and of (most) schemes are compactly generated.

Alas, the naive dual approach leads to a rather empty theory, in the sense that the resulting notion of cocompactness seldom appears in categories we wish to study: We show in Theorem~\ref{theorem cocompacts in derived category} that if $\catA$ is Grothendieck abelian, then the only cocompact object in $\catD(\catA)$ is $0$. On the other hand, non-trivial cocompact objects can only exist in \( \catK (\Modcat R) \) if we put certain subtle restrictions on the underlying set-theory (see Remark~\ref{rem cocompacts in homotopy category}).

So it becomes a natural goal to identify a more applicable variant of cocompactness, that is, a notion which not only allows for far-reaching theorems, but also actually shows up in categories one might be interested in. %This problem is yet to be solved.

In \cite{MR3946864}, coveting a more potent dual of (\ref{align the classical t-structure}), we introduced the weaker notion of 0-cocompactness ad hoc, and showed that if $\catX$ is a set of $0$-cocompact objects, then
\begin{align} \label{align our t-structure}
\left({}^\perp\catX, ({}^\perp\catX)^\perp\right) \tag{$\text{t}_2$}
\end{align}
is a (stable) t-structure in $\catT$. The point was that the definition was forgiving enough to apply to the homotopy categories studied in that paper.

In fact, the analogy to the classical theory has recently been strengthened with a Brown representability theorem for $0$-cocompactly cogenerated triangulated categories \cite{modoi2020weight}. Appendix~\ref{section dual brown rep} offers an enhanced version which additionally provides an explicit construction of the representing objects.

\medskip

In the current manuscript we explain how $0$-cocompact objects---as well as their duals, the $0$-compacts---come up in connection with certain dualities. In particular, there is an abundance of such objects in several categories of interest.

Below we present the main results of the paper, divided into three areas. All categories are $k$-categories for some commutative ring $k$. In this brief introduction we suppress any assumptions on existence of (co)products in the categories that appear; the full picture can be found in the pertinent sections.

\subsection*{I. Partial Serre duality}

The study and use of Serre functors goes back to the work of Bondal--Kapranov \cite{MR1039961} on mutations of exceptional collections. Since its introduction, the concept has become important in several areas. For instance, in representation theory the presence of a Serre functor is equivalent to the existence of almost split triangles \cite{MR910167,MR1112170}, and identifying objects along the Serre functor is the idea behind cluster categories \cite{MR2249625}; for geometers the Serre functor is a most utile weapon for handling the derived category of coherent sheaves \cite{MR1818984}. This classical concept of a Serre functor is limited to triangulated categories which are $\Hom$-finite over some field; we extend the notion in the following sense.

Fix an injective cogenerator $I$ of $\Modcat k$, and write $(-)^\ast=\Hom_k(-,I)$. A \emph{partial Serre functor} for a subcategory $\catX$ of a triangulated category $\catT$, is a functor $\mathbb S\colon \catX\to\mathbb \catT$ such that \[\catT(X,T)^\ast\cong\catT(T,\mathbb SX)\] naturally in $X\in\catX$ and in $T\in\catT$.

Over a field $k$, \cite{BIKP,MR3914144} established such duality formulae in certain singularity categories of algebras and in stable module categories of finite group schemes. On the other hand, for $k=\mathbb Z$ and $I=\mathbb Q/\mathbb Z$ the object $\mathbb SX$ is known by topologists as the `Brown--Comenetz dual' of $X$ \cite{MR405403}.

Experts may recall that the fact that a Serre functor in the sense of \cite{MR1039961} is exact, is a non-trivial result. In Theorem~\ref{theorem main general result on partial serre} (proven in Appendix~\ref{section exactness of serre}) we offer an enhancement to our setup: The collection $\catX$ of all the `Serre dualizable objects' in $\catT$ is a triangulated subcategory. Moreover, there is a partial Serre functor $\mathbb S \colon\catX\to\catT$ which is exact.

However, our main motivation for studying this concept is that it links the notions of compactness and 0-cocompactness.

\begin{introtheorem1}%[Theorem~\ref{theorem partial Serre implies X compact and Y 0-cocompact}]
If $\mathbb S\colon \catX\to\catT$ is a partial Serre functor, then $\catX$ consists of compact objects while $\mathbb S(\catX)$ consists of 0-cocompact objects.
%If two objects $X$ and $Y$ in $\catT$ satisfy
%\[ \catT(X,-)^\ast \cong \catT(-,Y),\]
%then $X$ is compact and $Y$ is $0$-cocompact.
\end{introtheorem1}

And so begins the important task of finding examples of partial Serre functors; this becomes a method for identifying 0-cocompact objects in practice.

\subsection*{II. Homotopy categories} If $\catT$ is a compactly generated triangulated category, then by Brown representability each compact object is Serre dualizable, and so there is a partial Serre functor $\mathbb  S\colon\catT^{\compacts}\to\catT$. It follows that $\catT$ is also 0-cocompactly cogenerated, by the essential image $\mathbb S(\catT^{\compacts})$ --- see Corollary~\ref{corollary compact generation gives partial serre}.

However, such an approach does not reveal any explicit information about the induced 0-cocompact objects. We would really prefer to actually \emph{construct} partial Serre functors, not least in categories where Brown representability fails. %As Brown representability typically fails in homotopy categories, we might as well start there:

%However, this approach does not in general reveal any explicit information about the representing 0-cocompact objects; these are merely known to exist. For our purposes, it would be more interesting to actually \emph{construct} partial Serre functors. %As Brown representability typically fails in homotopy categories, we might as well start there:

Let $R$ be any ring (if no other base ring is available, we can always choose $k=\mathbb Z$). Write $(-)^\vee=\Hom_R(-,R)$ and $\nu=(-^\vee)^\ast$, and let $\modcat R$ be the category of finitely presented $R$-modules.

\medskip

Each bounded complex $M$ over $\modcat R$ appears in an exact sequence \[P_1\stackrel{p}\to P_0\to M  \to 0,\] where the $P_i$ are contractible and belong to $\Cb(\proj R)$. Indeed, this is nothing but a projective presentation in the category of complexes. Let $\mathbb S_{\mathsf M} M = \Ker\left(\nu(p)\right)[2]$.

%Let $M$ be a bounded complex over $\modcat R$ and let \[P_1\stackrel{p}\to P_0\to M  \to 0\] be a projective presentation in $\Cb(\modcat R)$, which means that the $P_i$ are contractible complexes over $\proj R$.

\begin{introtheorem2}%[Theorem~\ref{theorem partial serre for Kb(mod R)}]
  Let $R$ be a ring. Then $\mathbb S_{\mathsf M}$ defines a partial Serre functor \[\mathbb  S_{\mathsf M} \colon  \Kb(\modcat R)\to\catK(\Modcat R)\]
\end{introtheorem2}

In Proposition~\ref{proposition AR translate without contractibles} we show how to calculate $\mathbb S_{\mathsf M}M$ using just complexes of projectives, and not projective objects in the category of complexes. In particular, if $M$ is a module, then the 0-cocompact object $\mathbb S_{\mathsf M} M$ is simply the complex \[\tau M\hookrightarrow \nu P'_1 \to \nu P'_0,\] where $P'_1\to P'_0\to M\to0$ is a projective presentation in $\modcat R$ and $\tau M$ is the `usual' AR-translate of $M$.

By construction, the colocalizing subcategory of $\catK(\Inj R)$ generated by the essential image $\mathbb S_{\mathsf M}(\modcat R)$, consists of complexes of pure-injectives. This spurs Corollary~\ref{corollary existence of pure resolutions}, which shows that each complex of $R$-modules admits a pure-injective (and a pure-projective) resolution.

%Inside $\catK(\Modcat R)$, the localizing category generated by $\modcat R$ consists of complexes of pure-projectives. Dually, the colocalizing category generated by the essential image $\mathbb S_{\mathsf M}(\modcat R)$ consists of complexes of pure-injectives. In combination with the t-structures (\ref{align the classical t-structure}) and (\ref{align our t-structure}), this reveals in particular that each complex of $R$-modules admits a pure-projective and a pure-injective resolution---see Corollary~\ref{corollary existence of pure resolutions}.

On the other hand, if $\Lambda$ is an Artin algebra then, choosing $I$ as an injective envelope of the semisimple $k/\!\rad k$, the functor $\mathbb S_{\mathsf M}$ becomes an auto-equivalence. In particular, $\Kb(\modcat \Lambda)$ is a set of 0-cocompact objects in $\catK(\Modcat\Lambda)$ (Observation~\ref{observation 0-cocompacts in Kb(Lambda)}).

%Suppose $\Lambda$ is an Artin algebra, i.e.\ $k$ is artinian and $\Lambda$ is finitely generated as $k$-module. Then, choosing $I$ as an injective envelope of $k/\!\rad k$, the functor $\mathbb S_{\mathsf M}$ becomes an auto-equivalence. In particular, $\Kb(\modcat \Lambda)$ is a set of 0-cocompact objects in $\catK(\Modcat\Lambda)$ (Observation~\ref{observation 0-cocompacts in Kb(Lambda)}).
%\todo{where to put $\Kb(\modcat \Lambda)$ has almost split triangles (Corollary~\ref{corollary Kb(Lambda) has AR})?}

\medskip

Recall from \cite{MR3212862,MR2923949} that the inclusion $\catK(\Inj R)\hookrightarrow\catK(\Modcat R)$ admits a left adjoint $\lambda$. If $X$ is a left bounded complex, then $\lambda X$ is an injective resolution of $X$.

Dually, by e.g.\ \cite{MR2680406} the inclusion $\catK(\Proj R)\hookrightarrow\catK(\Modcat R)$ admits a right adjoint $\rho$. If $X$ is a right bounded complex, then $\rho X$ is a projective resolution of $X$.

\begin{introtheorem3}%[Theorem~\ref{theorem serre functors in K(Inj) and K(Proj)}]
Let \( R \) be a ring, and let \( \catK^{\bounded}_{\mathsf{fpr}}(R) \) be the subcategory of \( \Kb(\Modcat R) \) consisting of complexes admitting degree-wise finitely generated projective resolutions.
\begin{enumerate}
\item There is a partial Serre functor \( \mathbb S_{\mathsf I} \colon \lambda( \catK^{\bounded}_{\mathsf{fpr}}(R)) \to \catK(\Inj R) \) given by \[\mathbb S_{\mathsf I}(\lambda X) = \nu \rho X. \]
\item There is a partial Serre functor \( \mathbb S_{\mathsf P} \colon \rho(\catK^{\bounded}_{\mathsf{fpr}}(R^{\op}))^{\vee} \to \catK(\Proj R) \) given by \[ \mathbb S_{\mathsf P}((\rho X)^{\vee}) = \rho(X^\ast). \]
\end{enumerate}
\end{introtheorem3}

Recently, the authors of \cite{BIKP} showed that if $A$ is a Gorenstein algebra which is finite dimensional over a field, then there are Serre functors \[ \mathbb S_{\mathsf{sg}}\colon\Dsg(A)\to\Dsg(A) \text{ and } \mathbb S_{\mathsf G}\colon\Gstable A  \to \Gstable A.\] We now realize that there is a bigger picture here: In Section~\ref{section transferring serre} we explain how partial Serre duality may be transported along an adjoint triple, and thus
\begin{itemize}
  \item if $R$ is a noetherian ring, then $\mathbb S_{\mathsf I}$ induces a partial Serre functor \[\mathbb S_{\mathsf{sg}} \colon \Dsg(R)\to\Kac(\Inj R) \text{ (Theorem~\ref{theorem partial serre functor on Dsg}) and }\]
  \item if $R$ also has a dualizing complex, then $\mathbb S_{\mathsf P}$ induces a partial Serre functor \[\mathbb S_{\mathsf{G}} \colon(\GStable R)^{\compacts}\to \GStable R \text{ (Theorem~\ref{theorem partial serre on gproj}). }\]
\end{itemize}

\subsection*{III. Auslander--Reiten theory} The concept of an almost split triangle in a triangulated category $\catT$, due to Happel \cite{MR910167}, is a powerful combinatorial tool: In fortunate cases, the collection of such triangles determines all the morphisms in $\catT$.

Existence theorems in this direction can thus be of some impact. For instance, if $\catT$ satisfies Brown representability and $X$ is a compact object with local endomorphism ring, then there is an almost split triangle
\begin{align*} \label{align intro almost split triangle}
  \tau X \to M \to X \to \tau X [1]. \tag{$\ast$}
\end{align*}
Here $\tau X [1]$ is the representing object of $\Hom_{\End(X)}(\catT(X,-), I_X)$, where $I_X$ is an injective envelope of the simple $\End(X)$-module---see e.g.\ Beligiannis~\cite{MR2079606} or Krause~\cite{MR1803642}. However, $\tau X$ can sometimes be calculated using a more global approach:

\begin{introtheorem4}%[Theorem~\ref{theorem our triangle summand of AR triangle}; Theorem~\ref{theorem.S triangle almost split}]
Suppose that the triangulated category $\catT$ is idempotent complete, and that $\mathbb S\colon\catX\to\catT$ is a partial Serre functor.

For each object $X$ in $\catX$ with local endomorphism ring, the triangle (\ref{align intro almost split triangle}) exists and appears as a summand of a triangle
\begin{align} \label{align our triangle}
\mathbb SX[-1]\to N \to X \to \mathbb SX. \tag{$\ast\ast$}
\end{align}

Suppose moreover that $k$ is noetherian, and that $I$ is the direct sum of the injective envelopes of the simple $k$-modules. If $\End(X)$ is $k$-finite, then (\ref{align our triangle})$=$(\ref{align intro almost split triangle}).
\end{introtheorem4}

By their definition, AR-triangles are completely self-dual. Partial Serre duality, on the  other hand, is not self-dual: The Serre dual of a compact object is just 0-cocompact. In Section~\ref{section non-degeneracy} we introduce the notion that \emph{composition from $X$ to $Y$ is non-degenerate} if any non-zero $\catT$-submodule of $\catT(X,-)$ or $\catT(-,Y)$ contains a non-zero morphism $X\to Y$.

In the light of Proposition~\ref{proposition comp to T(X,SX) is non-degenerate}, this is a weaker form of partial Serre duality: If $X$ admits a Serre dual, then composition from $X$ to $\mathbb S X$ is non-degenerate. Still, there is an existence criterion for almost split triangles even in these terms: Suppose that \( X \) and \( Y \) have local endomorphism rings, and that either $\End(X)$ or $\End(Y)$ is artinian. In Corollary~\ref{corollary non-deg implies almost split triangle} we show that if composition from $X$ to $Y$ is non-degenerate, then there is an almost split triangle of the form
\begin{align*} \label{align generic almost split triangle}
  Y[-1]\to E \to X \to Y. \tag{$\Delta$}
\end{align*}

Conversely, one might ask what the existence of such a triangle in general tells us about $X$. Clearly, composition from $X$ to $Y$ must be non-degenerate. Now the point is that this weaker form of duality \emph{is} self-dual, capturing more than classical compactness: In Theorem~\ref{theorem non-degeneracy implies 0-cocompactness with correct def} we show that if composition from $X$ to $Y$ is non-degenerate, then $X$ is 0-compact and $Y$ is 0-cocompact. It follows in particular that any object $X$ which appears in an almost triangle of the form (\ref{align generic almost split triangle}), is 0-compact---see Corollary~\ref{corollary almost split forces 0-(co)compactness}.

It might be helpful to summarize some of the connections between these versions of compactness and duality in a graph.

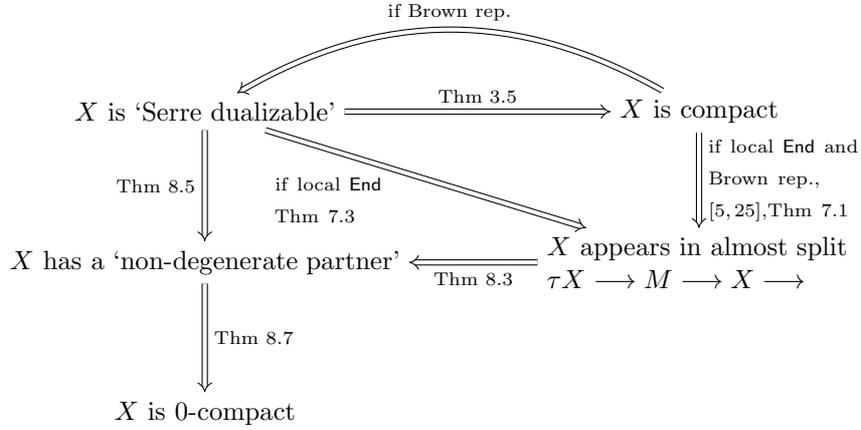
\[\begin{tikzpicture}
  \node (uu) at (6.5,2) {$X$ is compact};
  \node (ur) at (0,2) {$X$ is `Serre dualizable'};
  \node (ul) [text width=4cm] at (6.5,0) {$X$ appears in almost split $\tau X\to M \to X \to $};
  \node (d) at (0,0) {$X$ has a `non-degenerate partner'};
  \node (dd) at (0,-2) {$X$ is 0-compact};
  \draw[-implies, double equal sign distance] (uu) to node[right,text width=2cm]{\scriptsize if local \( \End \) and Brown rep., \cite{MR2079606,MR1803642},Thm~\ref{theorem Krause AR-theorem}}(ul);
  \draw[-implies, double equal sign distance,bend right=30] (uu) to node[above]{\scriptsize$\text{if Brown rep.}$}(ur);
  \draw[-implies, double equal sign distance] (ur) to node [above]{\scriptsize Thm~\ref{theorem partial Serre implies X compact and Y 0-cocompact}} (uu);
  \draw[-implies, double equal sign distance] (ul) to node [below] {\scriptsize Thm~\ref{theorem almost split vs non deg}} (d);
  \draw[-implies, double equal sign distance] (ur) to node [left] {\scriptsize Thm~\ref{proposition comp to T(X,SX) is non-degenerate}} (d);
  \draw[-implies, double equal sign distance] (d) to node [right] {\scriptsize Thm~\ref{theorem non-degeneracy implies 0-cocompactness with correct def}} (dd);
  \draw[-implies, double equal sign distance] (ur) to node [below left=-2mm,text width=2cm] {\scriptsize if local \( \End \) Thm~\ref{theorem our triangle summand of AR triangle}}(ul);
\end{tikzpicture}\]

\begin{ackn}
The authors thank Georgios Dalezios, Lidia Angeleri H\"ugel, Martin Kalck, Rosanna Laking, and Jorge Vit\'oria for discussions, comments and questions.
\end{ackn}

\section{(Co)compactness and 0-(co)compactness}

In a triangulated category $\catT$ with coproducts, an object $X$ is \emph{compact} if the natural morphism \[\coprod \catT(X,Y_i)\to\catT(X,\coprod Y_i)\] is invertible for each family $\{Y_i\}$. The subcategory of all compact objects in $\catT$ is denoted by $\catT^{\compacts}$. Dually---albeit appearing far less frequently in  the literature---if $\catT$ admits products, then an object $Y$ is \emph{cocompact} if the natural morphism \[\coprod\catT(X_i,Y)\to\catT(\prod X_i, Y)\] is an isomorphism for each collection $\{X_i\}$.

We will now recall the more recent notion of $0$-cocompactness as introduced in \cite{MR3946864}, together with its dual. A bit of terminology is involved:

For a class of objects $\catX$ in $\catT$, an object $G$ is said to be a \emph{contravariant $\catX$-ghost} if $\catT(G,\catX)=0$; a morphism $g$ is a \emph{contravariant $\catX$-ghost} if $\catT(g,\catX)=0$.  Dually, an object $G$ is a \emph{covariant $\catX$-ghost} if $\catT(\catX,G)=0$. In an abelian category, a diagram \[A_0 \stackrel{a_0}{\to} A_1 \stackrel{a_1}{\to} A_2 \to \cdots\] is \emph{dual Mittag-Leffler} (\emph{dual ML}) if the increasing chain $\Ker a_i \subset \Ker a_{i+1} a_i \subset \cdots$ stabilizes for each $i$.

%OLD: In \cite{MR3946864}, the notion of $0$-cocompactness was introduced. We now recall this concept, and give a reformulation in Proposition~\ref{proposition rephrased 0-cocompactness when shift invariant}. Let $\catX$ be a class of objects in a triangulated category $\catT$. An object $G$ is called a \emph{contravariant $\catX$-ghost} if $\catT(G,\catX)=0$; a morphism $g$ is a \emph{contravariant $\catX$-ghost} if $\catT(g,\catX)=0$.

\begin{definition} \label{definition 0-cocompactness and 0-compactness}
  An object $X\in\catT$ is \emph{$0$-cocompact} if $\holim\mathbb T$ is a contravariant $X$-ghost for each sequence \[\mathbb T\colon \cdots\to T_2\to T_1 \to T_0\] such that $\colim \catT(\mathbb T,X)=0$ and $\catT(\mathbb T,X[1])$ is dual ML.

  Dually, $X$ is called \emph{$0$-compact} if $\hocolim\mathbb T$ is a covariant $X$-ghost for each sequence \[\mathbb T\colon T_0\to T_1 \to T_2 \to \cdots\] such that $\colim \catT(X,\mathbb T)=0$ and $\catT(X[1],\mathbb T)$ is dual ML.
\end{definition}

It is clear that any (co)compact object is $0$-(co)compact; the question of the converse is more interesting. Indeed, the fact that $0$-cocompactness in practice does appear more often than cocompactness, is the mainspring of our work.

\begin{example} \label{example K(vect)}
Consider the category \( \catK( \Modcat k ) \) for a field \( k \). Then the only cocompact object is \( 0 \), while all objects are \( 0 \)-cocompact.

Since any complex is homotopy equivalent to a complex with zero differential we have \( \catK( \Modcat k ) \simeq (\Modcat k)^{\mathbb{Z}} \), and can mainly argue in the category of vector spaces. For the first claim, note that \( \coprod_I \Hom_k(k, X) \subsetneq \Hom_k( \prod_I k, X) \) whenever the index set \( I \) is infinite and \( X \) is non-zero.

For the second claim let \( X \in \catK( \Modcat k) \) and \( \mathbb{T} \) be as in the definition of \( 0 \)-cocompact. We assume that all complexes have zero differential. Let \( d \) be a degree such that \( X^d \neq 0 \). Pick an element \( ( \cdots, t_2, t_1, t_0 ) \in \limit T_i^d \). If \( t_n \neq 0 \) then we can find a linear map \( T_n^d \to X^d \) which does not send \( t_n \) to \( 0 \). This linear map will give rise to a non-zero element of \( \colim \Hom_k(T_i^d, X^d) \), and hence of \( \colim \Hom_{\catK(\Modcat k)}(T_i, X) \), contradicting the first assumption on the sequence. It follows that \(  \limit T_i^d = 0 \).

Now assume that the sequence \( \cdots \to T_2^{d-1} \to T_1^{d-1} \to T_0^{d-1} \) does not satisfy the Mittag-Leffler condition. That means there is a subsequence
\[ \cdots \xto{\varphi_3} T_{n_2}^{d-1} \xto{\varphi_2} T_{n_1}^{d-1} \xto{\varphi_1} T_{n_0}^{d-1} \]
such that \( T_{n_0}^{d-1} \supsetneq \Imm \varphi_1 \supsetneq \Imm \varphi_1\varphi_2 \supsetneq \cdots \). It follows that the maps in the sequence
\[ \Hom(T_{n_0}^{d-1}, X^d) \to \Hom(\Imm \varphi_1, X^d) \to \Hom(\Imm \varphi_1\varphi_2, X^d) \to \cdots \]
are all proper epimorphisms. Note that \( X^d = X[1]^{d-1} \), thus we have a contradiction to the assumption that \( \Hom(\mathbb{T}, X[1]) \) is dual ML. It follows that the sequence \( \cdots \to T_2^{d-1} \to T_1^{d-1} \to T_0^{d-1} \) does satisfy the Mittag-Leffler condition, and in particular \( \limit^1 T_i^{d-1} = 0 \).

Finally note that \( \holim \mathbb{T} = \limit \mathbb{T} \oplus \limit^1 \mathbb{T}[-1] \). Thus we have \( \Hom( \holim \mathbb{T}, X) = 0 \) if and only if for any \( d \) such that \( X^d \neq 0 \) we have \( \limit \mathbb{T}^d = 0 \) and \( \limit^1 \mathbb{T}^{d-1} = 0 \), which are exactly the two points established above.
\end{example}

\begin{theorem}\label{theorem cocompacts in derived category}
Let \( \catA \) be a Grothendieck abelian category which has exact products. Then the only cocompact object in \( \catD( \catA ) \) is \( 0 \).
\end{theorem}

\begin{remark}
The assumptions of the theorem are slightly stronger than what we need: For the proof, it suffices to assume that \( \catA \) is an abelian category with countable products and coproducts, and with enough injectives, which satisfies that the natural morphism from countable coproducts to products is monic , and that \( \catD( \catA ) \) is left complete in the sense of Neeman, see \cite{MR2875857}.
\end{remark}

In the proof we will utilize the following observation.

\begin{lemma} \label{lemma limits not cocompact}
Let  $C$ be a cocompact object in a triangulated category $\catT$. If we have $C = \holim C_i$ then \( C \) is a direct summand of a finite direct sum \( \oplus_{i=1}^n C_i \).
\end{lemma}

\begin{proof}
Consider the triangle \( C \to \prod_i C_i \to \prod_i C_i \to C[1] \) defining the homotopy limit. Note that any for any \( n \) we have the following short exact sequence.
\[ \begin{tikzcd}
0 \ar[r] \ar[d, >->] & \bigoplus_{i=1}^n C_i \ar[r]  \ar[d, >->] & \bigoplus_{i=1}^n C_i \ar[r]  \ar[d, >->] & 0  \ar[d, >->] \\
C \ar[r] \ar[d,->>] & \prod C_i \ar[r] \ar[d,->>] & \prod C_i \ar[r] \ar[d,->>] & C[1] \ar[d,->>]  \\
C \ar[r] & \prod_{i>n} C_i \ar[r] & \prod_{i>n} C_i \ar[r] & C[1]
\end{tikzcd} \]
The induced map in the top row is an isomorphism, whence the monomorphism splits coherently. It follows that the bottom row also is a direct summand of the middle row, and in particular is a triangle.

By cocompactness of \( C \), and hence also \( C[1] \), the map \( \prod_i C_i \to C[1] \) vanishes on all but finitely many factors. Thus by choosing \( n \) sufficiently large we can ensure that the rightmost map in the bottom row vanishes. It follows that \( C \) is a direct summand of \( \prod_{i>n} C_i \). Again using that \( C \) is cocompact it is in fact a direct summand of a finite subproduct.
\end{proof}

\begin{proof}[Proof of Theorem~\ref{theorem cocompacts in derived category}]
Note first that any object \( C \) is the homotopy limit of a sequence of bounded complexes (this property is called being ``left complete''): Since products in \( \catA \) are exact, products in \( \catD( \catA ) \) are calculated componentwise. Since the sequence of canonical left truncations is eventually constant in any given component, it follows that the homotopy limit of the canonical truncations is indeed \( C \).

If \( C \) is cocompact, then it follows by Lemma~\ref{lemma limits not cocompact} that \( C \) is a direct summand of a finite direct sum of its canonical left truncations, hence a left bounded complex. (Here we note that a complex \( X \) is left bounded if and only if \( \exists n \forall m > n \forall I \in \inj(\catA) \mid \catD( \catA )( X, I[m]) = 0 \), hence left bounded complexes are closed under direct summands in \( \catD( \catA ) \).)

Next we note that any left bounded complex is isomorphic in \( \catD( \catA ) \) to a left bounded complex of injectives. Such a complex is the homotopy limit of its brutal right truncations, which are finite complexes of injectives. Thus, envolking Lemma~\ref{lemma limits not cocompact} again, we see that any cocompact \( C \) is a direct summand of a finite direct sum of finite complexes of injectives, and thus itself a finite complex of injectives. (The finite complexes of injectives are characterized by the fact that \( \exists n \forall m > n \forall A \in \catA \mid \Hom(X, A[m]) = 0 \), hence this collection of complexes is closed under direct summands in \( \catD( \catA ) \).)

So we consider finite complexes of injectives. If the complex is not contractible, then we may assume that \( C \) is concentrated in non-positive degrees and that the map \( C^{-1} \to C^0 \) is not a split epimorphism.

We consider the map \( \Sigma \colon (C^0)^{(\mathbb N)} \to C^0 \), giving by identity on every component. Note that for any map \( \varphi \colon (C^0)^{(\mathbb N)} \to C^0 \) which factors through projection to a finite subcoproduct \( (C^0)^{(\mathbb N)} \to (C^0)^{ \{1, \ldots, n\} } \), the difference \( \Sigma - \varphi \) is still a split epimorphism.

Now, since \( C^0 \) is injective we may extend \( \Sigma \) to a map \( \hat \Sigma \colon (C^0)^{ \mathbb N } \to C^0 \). It follows immediately that also \( \hat \Sigma - \varphi \) is a split epimorphism for any \( \varphi \in \Hom_{\catA}(C^0, C^0)^{(\mathbb{N})} \). Consequently, the natural map
\[ \Hom_{\catD( \catA)}( C^0, C)^{(\mathbb N)} \to \Hom_{\catD(\catA)}((C^0)^{\mathbb{N}}, C) \]
does not hit \( \hat \Sigma \), hence is not an isomorphism.
\end{proof}

\begin{remark} \label{rem cocompacts in homotopy category}
For homotopy categories the situation is more subtle, and depends on our model of set theory. We will use methods we learned from \cite{MR1914985}---see this book also for a full account of the unexplained notation.

On the one hand, one can see that if there is a measurable cardinal, then no category \( \catK( \Modcat R ) \) can contain a non-zero cocompact object: If \( I \) is a set admitting a non-principal \( \omega_1 \)-ultrafilter, then an easy adaptation of \cite[Example~3.1]{MR1914985} gives an element of \( \Hom_{\catK(\Modcat R)}( X^I, X) \) which does not lie in the image of the  canonical map from \( \Hom_{\catK(\Modcat R)}(X, X)^{(I)} \) for any object \( X \).

On the other hand, if there is no measurable cardinal, and \( R \) is a slender ring (for instance \( R = \mathbb Z \), see \cite[Section~III.2]{MR1914985}), then the {\L}o\'s--Eda-theorem (see \cite[Corollary~3.3]{MR1914985}) implies that \( R \), considered as a complex concentrated in one degree, is cocompact.
\end{remark}

When it comes to closure properties, (co)compact objects are much better behaved than \( 0 \)-(co)compacts: The triangulated subcategory $\catT^{\compacts}$ is thick, and often easy to describe completely in concrete examples. The collection of $0$-cocompact objects in $\catT$, on the other hand, is typically much more difficult to control. For instance, a summand of a $0$-cocompact object need not be $0$-cocompact again.

We do however have the following closure property.

\begin{lemma} \label{lemma products}
Let \( X_i \) be a set-indexed collection of \( 0 \)-cocompact objects. Then \( \prod X_i \) is \( 0 \)-cocompact.
\end{lemma}

\begin{proof}
Since \( \Hom \) commutes with products in the second argument, it suffices to observe that a product of sequences has vanishing colimit only if each factor has vanishing colimit, and similarly is dual ML only if each factor is dual ML.
\end{proof}

\subsection*{Compact generation}
For each collection $\catS$ of objects in $\catT$ we consider \[\catS^\perp=\{T\in\catT \vert \catT(\catS,T[n])=0 \text{ for each } n\}.\] $\catS$ is said to \emph{generate} $\catT$ if $\catS^\perp=0$; if $\catT$ admits a generating set consisting of compact objects, then $\catT$ is \emph{compactly generated}. If $\catT$ is compactly generated by $\catS$, then $\catT$ coincides with its smallest triangulated subcategory which  contains $\catS$ and is closed under coproducts.

The following is Neeman's Brown representability theorem from \cite{MR1812507}.

\begin{theorem} \label{theorem brown representability}
  If $\catT$ is a compactly generated triangulated category, then $\catT$ satisfies Brown representability, that is, each cohomological functor $\catT^{\op}\to\Ab$ which takes coproducts of $\catT$ to products in $\Ab$, is isomorphic to $\catT(-,T)$ for some $T\in\catT$.
\end{theorem}

Some useful consequences of Theorem~\ref{theorem brown representability}, extracted from \cite{MR1308405, MR1812507}, are:

\begin{theorem} \label{thm:neemanthreeadjoints}
Suppose $F\colon \catT'\to\catT$ is an exact functor between triangulated categories with $\catT'$ compactly generated.
\begin{enumerate}%[label=\emph{(}\roman*\emph{)}]
\item $F$ admits a right adjoint if and only if $F$ preserves coproducts.
\item $F$ admits a left adjoint if and only if $F$ preserves products.
\item If $F$ admits a right adjoint $G$, then $F$ preserves compact objects if and only if $G$ preserves coproducts.
\end{enumerate}
\end{theorem}

Let us also recall a trick from \cite[Theorem~2.1]{MR1191736}.

\begin{theorem} \label{theorem Neeman trick}
 Let $\catT$ be a compactly generated triangulated category, and let $X\in\catT^{\compacts}$. Then the subcategory $X^\perp$ of $\catT$ is compactly generated again.

 Moreover, the left adjoint to the inclusion $X^\perp\hookrightarrow\catT$ induces an equivalence \[(X^\perp)^{\compacts}\simeq\catT^{\compacts}/\thick X\] up to direct summands.
\end{theorem}

\subsection*{0-cocompact cogeneration}
For a class of objects $\catS$ in $\catT$ we also consider \[{}^\perp\catS=\{T\in\catT \vert \catT(T[n],\catS)=0 \text{ for each } n\}.\] $\catS$ \emph{cogenerates} $\catT$ if ${}^\perp\catS=0$; if $\catT$ admits a cogenerating set which consists of 0-cocompact objects, then $\catT$ is \emph{0-cocompactly cogenerated}. By \cite[Theorem~6.6]{MR3946864}, if $\catT$ is $0$-cocompactly cogenerated by $\catS$, then $\catT$ coincides with its smallest triangulated subcategory which contains $\catS$ and is closed under products.

We end this section with some observations which are (weak) duals of results from the previous subsection. As one would expect, this story is far less complete than its classical counterpart.

For our dual version of Theorem~\ref{theorem brown representability} we refer to Appendix~\ref{section dual brown rep}.

As for Theorem~\ref{thm:neemanthreeadjoints}, we have the following partial `$0$-cocompact dual':

\begin{1weakdual} \label{theorem 1weakdual}
  Let $F\colon\catT'\to\catT$ be a triangle functor with a right adjoint $G$. If $F$ preserves countable products, then $G$ preserves $0$-cocompact objects.

  In particular, if $F$ additionally reflects $0$-objects, then $G$ takes any set of $0$-cocompact cogenerators for $\catT$ to a set of $0$-cocompact cogenerators for $\catT'$.
\end{1weakdual}

\begin{remark}
In contrast to the situation of Theorem~\ref{thm:neemanthreeadjoints}, here we do not get an ``if and only if'' statement. Indeed we have seen in Example~\ref{example K(vect)} that all objects in \( \catK(\Modcat k) \) are \( 0 \)-cocompact for a field \( k \). Thus any endofunctor of \( \catK(\Modcat k) \) preserves \( 0 \)-cocompacts, but clearly not every left adjoint endofunctor preserves countable products.
\end{remark}

\begin{proof} Let $X\in\catT$ be $0$-cocompact and consider a sequence \[\mathbb T \colon\cdots \to T_2\to T_1 \to T_0\] in $\catT'$ such that $\colim\catT'(\mathbb T, GX)=0$ and $\catT'(\mathbb T, GX[1])$ is dual ML. By adjunction we have that $\colim\catT(F\mathbb T, X)=0$ and that $\catT(F\mathbb T, X[1])$ is dual ML, hence $\catT(\holim F\mathbb T, X)=0$. If $F$ preserves countable products then it preserves homotopy limits, so in particular \[0=\catT(\holim F\mathbb T,X)=\catT(F\holim\mathbb T,X)\cong\catT'(\holim\mathbb T, GX)\] i.e.\ $GX$ is $0$-cocompact. The last claim follows immediately.
\end{proof}

Finally, our statement corresponding to Theorem~\ref{theorem Neeman trick} is

\begin{2weakdual} \label{theorem 2weakdual}
Suppose $\catT$ is 0-cocompactly cogenerated by $\catS$, and let $X\in\catS$. Then the subcategory ${}^\perp X$ of $\catT$ is $0$-cocompactly cogenerated again.
\end{2weakdual}
\begin{proof}
  The stable t-structure (\ref{align our t-structure}) --- see page~\pageref{align our t-structure} --- shows that the subcategory ${}^\perp X$ is an aisle, so by \cite{MR907948} the inclusion ${}^\perp X \hookrightarrow\catT$ admits a right adjoint $G$. We now observe that ${}^\perp X$ is closed under countable products: Take a countable subset $\{T_i\}$ of ${}^\perp X$. Then $\prod T_i = \holim \mathbb T$ for the obvious system
  \[\mathbb T\colon \cdots \to T_2\oplus T_1\oplus T_0 \to T_1\oplus T_0 \to T_0,\]
 and $\catT(\mathbb T,X[i])$ is the zero-sequence for each \( i \). In particular it is dual ML with vanishing colimit, so $\prod T_i\in{}^\perp X$ by the $0$-cocompactness of $X$. It follows from Theorem~\ref{thm:neemanthreeadjoints}${}^{\op}$ that the essential image $G(\catS)$ is a set of $0$-cocompact cogenerators for ${}^\perp X$.
\end{proof}

\section{Partial Serre functors}\label{section rel serre}

Recall that \( \catT \) is a \( k \)-category for some commutative ring \( k \). Let us choose an injective \( k \)-module \( I \), which cogenerates \( \Modcat k \). We write \( (-)^\ast = \Hom_k (-,I) \). Typical examples include $k$ being artinian and $I$ the injective envelope of $k/\!\rad k$, or $k$ being the ring of integers and $I=\mathbb Q/\mathbb Z$.

\begin{observation}
We collect a few central properties of the functor \( (-)^\ast \), all of which are immediate.
\begin{itemize}
\item \( (-)^\ast \) is contravariant and exact.
\item \( (-)^\ast \) reflects \( 0 \)-objects and isomorphisms.
\item There is a natural monomorphism \( \id \to (-)^{\ast\ast} \).
\end{itemize}
\end{observation}

\begin{definition}
A \emph{partial Serre functor} for a subcategory \( \catX\subset\catT \) is a functor \( \mathbb S \colon \catX \to \catT \) such that
\[ \catT(X,T)^\ast \cong \catT(T,\mathbb SX) \]
naturally in \( X \in \catX \) and in \( T \in \catT \).
\end{definition}

In the special case of a compactly generated triangulated category $\catT$, and $X= \catT^{\compacts}$, the existence of a partial Serre functor was observed by Rouquier in \cite[Corollary~4.23]{Rouquier} (see also our Corollary~3.6 below). In \cite[Example~5.12]{Ballard}, Ballard applied these functors to quasi-projective schemes, calling them ``Rouquier functors".

The functorial properties of partial Serre functors are interesting in their own right. In particular, in the case that $\catT$ is a triangulated category, a partial Serre functor is automatically a triangle functor. However, here we are mostly concerned with the connection between partial Serre functors and notions of (co)compactness. Therefore the proof of the following theorem, which summarizes the functorial properties, is postponed to Appendix~\ref{section exactness of serre}.

\begin{theorem} \label{theorem main general result on partial serre}
Suppose \( \catT \) is a triangulated category, and let \( \catX \) be the full subcategory consisting of all objects \( X \) such that \( \catT(X, -)^\ast \) is representable.

Then \( \catX \) is a triangulated subcategory of \( \catT \), and there is a partial Serre functor \( \mathbb S \colon \catX \to \catT \). Moreover, \( \mathbb S \) is a triangle functor.
\end{theorem}

\begin{observation}
Any partial Serre functor $\mathbb S\colon \catX\to\catT$ is faithful, since we have \[\catT(X,Y)\hookrightarrow\catT(X,Y)^{\ast \ast} \cong \catT(\mathbb SX, \mathbb SY)\] for each $X,Y\in\catX$ by applying the duality formula twice.

On the other hand, $\mathbb S$ is full only when the subcategory $\catX$ has `sufficiently small' $\Hom$-sets: If $k$ is a field, this amounts to $\catX$ being $\Hom$-finite; if $k=\mathbb Z$ and $I=\mathbb Q/\mathbb Z$, then $\mathbb S$ is full provided that $\catT(X,Y)$ is a finite abelian group for each $X,Y\in\catX$.
\end{observation}

Partial Serre duality links $0$-cocompact objects to compact objects:

\begin{theorem} \label{theorem partial Serre implies X compact and Y 0-cocompact}
  Let $\catT$ be a triangulated category. If $X,Y\in\catT$ satisfy \[\catT(X,-)^\ast \cong \catT(-,Y),\] then $X$ is compact and $Y$ is $0$-cocompact.
\end{theorem}
\begin{proof}
  Let us first show that $X$ is compact. Consider a set of objects $\{T_i\}\subset \catT$, and for each $i$ let $\mu_i \colon T_i \to \coprod T_i$ be the canonical morphism. By assumption we have
  \[ \begin{tikzcd}
  \catT(X,\coprod T_i)^\ast \ar[r,"\cong"]\ar[d,"{\catT(X,\mu_i)^\ast}"] & \catT(\coprod T_i, Y)  \ar[d,"\mu_i^\ast"] \\
  \catT(X, T_i)^\ast \ar[r,"\cong"]  & \catT(T_i, Y)
  \end{tikzcd} \]
  and taking products in the lower row gives another commutative diagram:
  \[ \begin{tikzcd}
  \catT(X,\coprod T_i)^\ast \ar[r,"\cong"]\ar[d] & \catT(\coprod T_i, Y)  \ar[d] \\
  \prod\catT(X, T_i)^\ast \ar[r,"\cong"]  & \prod\catT(T_i, Y)
  \end{tikzcd} \]
    Since the right hand vertical morphism is invertible, so is the left hand vertical one. Moreover, the left hand vertical morphism is the dual of the natural morphism \[\coprod \catT(X, T_i) \to \catT(X, \coprod T_i),\] so the claim follows since $(-)^\ast$ reflects isomorphisms.

    We now show that $Y$ is $0$-cocompact. Let \[\mathbb T\colon \cdots \to T_2 \to T_1 \to T_0\] be a sequence such that $\colim\catT(\mathbb T,Y)=0$ and $\catT(\mathbb T, Y[1])$ is dual ML. We want to conclude that $\holim \mathbb T$ is a contravariant $Y$-ghost. Equivalently, we can show that $\holim\mathbb T$ is a covariant $X$-ghost. The triangle \[\holim \mathbb T \to \prod T_i\to\prod T_i\to\holim\mathbb T[1]\] induces a long exact sequence
  \[\begin{tikzpicture}[descr/.style={fill=white,inner sep=1.5pt}]
        \matrix (m) [
            matrix of math nodes,
            row sep=1em,
            column sep=2.5em,
            text height=1.5ex, text depth=0.25ex
        ]
        { \cdots & \catT(X,\prod T_i[-1]) & \catT(X,\prod T_i[-1]) & \catT(X,\holim\mathbb T)\\
            & \catT(X,\prod T_i) & \catT(X,\prod T_i) & \cdots. \\
        };

        \path[overlay,->, font=\scriptsize]
        (m-1-1) edge (m-1-2)
        (m-1-2) edge (m-1-3)
        (m-1-3) edge (m-1-4)
        (m-1-4) edge[out=355,in=175] (m-2-2)
        (m-2-2) edge (m-2-3)
        (m-2-3) edge (m-2-4);
    \end{tikzpicture}\]
  In particular there is a short exact sequence \[0\to\limit^1\catT(X,\mathbb T[-1])\to\catT(X,\holim\mathbb T)\to\limit\catT(X,\mathbb T)\to0,\] and  it suffices to show that the outer terms vanish. Since the system $\catT(\mathbb T,Y[1])$ is the dual of the system $\catT(X,\mathbb T[-1])$, it follows that the latter is ML, so its derived limit vanishes. On the other hand, the vanishing of $\colim\catT(\mathbb T,Y)\cong\colim\left(\catT(X,\mathbb T)^\ast\right)$ implies the vanishing of $\colim\left(\catT(X,\mathbb T)^\ast\right)^\ast\cong \limit\left(\catT(X,\mathbb T)^{\ast\ast}\right)$. Moreover, since $\limit$ is left exact the monomorphism of diagrams $\catT(X,\mathbb T)\to\catT(X,\mathbb T)^{\ast\ast}$ induces a monomorphism $\limit\catT(X,\mathbb T)\to\limit\left(\catT(X,\mathbb T)^{\ast\ast}\right)$, whence $\limit\catT(X,\mathbb T)=0$.
\end{proof}

In particular, Theorem~\ref{theorem partial Serre implies X compact and Y 0-cocompact} says that if $\mathbb S\colon\catX\to\catT$ is a partial Serre functor, then $\catX\subset\catT^{\compacts}$, while the essential image $\mathbb S(\catX)$ consists of $0$-cocompact objects.

\begin{corollary} \label{corollary compact generation gives partial serre}
  Let $\catT$ be a compactly generated triangulated category. Then there is a partial Serre functor $\mathbb S\colon \catT^{\compacts}\to\catT$, and the essential image $\mathbb S(\catT^{\compacts})$ is a set of $0$-cocompact cogenerators for $\catT$.
\end{corollary}
\begin{proof}
  For each compact object $X$, the functor $\catT(X,-)^\ast$ is representable by Brown representability (Theorem~\ref{theorem brown representability}), so by Theorem~\ref{theorem main general result on partial serre} there is a partial Serre functor $\mathbb S \colon \catT^{\compacts}\to \catT$. The set $\mathbb S(\catT^{\compacts})$ consists of $0$-cocompact objects by Theorem~\ref{theorem partial Serre implies X compact and Y 0-cocompact}, hence the last claim follows from the fact that $(-)^\ast$ reflects the vanishing of $k$-modules.
\end{proof}

\begin{example} \label{example 0-cocompacts in D(Lambda)}
  Let $\Lambda$ be an Artin algebra. Then $\catD(\Modcat \Lambda)^{\compacts}=\perf\Lambda$, and \[\mathbb S=-\otimes_\Lambda^{\mathbb L}D\Lambda\colon\perf\Lambda\to\catD(\Modcat\Lambda)\] is a partial Serre functor (we will give a more general argument in Example~\ref{ex Serre for derived}) inducing an equivalence $\perf\Lambda\simeq\mathbb S(\perf\Lambda)=\{\text{bounded complexes over }\inj\Lambda\}$ of subcategories of $\catD(\Modcat\Lambda)$. Note that $\mathbb S$ is an autoequivalence on $\perf\Lambda$ if and only if $\Lambda$ is Gorenstein, a fact already observed in \cite{MR910167}, in which case each perfect complex is 0-cocompact in $\catD(\Modcat\Lambda)$.

%  We would like to point out that the derived category does not have \textit{any} non-zero cocompact objects. Indeed, it is not hard to see that even a field $k$ is not cocompact in $\catD(\Modcat k)$: The vector space $k^\mathbb Z$ has a basis of cardinality $\vert C\vert>\vert\mathbb Z\vert$, so \[\Hom_{\catD}(k^\mathbb Z, k) \cong \Hom_k(k^{(C)}, k) \cong k^C\] which is certainly not isomorphic to $\Hom_{\catD}(k,k)^{(\mathbb Z)} \cong k^{(\mathbb Z)}.$
\end{example}

Recall that a triangle $X\to Y\to Z\to X[1]$ is called \emph{pure} if the morphism $Z\to X[1]$ is a covariant $\catT^{\compacts}$-ghost. Equivalently, for each compact object $C$, the induced sequence $0\to\catT(C,X)\to\catT(C,Y)\to\catT(C,Z)\to0$ is exact. An object $E\in\catT$ is \emph{pure-injective} if $\catT(-,E)$ takes pure triangles to short exact sequences.

\begin{remark}
If $\mathbb S\colon\catX\to\catT$ is a partial Serre functor, then $\mathbb SX$ is pure-injective for each $X\in\catX$. Indeed, we only need to check that $\catT(-,\mathbb SX)\cong\catT(X,-)^\ast$ takes pure triangles to short exact sequences. As $X$ is compact, the functor $\catT(X,-)$ does enjoy this property by definition. Since $(-)^\ast$ is exact, the claim follows.
\end{remark}

We do not know if $0$-cocompactness in general implies pure-injectivity. On the other hand, as was pointed out to us by Angeleri H\"ugel, there are pure-injective objects which are not $0$-cocompact:

\begin{example}
  Since $\mathbb Q$ is injective as $\mathbb Z$-module, it is pure-injective in $\catD(\Modcat \mathbb Z)$. However, $\mathbb  Q$ is not $0$-cocompact: In the proof of Theorem~\ref{theorem Neeman trick}${}^{\op}$ we saw that if an object $X$ is $0$-cocompact, then the subcategory ${}^\perp X$ is closed under countable products, and it is not hard to realize that ${}^\perp \mathbb Q$ does not enjoy this property. Indeed, let $p$ be a prime, and consider the Prüfer $p$-group $P= \colim \mathbb Z/(p^i)$. Then $\Hom_{\mathbb Z}(P,\mathbb Q)=0$, since $\Hom_{\mathbb Z}(\mathbb Z/(p^i),\mathbb Q)=0$ for each $i$. However, the torsion submodule of $P^\mathbb N$ is a proper submodule, so $\Hom_{\mathbb Z}(P^\mathbb N, \mathbb Q)\ne0$.

   Similarly, for a field $k$, the pure-injective $k(X)$ is not $0$-cocompact in $\catD(\Modcat k[X])$.
\end{example}

\section{Partial Serre functors in homotopy categories} \label{section construction  of rel serre}

The aim of this section is to construct, in elementary terms, partial Serre functors for homotopy categories of module categories (Theorem~\ref{theorem partial serre for Kb(mod R)}), as well as for homotopy categories of injective or of projective modules (Theorem~\ref{theorem serre functors in K(Inj) and K(Proj)}).

The homotopy category of all modules is typically not compactly generated, and does not even satisfy Brown representability. Thus we cannot apply the general abstract existence result from Corollary~\ref{corollary compact generation gives partial serre}. Also for the homotopy categories of injectives or of projectives, the results in the current section apply beyond the cases where these categories are known to be compactly generated.

Even in cases where the abstract existence result does apply, having an explicit construction can be useful: For instance if we want to explicitly describe almost split triangles, rather than just claim their existence, we need to have an explicit description of the corresponding partial Serre functor first.

Before actually constructing anything, we record some facts from homological algebra that will be useful in the sequel.

\subsection*{Totalization}Let $R$ be a ring. By a double complex $X$ we mean a diagram
\[ \begin{tikzcd}
& [-1.6em]\ar[d,-,thick,dotted] & \ar[d,-,thick,dotted]&[-1.6em] \\[-1em]
\ar[r,-,thick,dotted]& X^{i,j}\ar[r,"r_X"]\ar[d,"c_X"]&[4em] X^{i+1,j} \ar[d,"c_X"]& \ar[l,-,thick,dotted] \\
\ar[r,-,thick,dotted]& X^{i,j+1}\ar[r,"r_X"]& X^{i+1,j+1} & \ar[l,-,thick,dotted] \\[-1em]
& \ar[u,-,thick,dotted] & \ar[u,-,thick,dotted]&
\end{tikzcd} \]
%\[\begin{tikzpicture}[scale=1]
%  \node (1cl) at (1,3.5) {\scriptsize$\vdots$};
%  \node (1cr) at (3,3.5) {\scriptsize$\vdots$};
%  \node (2l) at (-.5,2.5) {\scriptsize$\cdots$};
%  \node (2cl) at (1,2.5) {$X^{i,j}$};
%  \node (2cr) at (3,2.5) {$X^{i+1,j}$};
%  \node (2r) at (4.5,2.5) {\scriptsize$\cdots$};
%  \node (3l) at (-.5,1) {\scriptsize$\cdots$};
%  \node (3cl) at (1,1) {$X^{i,j+1}$};
%  \node (3cr) at (3,1) {$X^{i+1,j+1}$};
%  \node (3r) at (4.5,1) {\scriptsize$\cdots$};
%  \node (4cl) at (1,0) {\scriptsize$\vdots$};
%  \node (4cr) at (3,0) {\scriptsize$\vdots$};
%  \draw [->] (1cl) to (2cl);
%  \draw [->] (1cr) to (2cr);
%  \draw [->] (2l) to (2cl);
%  \draw [->] (2cr) to (2r);
%  \draw [->] (3l) to (3cl);
%  \draw [->] (3cr) to (3r);
%  \draw [->] (3cl) to (4cl);
%  \draw [->] (3cr) to (4cr);
%  \draw [->] (2cl) to node[above]{\scriptsize$r_X$}(2cr);
%  \draw [->] (3cl) to node[above]{\scriptsize$r_X$}(3cr);
%  \draw [->] (2cl) to node[right]{\scriptsize$c_X$}(3cl);
%  \draw [->] (2cr) to node[right]{\scriptsize$c_X$}(3cr);
%\end{tikzpicture}\]
in which the rows and columns are complexes of $R$-modules, and moreover each square commutes. In the following discussion, we denote by $\Tot^\amalg(X)$ the totalization of $X$ with respect to coproducts, while $\Tot^\Pi(X)$ is the totalization of $X$ with respect to products. We denote by $\partial_X$ the differential on either variant of the total complex. Recall that \[\partial_X\vert_{\text{\tiny$X^{i,j}$}}=r_X + (-1)^i c_X.\]

Let $f\colon X\to Y$ be a morphism of double complexes. Then we can form a double complex $\Cone^{\row}(f)$ by taking the row-wise mapping cones of $f$ and equipping the columns of this object with the obvious differentials. In explicit terms, if $X=(X^{i,j}, r_X, c_X)$ and $Y=(Y^{i,j}, r_Y, c_Y)$, then $\Cone^{\row}(f)$ is as follows.
\[ \begin{tikzcd}
&[-1.6em] \ar[d,-,thick,dotted] &[2em] \ar[d,-,thick,dotted]&[-1.6em] \\[-1em]
\ar[r,-,thick,dotted] & Y^{i,j}\oplus X^{i+1,j}\ar[r,"\left(\begin{pampmatrix}r_Y& f\\0& -r_X\end{pampmatrix}\right)"]\ar[d,"\left(\begin{pampmatrix}c_Y& 0\\0& c_X\end{pampmatrix}\right)"]& Y^{i+1,j}\oplus X^{i+2,j} \ar[d,"\left(\begin{pampmatrix}c_Y& 0\\0& c_X\end{pampmatrix}\right)"]& \ar[l,-,thick,dotted] \\[.5em]
\ar[r,-,thick,dotted]& Y^{i,j+1}\oplus X^{i+1,j+1} \ar[r,"\left(\begin{pampmatrix}r_Y& f\\0& -r_X\end{pampmatrix}\right)"]& Y^{i+1,j+1}\oplus X^{i+2,j+1} & \ar[l,-,thick,dotted] \\[-1em]
& \ar[u,-,thick,dotted] & \ar[u,-,thick,dotted]&
\end{tikzcd} \]
%\[\begin{tikzpicture}[scale=1]
%  \node (1cl) at (1,3.5) {\scriptsize$\vdots$};
%  \node (1cr) at (6,3.5) {\scriptsize$\vdots$};
%  \node (2l) at (-1.2,2.5) {\scriptsize$\cdots$};
%  \node (2cl) at (1,2.5) {$Y^{i,j}\oplus X^{i+1,j}$};
%  \node (2cr) at (6,2.5) {$Y^{i+1,j}\oplus X^{i+2,j}$};
%  \node (2r) at (8.4,2.5) {\scriptsize$\cdots$};
%  \node (3l) at (-1.2,1) {\scriptsize$\cdots$};
%  \node (3cl) at (1,1) {$Y^{i,j+1}\oplus X^{i+1,j+1}$};
%  \node (3cr) at (6,1) {$Y^{i+1,j+1}\oplus X^{i+2,j+1}$};
%  \node (3r) at (8.4,1) {\scriptsize$\cdots$};
%  \node (4cl) at (1,0) {\scriptsize$\vdots$};
%  \node (4cr) at (6,0) {\scriptsize$\vdots$};
%  \draw [->] (2l) to (2cl);
%  \draw [->] (2cl) to node[above]{$\left(\begin{smallmatrix}r_Y&f\\0&-r_X\end{smallmatrix}\right)$}(2cr);
%  \draw [->] (2cr) to (2r);
%  \draw [->] (3l) to (3cl);
%  \draw [->] (3cl) to node[above]{$\left(\begin{smallmatrix}r_Y&f\\0&-r_X\end{smallmatrix}\right)$}(3cr);
%  \draw [->] (3cr) to (3r);
%  \draw [->] (1cl) to (2cl);
%  \draw [->] (2cl) to node[right]{$\left(\begin{smallmatrix}c_Y&0\\0&c_X\end{smallmatrix}\right)$}(3cl);
%  \draw [->] (3cl) to (4cl);
%  \draw [->] (1cr) to (2cr);
%  \draw [->] (2cr) to node[right]{$\left(\begin{smallmatrix}c_Y&0\\0&c_X\end{smallmatrix}\right)$} (3cr);
%  \draw [->] (3cr) to (4cr);
%\end{tikzpicture}\]
\begin{lemma} \label{lemma totalization is associative}
  Let $f\colon X \to Y$ be a morphism of double complexes. Then \[\Cone\left(\Tot^\ast(X)\xto{\Tot^\ast(f)}(\Tot^\ast(Y))\right)=\Tot^\ast\left(\Cone^{\row}(f)\right)\] as complexes for $\ast\in\{\amalg,\Pi\}$.
\end{lemma}
\begin{proof}
  Recall that \[\Cone(\Tot^\Pi(f))^n = \left(\prod_{i+j=n} Y^{i,j} \right)\oplus \left(\prod_{i+j=n+1}X^{i,j}\right),\] and that the differentials $d_{\catC}$ of this complex are given by \[d_{\catC}\vert_{\text{\tiny$Y^{i,j}$}} = \partial_Y \text{ and } d_{\catC}\vert_{\text{\tiny$X^{i,j}$}}=f-\partial_X.\] On the other hand, $\Cone^{\row}(f)^{i,j}=Y^{i,j}\oplus X^{i+1,j}$, and the differentials $d_{\catC^{\row}}$ of this double complex are given as $d_{\catC^{\row}}\vert_{\text{\tiny$Y^{i,j}$}} = r_Y + c_Y $ and $\text{ and } d_{\catC^{\row}}\vert_{\text{\tiny$X^{i,j}$}} = f - r_X + c_X$. In particular, \[\Tot^\Pi(\Cone^{\row}(f))^n = \prod_{i+j=n}(Y^{i,j}\oplus X^{i+1,j}),\] so the two complexes of the proposition do coincide in each degree. Moreover, the differentials $\partial_{\catT}$ of the latter complex are given by \[\partial_{\catT}\vert_{\text{\tiny$Y^{i,j}$}}= r_Y + (-1)^i c_Y = \partial_Y\] and, since $X^{i,j}$ lives in degree $(i-1,j)$ of this complex, \[\partial_{\catT}\vert_{\text{\tiny$X^{i,j}$}} = f - r_X + (-1)^{i-1}c_X = f-\partial_X.\]

  Of course, the same can be said using $\coprod$ instead of $\prod$.
\end{proof}

If $f\colon X\to Y$ is a morphism of double complexes, then for each $n$ we have a chain map $f^{\bullet, n}\colon X^{\bullet, n}\to Y^{\bullet, n}$. Visually, $f^{\bullet, n}$ is the `horizontal layer' at height $n$ in the triple complex $f$. Notice that $f^{\bullet, n}$ is a quasi-isomorphism if and only if the $n$'th row of the double complex $\Cone^{\row}(f)$ is acyclic.

\begin{lemma} \label{lemma our acyclic assembly}
  Let $f\colon X\to Y$ be a morphism of double complexes.
  \begin{enumerate}
    \item If $X$ and $Y$ are left bounded and each $f^{\bullet, n}$ is a quasi-isomorphism, then the chain map \[\Tot^\amalg(f)\colon\Tot^\amalg(X)\to\Tot^\amalg(Y)\] is a quasi-isomorphism.
    \item If $X$ and $Y$ are right bounded and each $f^{\bullet, n}$ is a quasi-isomorphism, then the chain map \[\Tot^\Pi(f)\colon\Tot^\Pi(X)\to\Tot^\Pi(Y)\] is a quasi-isomorphism.
  \end{enumerate}
\end{lemma}
\begin{proof}
  Claim $(2)$ is dual to claim $(1)$, so it suffices to prove the latter.

  By Lemma~\ref{lemma totalization is associative}, it suffices to show that $\Tot^\amalg(\Cone^{\row}(f))$ is acyclic. But by assumption, $\Cone^{\row}(f)$ is a double complex in which each row is acyclic and each diagonal is bounded on the lower left. The claim now follows from the Acyclic Assembly Lemma, see e.g.\ \cite[Lemma 2.7.3 and subsequent Remark]{MR1269324}.
\end{proof}

When $R$ is a ring and $X,Y\in\catC(\Modcat R)$, we write $\HOM_R(X,Y)$ for the double complex having $\Hom_R(X^i,Y^j)$ in position $(i,j)$. In particular \[\Hom_R(X,Y)=\Tot^\Pi(\HOM_R(X,Y)),\] meaning $\Hom_{\catC}(X,Y)=\Z^0 \Hom_R(X,Y)$, and $\Hom_{\catK}(X,Y)=\HH^0 \Hom_R(X,Y)$.

If $Z\in\catC(\Modcat R^{\op})$, then $X\tilde\otimes_R Z$ denotes the double complex having $X^j \otimes_R Z^i$ in position $(i,j)$---note the choice of coordinates---and \[X\otimes_R Z = \Tot^\amalg(X\tilde\otimes_R Z).\] The usual $\otimes$--$\Hom$-adjunction extends to an isomorphism of double complexes
\begin{align} \label{align adjunction type iso for double cpxs}
(X\tilde\otimes_R Z)^\ast \cong \HOM_R(X,Z^\ast).
\end{align}

Let us write $(-)^\vee=\Hom_R(-,R)$. Note that this functor induces an equivalence \( (\proj R)^{\op} \to \proj R^{\op} \), and hence also \[ \catC(\proj R)^{\op} \to \catC(\proj R^{\op}) \text{ and }\catK(\proj R)^{\op} \to \catK(\proj R^{\op}).\]

\begin{lemma} \label{lemma hom alg fact}
  Let $R$ be a ring, let $M$ be a complex of $R$-modules, and let \( P \) be a complex of finitely generated projective $R$-modules. Then we have an isomorphism \[\HOM_R(P,M)\cong M\tilde\otimes_R P^\vee \] of double complexes,  which is natural in $P$ and in $M$.
\end{lemma}

\begin{proof}
It suffices to observe that for $M\in\Modcat R$ and $P\in\proj R$, there is a natural morphism $M\otimes_R P^\vee \to \Hom_R(P,M)$ given by $m\otimes \phi\longmapsto\left[p\mapsto m\cdot\phi(p)\right]$, which is invertible when $P=R$.
\end{proof}

\begin{remark}
If \( P \) (or \( M \)) is a bounded complex then this isomorphism clearly implies that
\[ \Hom_R(P, M) \cong M \otimes_R P^{\vee}. \]
However, in general these two complexes are different, since the left hand side is a product totalization while the right hand side is a coproduct totalization.
\end{remark}

\subsection*{The homotopy category of modules} \label{subsection rel serre duality in K(Mod)} Let $R$ be a ring. Akin to the classical Auslander--Reiten translation of a finitely presented module, we find for each bounded complex $M$ of finitely presented $R$-modules a complex $\mathbb{S} M$ as follows. Let \[P_1 \stackrel{p}\to P_0 \to M \to 0\] be a projective presentation in the category $\Cb(\modcat R)$. Explicitly, this means that $P_1$ and $P_0$ belong to $\Cb(\proj R)$, and are moreover contractible. Define
\[ \mathbb{S} M = \Ker(\nu(p))[2],\]
where $\nu=(-^\vee)^\ast$.

The fact that this \( \mathbb{S} \) defines a partial Serre functor \( \Kb(\modcat R) \to \catK(\Modcat R) \) is a relatively straightforward extension of the familiar results from $\modcat R$, but we give a thorough account for the convenience of the reader.

We first remark that the required projective presentation exists:

\begin{construction} \label{construction proj pres by contractibles}
For a bounded complex $M$ over $\modcat R$, say \[M = \left(M^0 \to M^1\to M^2 \to \cdots \to M^n\right),\] pick epimorphisms $P_0^i\to M^i$ with $P_0^i\in\proj R$. Form the commutative diagram
\[ \begin{tikzcd}
P_0^0 \ar[r] \ar[d,->>]& P_0^0\oplus P_0^1 \ar[r] \ar[d,->>] & P_0^1\oplus P_0^2 \ar[r]\ar[d,->>] & \cdots \ar[r] &  P_0^{n-1}\oplus P_0^n \ar[r]\ar[d,->>]& P^n_0\ar[d,->>]\\
M^0 \ar[r] & M^1 \ar[r]  & M^2 \ar[r] & \cdots \ar[r] &  M^n \ar[r]& 0
\end{tikzcd} \]
where the top row is equipped with the obvious differentials making it a contractible complex. Now repeat, replacing $M$ by the induced complex of kernels, to obtain a contractible complex $P_1$, and hence a projective presentation of $M$ in $\Cb(\modcat R)$.
\end{construction}

Now we work towards a version of Auslander's defect formula.

\begin{lemma} \label{lemma.prep_for_defect}
Let \( R \) be a ring, \( P \) a finite contractible complex of finitely generated projectives. Then
\[ \Hom_{\catC(\Modcat R)}(P, -)^\ast \cong \Hom_{\catC(\Modcat R)}(-, \nu P[1]) \]
functorially in \( P \).
\end{lemma}

\begin{proof}
As double complexes, we know that
\begin{align*}
   \HOM_R(P, -)^\ast & \cong (- \tilde\otimes_R P^{\vee})^\ast && \text{Lemma~\ref{lemma hom alg fact}} \\
   & \cong \HOM_R(-, (P^{\vee})^\ast) && \text{(\ref{align adjunction type iso for double cpxs})} \\
   & = \HOM_R(-, \nu P).
\end{align*}
%\[ \HOM_R(P, -)^\ast \underset{\text{Lemma~\ref{lemma hom alg fact}}}{\cong} (- \tilde\otimes_R P^{\vee})^\ast \underset{\text{(\ref{align adjunction type iso for double cpxs})}}{\cong} \HOM_R(-, \underbrace{(P^{\vee})^\ast}_{= \nu P}). \]
Since \( P \) is a finite complex, all these double complexes have finite diagonals. So in particular \( \Tot^{\amalg} \) and \( \Tot^{\Pi} \) coincide and commute with dualizing. Therefore, as complexes we have
\[ \Hom_R(P, -)^{\ast} \cong \Hom_R(-, \nu P). \]
Note that since \( P \) is contractible these two complexes are exact. It follows that
\begin{align*}
 \Hom_{\catC(\Modcat R)}(P, -)^\ast &= (Z^0(\Hom_R(P, -)))^{\ast} \\
 & \cong Z^1(\Hom_R(P, -)^{\ast}) \\
& \cong Z^1(\Hom_R(-, \nu P)) \\ 
& = Z^0(\Hom_R(-, \nu P[1])) = \Hom_{\catC(\Modcat R)}(-, \nu P[1]). \qedhere
\end{align*}
\end{proof}

Let $\mathbb E \colon 0\to A\to B\to C\to0$ be an extension in $\catC(\Modcat R)$. The \emph{defects} $\mathbb E_{\defect}$ and $\mathbb E^{\defect}$ are defined, respectively, by exactness of the sequences
\begin{align*}
& 0\to\Hom_{\catC}(C,-)\to\Hom_{\catC}(B,-)\to\Hom_{\catC}(A,-)\to\mathbb E_{\defect}\to0; \\
& 0\to\Hom_{\catC}(-,A)\to\Hom_{\catC}(-,B)\to\Hom_{\catC}(-,C)\to\mathbb E^{\defect} \to0.
\end{align*}

\begin{proposition}[Auslander's defect formula] \label{proposition defect formula}
Let $R$ be a ring and take an extension $\mathbb E \colon 0\to A\to B\to C\to0$ in $\catC(\Modcat R)$. For each $M\in\Cb(\modcat R)$ there is an isomorphism  \[\mathbb E^{\defect}(M)^\ast\cong \mathbb E_{\defect} (\mathbb{S} M [-1])\] which is natural in $M$.
\end{proposition}
\begin{proof}
Let $P_1 \stackrel{p}\to P_0\to M\to0$ be a projective presentation. Note that the \( P_i \) satisfy the assumptions of Lemma~\ref{lemma.prep_for_defect} above. The exact sequence
\[ 0 \to \Hom_{\catC(\Modcat R)}(M, -) \to \Hom_{\catC(\Modcat R)}(P_0, -) \to \Hom_{\catC(\Modcat R)}(P_1, -) \]
dualizes to the upper row in the following diagram of functors on $\catC(\Modcat R)$ with exact rows, where we write \( (-,-) \) for \( \Hom_{\catC(\Modcat R)}(-, -) \).

\[ \begin{tikzcd}
& (P_1, -)^\ast \ar[r] \ar[d,"\cong"]& (P_0, -)^\ast \ar[r,->>] \ar[d,"\cong"]&  (M, -)^\ast \\
(-, \Ker \nu(p)[1]) \ar[r,>->]& (-,\nu P_1[1])\ar[r]  & (-, \nu P_0[1])
 \end{tikzcd} \]
The vertical isomorphisms are precisely the ones provided by Lemma~\ref{lemma.prep_for_defect}. Note that \( \Ker \nu(p)[1] = \mathbb{S}M[-1] \) by definition of \( \mathbb{S} \).

From $\mathbb E$ we thus get the following commutative diagram with exact rows and columns.
\[ \begin{tikzcd}
(C, \mathbb SM[-1]) \ar[r,>->] \ar[d,>->] & (B,\mathbb S M[-1]) \ar[r] \ar[d,>->] & (A,\mathbb SM[-1]) \ar[d,>->] \\
(P_1, C)^\ast \ar[r,>->] \ar[d] & (P_1,B)^\ast \ar[r,->>] \ar[d] & (P_1,A)^\ast  \ar[d] \\
(P_0, C)^\ast \ar[r,>->] \ar[d,->>] & (P_0,B)^\ast \ar[r,->>] \ar[d,->>] & (P_0,A)^\ast \ar[d,->>] \\
(M, C)^\ast \ar[r]  & (M,B)^\ast \ar[r,->>] & (M,A)^\ast
\end{tikzcd} \]
The Snake Lemma yields that the cokernel in the first row, which is \( \mathbb{E}_{\defect}(\mathbb{S}M[-1]) \), coincides with the kernel in the last row, which is \( \mathbb{E}^{\defect}(M)^\ast \).
\end{proof}

\begin{theorem}  \label{theorem partial serre for Kb(mod R)}
Let \( R \) be a ring. Then \( \mathbb{S} \) defines a partial Serre functor
\[ \mathbb{S} \colon \Kb(\modcat R) \to \catK(\Modcat R) \colon M \longmapsto \Ker(\nu(p))[2]. \]
\end{theorem}

\begin{proof}
Let \( M \in \Kb(\modcat R) \) and \( X \in \catK(\Modcat R) \). We need to show that
\[ \Hom_{\catK}(M, X)^\ast \cong \Hom_{\catK}(X, \mathbb{S}M) \]
naturally in \( M \) and \( X \).

Let \( C_X =\coCone(\id_X) \). Then we have the short exact sequence
\[ \mathbb{E} \colon  0 \to X[-1] \to C_X \to X \to 0 \]
in \( \catC(\Modcat R) \). Moreover, a map to \( X \) is null-homotopic if and only if it factors through \( C_X \). In other words,
\[ \Hom_{\catK}(-, X) = \mathbb{E}^{\defect}. \]
Similarly \( \Hom_{\catK}(X[-1], -) = \mathbb{E}_{\defect} \). We know from Proposition~\ref{proposition defect formula} that
\begin{align*} \Hom_{\catK}(M, X)^\ast  &= \mathbb E^{\defect}(M)^\ast \\ & \cong \mathbb E_{\defect} (\mathbb{S} M [-1]) =  \Hom_{\catK}(X[-1], \mathbb{S} M[-1])  = \Hom_{\catK}(X, \mathbb{S} M). \qedhere
\end{align*}
\end{proof}

\begin{observation} \label{observation 0-cocompacts in Kb(Lambda)}
  Let $\Lambda$ be an Artin algebra.
    We can now offer an arguably more conceptual explanation than the one given in \cite[Corollary~6.10]{MR3946864} of the fact that bounded complexes of finitely generated modules are $0$-cocompact objects in the homotopy category $\catK(\Modcat \Lambda)$: We choose \( I \) to be an injective envelope of \( k /\!\rad k \).

    In this setup $\mathbb{S}$ admits a quasi-inverse \( \mathbb{S}^- \). Explicitly, for each $M$ choose an injective copresentation $0\to M \to I^0 \to I^1$ in $\Cb(\modcat \Lambda)$, and let \[\mathbb{S}^-(M)= \Cok\left(\nu^- I^0 \to \nu^-I^1\right)[-2].\] Here we use the fact that $\nu^- = \Hom_{\Lambda^{\op}}((-)^\ast,\Lambda)$ is a quasi-inverse of \( \nu \) as a functor from finitely generated projective to finitely generated injective modules. It is immediate from the construction that \( \mathbb{S} \) and \( \mathbb{S}^- \) are mutually quasi-inverse.

    In particular, the partial Serre functor $\mathbb S$ is an auto-equivalence on the subcategory $\Kb(\modcat\Lambda)$, which thus consists of $0$-cocompact objects by Theorem~\ref{theorem partial Serre implies X compact and Y 0-cocompact}.
  \end{observation}

For explicit calculation, it is sometimes convenient to not need projective objects in the category of complexes, but rather just complexes of projectives. The following gives an alternative way of calculating $\mathbb S M$ using these.

\begin{proposition} \label{proposition AR translate without contractibles}
If $Q\stackrel{q}\to P \to M \to 0$ is an exact sequence in $\Cb(\modcat R)$ with $Q$ and $P$ consisting of projective modules, then we have \[\mathbb{S} M = \Tot\left(\Ker(\nu(q)) \hookrightarrow\nu Q \xto{\nu(q)} \nu P \right) \]
in \( \catK(\Modcat R) \). Here \( \nu P \) is the \( 0 \)-th column of the double complex.
\end{proposition}

\begin{proof} Observe that if \( P \) and \( Q \) happen to be projective in the category of complexes, then the claimed formula is just a restatement of the definition of \( \mathbb{S} \). Indeed, in that case \( \nu P \) and \( \nu Q \) are contractible, and thus the total complex is isomorphic to \( \Ker \nu(q)[2] \). The proof now consists of two independent steps, showing respectively that we may replace \( P \) and \( Q \) by projectives in the category of complexes, without changing the result of the formula.

\emph{Step 1:} Let $Q\stackrel{q}\to P \to M \to 0$ be an exact sequence of complexes with $Q$ and $P$ consisting of projectives.

Choose $\overline P = \coCone(\id_P)$. Then $\overline P$ is projective in the category of complexes, and appears in a canonical degree-wise split exact sequence \[0 \to P[-1] \stackrel{f}\to \overline P \stackrel{g}\to P \to 0.\] Since $g$ is an epimorphism, the pullback of $g$ and $q$ is even bicartesian. In particular, the middle row of the following diagram is also exact.
\[ \begin{tikzcd}
Q \ar[r,"q"] & P \ar[r,->>]  & M \ar[d,equal] \\
\tilde Q \ar[r,"\tilde q"] \ar[u] & \overline P \ar[r,->>] \ar[u,swap,"g"] & M  \\
P[-1] \ar[r,equal] \ar[u,swap,"h"] & P[-1] \ar[u,swap,"f"]
\end{tikzcd} \]
Notice that $\tilde Q$ is again a complex of projectives; in fact, $\tilde Q = \coCone(q)$. Moreover, since the middle column is degree-wise split, so is the leftmost column. Indeed, a splitting of $h$ is obtained by composing $\tilde q$ with a splitting of $f$. It follows that application of $\nu$ and then totalization, gives
\begin{align*} \label{align first step}
  \Tot\left(\Ker(\nu(q)) \hookrightarrow \nu Q \xto{\nu(q)} \nu P\right) = \Tot\left(\Ker(\nu(\tilde q)) \hookrightarrow \nu \tilde Q \xto{\nu(\tilde q)} \nu \overline P\right) \tag{$\ast$}
\end{align*}
up to a contractible summand.

\emph{Step 2:} Let $Q\stackrel{q}\to P \to M \to 0$ be an exact sequence of complexes with $Q$ and $P$ consisting of projectives.

Choose $\overline Q = \coCone(\id_Q)$. Then $\overline Q$ is projective in the category of complexes, and appears in a canonical degree-wise split exact sequence \[0 \to Q[-1] \stackrel{f}\to \overline Q \stackrel{g}\to Q \to 0.\] Form the commutative diagram
\[ \begin{tikzcd}
Q \ar[r,"q"] & P \ar[r,->>]  & M \ar[d,equal] \\
\overline Q \ar[r,"p"] \ar[u] & P \ar[r,->>] \ar[u,equal] & M  \\
Q[-1] \ar[u,swap,"f"]
\end{tikzcd} \]
and note that also the bottom row is exact: Since $g$ is an epimorphism, $p$ and $q$ have the same image. Application of $\nu$ yields the commutative diagram below; the fact that the bottom left corner is $\nu Q[-1]$, follows from the Snake Lemma.
\[ \begin{tikzcd}
\Ker(\nu(q)) \ar[r,>->] & \nu Q \ar[r,"\nu(q)"]  & \nu P \ar[d,equal] \\
\Ker(\nu(p)) \ar[r,>->] \ar[u] & \nu \overline Q \ar[r,"\nu(p)"] \ar[u] & \nu P  \\
\nu Q[-1] \ar[r,equal] \ar[u,swap,"h"] & \nu Q[-1] \ar[u,swap,"\nu(f)"]
\end{tikzcd} \]
Since the middle column is degree-wise split, so is the leftmost one: A splitting of $h$ is given by composing $\Ker(\nu(p))\hookrightarrow\nu \overline Q$ with a splitting of $\nu(f)$. It follows that, up to contractible summands,
\begin{align*} \label{align second step}
\Tot\left(\Ker(\nu(q))\hookrightarrow\nu Q \xto{\nu(q)} \nu P\right) = \Tot\left(\Ker(\nu(p))\hookrightarrow \nu \overline Q\xto{\nu(p)} \nu P\right). \tag{$\ast\ast$}
\end{align*}

%\begin{align*} \label{align second step}
%  \renewcommand{\arraystretch}{1.2}% Spread out array rows
%  \begin{array}{r  l}
%      \Tot\left(\Ker\hookrightarrow\nu Q \to \nu P\right) & = \Tot\left(\Ker\hookrightarrow \nu\coCone(\id_Q)\to \nu P\right) \\
%      & \cong \Ker\left(\nu\coCone(\id_Q)\to \nu P\right) \tag{$\ast\ast$} \\
%      & = \tau M.
%\end{array}
%\end{align*}

Now to complete the proof of the proposition, suppose $Q\stackrel{p}\to P \to M \to 0$ is an exact sequence in $\Cb(\modcat R)$ and that $Q$ and $P$ consists of projective modules. Successive application of the above two steps reveals a commutative diagram
\[ \begin{tikzcd}
Q \ar[r,"q"] & P \ar[r,->>]  & M \ar[d,equal] \\
\tilde Q \ar[r,"\tilde q"] \ar[u] & \overline P \ar[r,->>] \ar[u] \ar[d,equal] & M \ar[d,equal] \\
\overline Q \ar[r,"p"] \ar[u] & \overline P \ar[r,->>] & M
\end{tikzcd} \]
%\[\begin{tikzpicture}[scale=.6]
%\node (1lll) at (-3,2) {$Q$};
%\node (1ll) at (0,2) {$P$};
%\node (1l) at (3,2) {$M$};
%\node (1c) at (5,2) {$0$};
%\node (2lll) at (-3,0) {$\tilde Q$};
%\node (2ll) at (0,0) {$\overline P$};
%\node (2l) at (3,0) {$M$};
%\node (2c) at (5,0) {$0$};
%\node (3lll) at (-3,-2) {$\overline Q$};
%\node (3ll) at (0,-2) {$\overline P$};
%\node (3l) at (3,-2) {$M$};
%\node (3c) at (5,-2) {$0$};
%\draw [->] (1lll) to node[above] {\tiny$q$}(1ll);
%\draw [->] (1ll) to (1l);
%\draw [->] (1l) to (1c);
%\draw [->] (2lll) to node[above] {\tiny$\tilde q$}(2ll);
%\draw [->] (2ll) to (2l);
%\draw [->] (2l) to (2c);
%\draw [->] (3lll) to node[above] {\tiny$p$}(3ll);
%\draw [->] (3ll) to (3l);
%\draw [->] (3l) to (3c);
%\draw [->] (2lll) to (1lll);
%\draw [->] (2ll) to (1ll);
%\draw [thick,double distance=2pt,-] (2l) to (1l);
%\draw [->] (3lll) to (2lll);
%\draw [thick,double distance=2pt,-] (3ll) to (2ll);
%\draw [thick,double distance=2pt,-] (3l) to (2l);
%\end{tikzpicture}\]
with exact rows, where $\overline Q$ and $\overline P$ are projective objects in the category of complexes. Then (\ref{align first step}) and (\ref{align second step}) give us, up to contractible summands,
\begin{align*}
 \Tot\left(\Ker(\nu(q))\hookrightarrow \nu Q \xto{\nu(q)}\nu P\right)  & =  \Tot\left(\Ker(\nu(\tilde q))\hookrightarrow \nu \tilde Q \xto{\nu(\tilde q)}\nu \overline P\right)  \\
 & =  \Tot\left(\Ker(\nu(p))\hookrightarrow \nu \overline Q \xto{\nu(p)}\nu \overline P \right) \\
 & = \Ker(\nu (p))[2] \\ & = \mathbb S M. \qedhere
\end{align*}
\end{proof}

\begin{example}
If $M$ is an $R$-module, then we can simply use a projective presentation in $\modcat R$ in order to calculate $\mathbb{S} M$. That is, an exact sequence $P_1\stackrel{p}\to P_0 \to M$ of modules with $P_0, P_1 \in\proj R$, resulting in \[\mathbb{S} M =  \Tot\left(\tau M \hookrightarrow \nu P_1 \to \nu P_0\right)=\left(\tau M \hookrightarrow \nu P_1 \to \nu P_0\right),\] where $\tau$ denotes the `usual' Aulander--Reiten translation on the category of finitely presented modules.

Note that \( \nu \) is right exact, so in this case \( \mathbb{S} M \) is quasi-isomorphic to \( \nu M \).

In particular, if \( M \) is a finitely generated projective module, then our formula degenerates and \( \mathbb{S} M = \nu M \).
\end{example}

\subsection*{Homotopy categories of injectives and of projectives} Let $R$ be a ring. It follows from \cite{MR2439608} that the inclusion $\catK(\Proj R)\hookrightarrow\catK(\Modcat R)$ admits a right adjoint $\rho$. Note that if $X$ is right bounded, then $\rho X$ is a projective resolution of $X$. Indeed, let $Q\in\catK(\Proj R)$ and suppose $P$ is a right bounded complex which is quasi-isomorphic to $X$. Denoting by $Q^{\leq n} $ the brutal truncation of $Q$, which is trivially a quotient of $Q$, we get 
\[\Hom_{\catK}(Q,X) =\Hom_{\catK}(Q^{\leq n},X) \cong\Hom_{\catK}(Q^{\leq n},P)=\Hom_{\catK}(Q,P) \] for sufficiently large $n$, where the isomorphism in the middle holds because $Q^{\leq n}$ is homotopically projective.

On the other hand, by \cite{MR3212862}, the inclusion $\catK(\Inj R)\hookrightarrow\catK(\Modcat R)$ admits a left adjoint $\lambda$. If $X$ is left bounded, then $\lambda X$ is an injective resolution of $X$.

\begin{theorem} \label{theorem serre functors in K(Inj) and K(Proj)}
Let \( R \) be a ring and let \( \catK^{\bounded}_{\mathsf{fpr}}(R) \) denote the subcategory of \( \Kb(\Modcat R) \) consisting of complexes admitting degree-wise finitely generated projective resolutions.
\begin{enumerate}
\item There is a partial Serre functor \( \mathbb{S} \colon \lambda( \catK^{\bounded}_{\mathsf{fpr}}(R) ) \to \catK(\Inj R) \) given by \[ \mathbb{S}(\lambda X) = \nu \rho X.\]
\item There is a partial Serre functor \( \mathbb{S} \colon \rho(\catK^{\bounded}_{\mathsf{fpr}}(R^{\op}))^{\vee} \to \catK(\Proj R) \) given by \[ \mathbb{S}((\rho X)^{\vee}) = \rho(X^\ast). \]
\end{enumerate}
\end{theorem}

\begin{remark}
\begin{enumerate}
\item In particular it follows that \( \lambda( \catK^{\bounded}_{\mathsf{fpr}}(R) ) \) is a set of compact objects in \( \catK(\Inj R) \). Note that if \( R \) is right coherent then \(  \lambda(\catK^{\bounded}_{\mathsf{fpr}}(R)) \) is nothing but the bounded derived category of finitely presented modules, realized inside \( \catK(\Inj R) \) via injective resolutions. If \( R \) is even noetherian then it is shown in \cite{MR2157133} that these are in fact all the compact objects, and that \( \catK(\Inj R) \) is compactly generated.
\item Similarly, the set \(  \rho(\catK^{\bounded}_{\mathsf{fpr}}(R^{\op}))^{\vee} \) consists of compact objects in \( \catK( \Proj R ) \). In fact, by \cite[Proposition~7.12]{MR2439608} these are precisely all compact objects. If \( R \) is left coherent then  \(  \rho(\catK^{\bounded}_{\mathsf{fpr}}(R^{\op}))^{\vee} \)  is equivalent to the opposite of the bounded derived category of finitely presented left \( R \)-modules and by \cite[Proposition~7.14]{MR2439608} the category \( \catK(\Proj R ) \) is compactly generated.
\end{enumerate}
\end{remark}

\begin{proof}
(1): Let $I\in\catK(\Inj R)$, and let $X\in  \catK^{\bounded}_{\mathsf{fpr}}(R)$. Pick a projective resolution with finitely generated terms $\rho X$ of $X$, and let $\rho X \to X$ be a quasi-isomorphism. In particular, $\rho X$ is right bounded. Hence, by Lemma~\ref{lemma our acyclic assembly}, the induced morphism $\HOM_R(X,I) \to \HOM_R(\rho X, I)$ of left bounded double complexes totalizes to a quasi-isomorphism \[\Tot^\amalg(\HOM_R(X,I)) \to \Tot^\amalg(\HOM_R(\rho X, I)).\] Since $X$ is a bounded complex, this means in particular that
  \begin{align*}
    \Hom_{\catK}(X,I) & = \HH^0\Tot^\Pi(\HOM_R(X, I)) \\
    & = \HH^0\Tot^\amalg(\HOM_R(X, I)) \\
    & \cong \HH^0\Tot^\amalg(\HOM_R(\rho X, I)).
  \end{align*}
  Moreover, since $\rho X$ consists of finitely generated projective $R$-modules, Lemma~\ref{lemma hom alg fact} gives an isomorphism of double complexes\[\HOM_R(\rho X,I)\cong I\tilde\otimes_R(\rho X)^\vee.\]
With these observations we can verify the claim of the theorem as follows.
  \begin{align*}
  \Hom_{\catK}(\lambda X, I)^\ast & \cong \Hom_{\catK}(X, I)^\ast && \text{\( \lambda \) is left adjoint} \\
  & \cong \HH^0 \Tot^\amalg(\HOM_R(\rho X,I))^\ast \\
  & \cong \HH^0 \Tot^\amalg (I\tilde\otimes_R (\rho X)^\vee)^\ast \\
  & \cong \HH^0 \Tot^\Pi \left((I\tilde\otimes_R (\rho X)^\vee)^\ast\right) && \text{dual of \( \amalg \) is \( \Pi \)} \\
  & \cong \HH^0 \Tot^\Pi \HOM_R(I, ((\rho X)^\vee)^\ast) && \text{by (\ref{align adjunction type iso for double cpxs})} \\
  & = \Hom_{\catK}(I, \nu\rho X).
  \end{align*}

(2): Let $P\in\catK(\Proj R)$, and let $Y\in\catK^{\bounded}_{\mathsf{fpr}}(R^{\op})$. Pick a projective resolution with finitely generated terms $\rho Y$ of $Y$, and let $\rho Y \to Y$ be a quasi-isomorphism. Since $Y$ and $\rho Y$ are right bounded, so are $P\tilde\otimes_R Y $ and $P \tilde\otimes_R \rho Y$. In particular, by Lemma~\ref{lemma our acyclic assembly} the induced morphism $P\tilde\otimes_R\rho Y \to P \tilde\otimes_R Y$ of double complexes totalizes to a quasi-isomorphism \[\Tot^\Pi(P\tilde\otimes_R\rho Y)\to\Tot^\Pi(P\tilde\otimes_R Y).\] Moreover, since $\rho Y$ consists of finitely generated $R^{\op}$-modules, we have \[\HOM_R((\rho Y)^\vee, P) \cong P \tilde\otimes_R \rho Y\] by Lemma~\ref{lemma hom alg fact}. Combining these observations yields
\begin{align*}
  \Hom_{\catK}((\rho Y)^\vee, P) & = \HH^0\Tot^\Pi(\HOM_R(\rho Y^\vee, P)) \\
  & \cong \HH^0\Tot^\Pi(P \tilde\otimes_R \rho Y) \\
  & \cong \HH^0\Tot^\Pi(P \tilde\otimes_R Y) \\
  & \cong \HH^0\Tot^\amalg(P \tilde\otimes_R Y),
\end{align*}
where the last equality holds since $Y$ is a bounded complex. The claim now follows by the following calculation.
\begin{align*}
(\HH^0\Tot^\amalg(P \tilde\otimes_R Y))^\ast & \cong \HH^0\Tot^\Pi\left((P \tilde\otimes_R Y)^\ast \right) && \text{dual of \( \amalg \) is \( \Pi \)}\\
& \cong \HH^0\Tot^\Pi\HOM_R(P, Y^\ast) && \text{by (\ref{align adjunction type iso for double cpxs})} \\
& = \Hom_{\catK}(P, Y^\ast) \\
& \cong \Hom_{\catK}(P, \rho (Y^\ast)) && \text{\( \rho \) is right adjoint} \qedhere
\end{align*}
\end{proof}

\section{Transferring partial Serre functors to subcategories} \label{section transferring serre}
Let $R$ be a noetherian ring. Recall that an $R$-module is \emph{Gorenstein projective} if it appears as a boundary of a totally acyclic complex over $\Proj R$. Such modules form the full subcategory $\GProj R$ of $\Modcat R$. We write $\Gproj R = \GProj R \cap \modcat R$.

The categories $\GProj R$ and $\Gproj R$ are Frobenius exact, so the stabilizations $\GStable R$ and $\Gstable R$ are triangulated. The \emph{singularity category} of $R$ is the Verdier quotient $\Dsg(R)=\Db(\modcat R)/\perf R$, and there is an exact embedding \[\Gstable R\hookrightarrow\Dsg(R).\]

If $R$ is Gorenstein, then each finitely generated $R$-module admits a `Gorenstein projective approximation'. This means that the inclusion $\Gstable R\hookrightarrow\stablemodules R$ admits a right adjoint functor $\GP$---the reader is referred to \cite{buchweitz1986maximal,MR1044344}, see also \cite[Proposition~2.5]{BIKP}.

The following is \cite[Theorem~2.9]{BIKP}.

\begin{theorem} \label{theorem BIKP}
  Let $A$ be a Gorenstein algebra which is finite dimensional over a field. There are (classical) Serre functors \[ \Gstable A\xto{[-1]\circ\GP\circ\nu} \Gstable A \text{\,\,  and \,\,\,} \Dsg(A)\xto{[-1]\circ\overline{\mathbb{L} \nu}}\Dsg(A).\]

 In \( \nu \) here we dualize with respect to the base field. Note that the derived functor \( \mathbb L \nu \) induces a well-defined functor on singularity categories, denoted by \( \overline{ \mathbb L \nu} \) above.
\end{theorem}

The goal of the current subsection is to extend Theorem~\ref{theorem BIKP}, by relaxing the size condition, lifting the homological restriction of Gorensteinness, and providing partial Serre functors inside larger ambient categories. This will be achieved in Theorem~\ref{theorem partial serre functor on Dsg} and Theorem~\ref{theorem partial serre on gproj}. Our strategy is to investigate how the partial Serre functors of Theorem~\ref{theorem serre functors in K(Inj) and K(Proj)} induce partial Serre functors in certain subcategories of $\catK(\Inj R)$ and of $\catK(\Proj R)$. This will rely on the following general observation.

\begin{lemma} \label{lemma serre duality from an adjoint triple}
Take an adjoint triple of triangle functors
\[ \begin{tikzcd}
\catT' \ar[r,"e"] \ar[r,<-,"\bigR",bend right=40] \ar[r,<-,"\bigL",bend left=40] &[2em] \catT
\end{tikzcd} \]
and let $\mathbb S\colon \catX \to \catT$ be a partial Serre functor for a subcategory $\catX\subset\catT$.

Then the essential image $\bigL(\catX)\subset\catT'$ admits a partial Serre functor $\mathbb S'\colon \bigL(\catX) \to \catT'$ given by $\mathbb S'(\bigL X)= \bigR\mathbb S X$ for each $X\in\catX$.
\end{lemma}
\begin{proof}
$\catT' (\bigL X,-)^\ast \cong \catT(X, e-)^\ast \cong \catT(e-, \mathbb S X) \cong \catT'(-, \bigR \mathbb S X)$.
\end{proof}

\begin{example} \label{ex Serre for derived}
Note that in Lemma~\ref{lemma serre duality from an adjoint triple} it suffices for \( \bigL \) to be defined on \( \catX \) and \( \bigR \) to be defined on \( \mathbb{S} (\catX) \). An instance employing such a partially defined right adjoint is the following.

Consider \( \catD(\Modcat R) \), identified (via injective Cartan--Eilenberg resolutions) with the full subcategory of homotopically injective complexes in \( \catK(\Inj R) \). Note that this inclusion has a left adjoint: the natural projection \( q \colon \catK(\Inj R) \to \catD(\Modcat R) \). In general \( q \) does not preserve compacts. Therefore our inclusion cannot have a right adjoint in general.

However if we consider \( \catX = \lambda(\Kb(\proj R)) \), the category of perfect complexes as a subcategory of \( \catK(\Inj R) \) via injective resolutions, then the situation improves: By Theorem~\ref{theorem serre functors in K(Inj) and K(Proj)} we have a partial Serre functor
\[ \mathbb{S}_{\catK(\Inj R)} \colon \catX \to \catK(\Inj R) \colon (\lambda X) \longmapsto \nu X \]
for \( X \in \Kb(\proj R) \). Of course now the essential image \( \mathbb{S}_{\catK(\Inj R)} (\catX) \) consists of finite complexes of injectives, and thus lies inside our subcategory of homotopically injectives. Trivially we get a right adjoint defined on \( \mathbb{S}_{\catK(\Inj R)} (\catX) \), namely the identity.

Thus we obtain the partial Serre functor \( \Kb(\proj R) \to \catD(\Modcat R) \) given by \( \nu \).
\end{example}

\begin{observation} \label{observation induced rel serre along left adjoint}
  Consider a compactly generated triangulated category $\catT$ and some $X\in\catT^{\compacts}$. By Theorem~\ref{theorem Neeman trick} the subcategory $X^\perp\subset\catT$ is compactly generated again, so Corollary~\ref{corollary compact generation gives partial serre} yields the existence of partial Serre functors \[\mathbb S\colon\catT^{\compacts}\to\catT \text{ and } \mathbb S'\colon (X^{\perp})^{\compacts}\to X^{\perp}.\] We may now observe that the latter is induced by the former: The subcategory $X^\perp$ is closed under products and, since $X$ is compact, also under coproducts. By Theorem~\ref{thm:neemanthreeadjoints} the inclusion functor $X^\perp \hookrightarrow \catT$ thus admits a left adjoint $\bigL$ and a right adjoint $\bigR$. Moreover, $(X^\perp)^{\compacts}= \bigL(\catT^{\compacts})$ by Theorem~\ref{theorem Neeman trick} (up to direct summands). In particular, since the partial Serre functor $\mathbb S'$ is unique up to natural isomorphism, Lemma~\ref{lemma serre duality from an adjoint triple} yields the following commutative diagram.
  \[ \begin{tikzcd}
  (X^\perp)^{\compacts} \ar[d,"\mathbb S' "]  & \catT^{\compacts} \ar[d,"\mathbb S"] \ar[l,swap,"\bigL"]  \\
  X^\perp & \catT  \ar[l,swap,"\bigR"]
  \end{tikzcd} \]
\end{observation}

\subsection*{The singularity category}
Let $\Kac(\Inj R)$ be the subcategory of acyclic complexes in $\catK(\Inj R)$. Note that $\lambda (R)$ is compact when considered as an object in $\catK(\Inj R)$, and moreover that $\Kac(\Inj R) = (\lambda (R))^\perp$ as subcategories of $\catK(\Inj R)$. Using this fact, one may construct --- see e.g.\ \cite{MR2157133} --- for $R$ noetherian, an adjoint triple
\[ \begin{tikzcd}
\Kac(\Inj R) \ar[r,hookrightarrow] & \catK(\Inj  R).  \ar[l,bend right=45,swap,"\bigL"] \ar[l,bend left=45,swap,"\bigR"]
\end{tikzcd} \]
The left adjoint $\bigL$ takes the subcategory $\catK(\Inj R)^{\compacts}\simeq\Db(\modcat R)$ to a compact generating set for $\Kac(\Inj R)$; up to direct summands, $\Kac(\Inj R)^{\compacts}$ is equivalent to the singularity category $\Dsg(R)$: More explicitly, the sequence of constructions
\[ \Dsg(R) \xto{\text{pick preimage}} \Db(\modcat R) \xto{\bigL \lambda} \Kac(\Inj R)^{\compacts} \]
is a well-defined functor, which is fully faithful and dense up to summands.

\begin{theorem}  \label{theorem partial serre functor on Dsg}
Let $R$ be a noetherian ring. There is a partial Serre functor \[\mathbb S\colon \Dsg(R)\to\Kac(\Inj R)\] given as follows. For an object \( \overline{X} \) in \( \Dsg(R) \), pick a representative \( X \in \Db(\modcat R) \).  Then $\mathbb S \overline{X} =\bigR \nu \rho (X)$.
\end{theorem}

\begin{proof}
In view of the adjoint triple above, the claim follows immediately from Theorem~\ref{theorem serre functors in K(Inj) and K(Proj)} and Observation~\ref{observation induced rel serre along left adjoint} --- note here that \( \overline{X} \) is identified with \( \bigL \lambda X \) when considering it as an object of \( \Kac(\Inj R) \).
\end{proof}

It may not be completely obvious that Theorem~\ref{theorem partial serre functor on Dsg} is in fact a generalization of Theorem~\ref{theorem BIKP}. However, If \( R \) is Gorenstein in the sense of Iwanaga \cite{MR597688}, then we have the following more direct description.

\begin{corollary} \label{cor partial serre for Dsg(gorenstein)}
Let \( R \) be an Iwanaga--Gorenstein noetherian ring. The partial Serre functor $\mathbb S\colon \Dsg(R)\to\Kac(\Inj R)$ of Theorem~\ref{theorem partial serre functor on Dsg} is induced by the functor \( \mathbb{L} \nu[-1] \colon \Db(\modcat R) \to \Db(\Modcat R) \). More explicitly we have
\[ \mathbb{S} \bigL \lambda X = \bigL \lambda \mathbb{L} \nu X[-1] \]
for \( X \in \Db(\modcat R) \).
\end{corollary}

\begin{proof}
Note first that \( \mathbb{L} \nu \) maps finite complexes of projectives to finite complexes of injectives, which in turn vanish when pushed to \( \Kac(\Inj R) \). In particular both sides of the above equation vanish on \( \Kb(\proj R) \), and we may replace \( X \) by a Gorenstein projective finitely generated module concentrated in degree \( 0 \). Note in particular that after this replacement \( \mathbb{L} \nu X = \nu X \).

By definition of Gorenstein projective there is a totally acyclic complex \( P \) of finitely generated projectives, such that \( X = \boundary^1 P \). After applying \( \nu \) we obtain the following complex, which is still exact.
\[ \nu P \colon \underbrace{\cdots \to \nu P^{-1} \to \nu P^0}_{\nu \rho X} \to \nu P^1 \to \nu P^2 \to \cdots \]
Recall that there are no maps from acyclic complexes to left bounded complexes of injectives in the homotopy category. Therefore the natural map
\[ \Hom_{\catK}(-, \nu P) \to \Hom_{\catK}(-, \nu \rho X) \]
induces an isomorphism of functors on \( \Kac(\Inj R) \). Since \( \nu P \in \Kac(\Inj R) \) this means that \( \bigR \nu \rho X = \nu P \).

On the other hand, observe that the right half of \( \nu P \) is an injective resolution of \( \nu X \), more precisely
\[ \lambda \nu X [-1] = \left( \cdots \to 0 \to 0 \to \nu P^1 \to \nu P^2 \to \cdots\right), \]
where the shift is due to the fact that the complex on the right hand side begins in degree \( 1 \). Similarly to before, we can observe that \( \bigL \lambda \nu X [-1] = \nu P\): Here we use that \( \Hom_{\catK}(I, \Kac(\Inj R)) = 0 \) for any right bounded complex of injectives \( I \). Indeed, given any map of complexes \( I \to \Kac(\Inj R) \), using the fact that the terms of \( I \) have finite projective dimension (see \cite[Theorem~2]{MR597688}) one iteratively from right to left constructs a null-homotopy.

In particular \( \nu P \) lies in the image of \( \bigL \lambda \) again, and we have
\[ \mathbb{S} \bigL \lambda X \overset{\text{Thm~\ref{theorem partial serre functor on Dsg}}}{=} \bigR \nu \rho (X) = \nu P = \bigL \lambda \nu X [-1]. \qedhere \]
\end{proof}

\subsection*{Gorenstein projectives} Before we go on, let us briefly revisit the Gorenstein projective approximation functor. Let $\Ktac(\Proj R)$ denote the subcategory of totally acyclic complexes in $\catK(\Proj R)$, i.e.\ those exact complexes which remain exact under $\Hom_R(-,\Proj R)$. Recall e.g.\ from \cite[Observation~2.21]{MR3946864} that if the noetherian ring $R$ admits a dualizing complex, then there is an adjoint triple
\[ \begin{tikzcd}
\Ktac(\Proj R) \ar[r,hookrightarrow] & \catK(\Proj R).  \ar[l,bend right=45,swap,"\bigL"] \ar[l,bend left=45,swap,"\bigR"]
\end{tikzcd} \]
The construction of this triple utilizes the fact that if \( D_R \) is a dualizing complex then \[\Ktac(\Proj R)=\left(R\oplus \rho\Hom(D_R,\lambda R)\right)^\perp\] as subcategories of $\catK(\Proj R)$, and that both $R$ and $\rho\Hom(D_R, \lambda R)$ are compact in $\catK(\Proj R)$. See  \cite[Proposition~2.15]{MR3946864}.
\begin{proposition} \label{proposition description of the functor GP}
Let $R$ be a noetherian ring which admits a dualizing complex. The inclusion $\GStable R\hookrightarrow\stableModules R$ admits a right adjoint $\GP\colon\stableModules R\to\GStable R$, which is given by choosing a preimage in $\Modcat R$ before applying \[\catK(\Modcat R)\stackrel{\rho}\to\catK(\Proj R)\stackrel{\bigR}\to\Ktac(\Proj R)\stackrel{\boundary^1}\to\GStable R.\]
\end{proposition}
\begin{proof}
  Recall that taking (e.g.\ first) boundaries gives a triangle equivalence \[\boundary^1\colon\Ktac(\Proj R)\stackrel{\simeq}\to\GStable R.\] The quasi-inverse $\CR$ of $\boundary^1$ takes a Gorenstein projective $R$-module $X$ to its \emph{complete resolution} $\CR X$, i.e.\ a totally acyclic complex over $\Proj R$ with $\boundary^1(\CR X)=X$.

  Now, take $X\in\GProj R$ and $M\in\Modcat R$. We first observe that \[\underline\Hom_R(X,M)\cong\Hom_{\catK}(\CR X,  M).\] Indeed, there is an epimorphism $\phi\colon \Hom_R(X,M)\to\Hom_{\catK}(\CR X,M)$, as indicated by the following diagram.
  \[ \begin{tikzcd}[row sep=4mm]
  \CR X\colon  \cdots \ar[r] & P^{-1} \ar[r] \ar[dd] &  P^0  \ar[dr,->>] \ar[dd] \ar[rr] && P^1 \ar[r]\ar[dd] & \cdots  \\
  &&&X\ar[ur,>->] \ar[dl]&&& \\
  M\colon  \cdots \ar[r] & 0 \ar[r]  &  M  \ar[rr] && 0 \ar[r] & \cdots
  \end{tikzcd} \]
  The total acyclicity of $\CR X$ means in particular that the morphism $\iota\colon X\hookrightarrow P^1$ is a left $\Proj R$-approximation of $X$. Hence, a morphism $f\colon X\to M$ factors through $\Proj R$ if and only if it factors through $\iota$, which is further equivalent to $f\in\Ker\phi$.

  To complete the proof, it now suffices to use the right adjointness of $\rho$ and $\bigR$, and the fact that $\boundary^1$ and $\CR$ are quasi-inverse:
  \begin{align*}
     \underline\Hom_R(X,M) & \cong \Hom_{\catK}(\CR X, M) \\
   & \cong \Hom_{\catK}(\CR X, \bigR\rho M) \\
   & \cong \underline\Hom_R(X, \boundary^1\bigR\rho M) \qedhere
 \end{align*}
  \end{proof}
\begin{remark}
  In the case of an Artin algebra $\Lambda$, this description of the functor $\GP$ is fairly explicit. Indeed, the standard dual $D\Lambda$ is a dualizing complex, and $\bigR(X)$ can be calculated by taking `iterated approximations' of $X$ by products of copies of $\Lambda \oplus \rho D\Lambda$. For details, the reader is referred to \cite[Theorem~6.6; Corollary~6.12]{MR3946864}
\end{remark}

We are now ready to complete this subsection with the following extension of Theorem~\ref{theorem BIKP} for categories of Gorenstein projectives.

\begin{theorem} \label{theorem partial serre on gproj}
  Let $R$ be a noetherian ring which admits a dualizing complex. There is a partial Serre functor \[\mathbb S\colon (\GStable R)^{\compacts}\to\GStable R.\]
For \( X \in \Gstable R \) it is given by $\mathbb S(X)=\GP\nu(X)[-1]$.
\end{theorem}

\begin{proof}[Proof of Theorem~\ref{theorem partial serre on gproj}]
By Observation~\ref{observation induced rel serre along left adjoint}, the adjoint triple ensures the existence of a partial Serre functor \[\mathbb S_{\tac} \colon \Ktac(\Proj R)^{\compacts}\to\Ktac(\Proj R),\] induced by the partial Serre functor $\catK(\Proj R)^{\compacts}\to\catK(\Proj R)$ of Theorem~\ref{theorem serre functors in K(Inj) and K(Proj)}.

 It is clear that we can transfer this partial Serre functor to \( \GStable R \) as the composition \[\mathbb S\colon(\GStable R)^{\compacts}\xto{\CR}\Ktac(\proj R)^{\compacts}\xto{\mathbb S_{\tac}} \Ktac(\Proj R)\stackrel{\boundary^1}\to\GStable R.\]

\medskip
Now let \( X \in \Gproj R \), and consider
\[ \begin{tikzcd}
\CR X\colon \cdots \ar[r] & P^{-1} \ar[r]  &  P^0  \ar[dr,->>] \ar[rr] && P^1 \ar[r] & P^2 \ar[r]& \cdots  \\
&&&X\ar[ur,>->] &&&
\end{tikzcd} \]
Observe that $\left(\cdots \to (P^2)^{\vee} \to (P^1)^{\vee} \to 0 \to \cdots\right)$ is a projective resolution of \( X^{\vee}[1] \) (when considering \( (P^i)^{\vee} \) in homological degree \( -i \)). It follows that the right half of the above complex is \( \rho( X^{\vee}[1] )^{\vee} \). Since there are no maps in the homotopy category from right bounded complexes of projectives to acyclic complexes we see that
\[ \Hom_{\catK}( \CR X, -) \to \Hom_{\catK}(\rho(X^{\vee}[1])^{\vee}, -) \]
induces a natural isomorphism on \( \Ktac(\Proj R) \), i.e.\ \( \CR X =  \bigL (\rho( X^{\vee}[1])^{\vee}) \).

Now
\begin{align*}
\mathbb{S}(X) & =  \boundary^1 \mathbb S_{\tac}\CR(X) =  \boundary^1 \mathbb S_{\tac} \bigL (\rho( X^{\vee}[1])^{\vee}) \\
& = \boundary^1 \bigR \mathbb{S}_{\catK(\Proj R)} (\rho( X^{\vee}[1])^{\vee}) && \text{by Observation~\ref{observation induced rel serre along left adjoint}} \\
& = \boundary^1 \bigR \rho \underbrace{(X^{\vee}[1])^*)}_{= \nu X[-1]} && \text{by Theorem~\ref{theorem serre functors in K(Inj) and K(Proj)}} \\
& =  \GP\nu(X)[-1] && \text{Proposition~\ref{proposition description of the functor GP}} \qedhere
\end{align*}
\end{proof}

\section{Pure resolutions}

Let $R$ be a ring. Following Cohn \cite{MR106918}, a sequence \[\mathbb E\colon0\to A\to B\to C\to0\] in $\Modcat R$ is called \emph{pure exact} if $\mathbb E\otimes_R N$ is exact for each $N\in\Modcat R^{\op}$ or, equivalently, if $\Hom_R(M,\mathbb E)$ is exact for each $M\in\modcat R$. The pure exact sequences define an exact structure on $\Modcat R$; an $R$-module is \emph{pure-projective} (resp. \emph{pure-injective}) if it is projective (injective) with respect to this exact structure. We denote by $\PProj R$ (resp. $\PInj R$) the subcategory of pure-projectives (pure-injectives) in $\Modcat R$.

In this section we will utilize the following theorem. For a collection of objects $\catS$ in a triangulated category, we denote by $\Loc(\catS)$ the smallest triangulated subcategory which contains $\catS$ and is closed under coproducts. On the other hand, $\Coloc(\catS)$ is the smallest triangulated subcategory which contains $\catS$ and is closed under products.

\begin{theorem} \label{theorem the two t-structures}
  Let $\catT$ be a triangulated category.
  \begin{enumerate}
    \item If $\catT$ admits coproducts and $\catS$ is a set of compact objects, then \[\left({}^\perp(\catS^\perp), \catS^\perp\right)\] is a stable t-structure in $\catT$.  Moreover, ${}^\perp(\catS^\perp)=\Loc(\catS)$.
    \item If $\catT$ admits products and $\catS$ is a set of 0-cocompact objects, then \[\left({}^\perp\catS, ({}^\perp\catS)^\perp\right)\] is a stable t-structure in $\catT$.  Moreover, $({}^\perp\catS)^\perp=\Coloc(\catS)$.
  \end{enumerate}
 \end{theorem}
 \begin{proof}
   For (1) see for example \cite{MR2327478,MR2927802,MR1191736}; (2) is \cite[Theorem~6.6]{MR3946864}.
 \end{proof}

Since any finitely presented $R$-module is pure-projective and, as with any notion of projectivity, $\PProj R$ is closed under summands and coproducts, we have
\begin{equation} \label{eqn inclusion 1}
  \Loc(\modcat R)\subset\catK(\PProj R), \tag{$\text{i}_1$}
\end{equation} up to isomorphism.

On the other hand, recall from Section~\ref{section rel serre} the functor $(-)^\ast=\Hom_k(-,I)$ (note that we may always choose \( k = \mathbb Z \) if there is no other base ring). For each $N\in\Modcat R^{\op}$, the dual $N^\ast$ belongs to $\PInj R$. Indeed, if $\mathbb E$ is a pure exact sequence, then the sequence $\Hom_R(\mathbb E, N^\ast)\cong \Hom_k(\mathbb E\otimes_R N, I)$ is exact. From the description in Theorem~\ref{theorem partial serre for Kb(mod R)} of the partial Serre functor $\mathbb S\colon \Kb(\modcat R)\to\catK(\Modcat R)$, we infer
\begin{equation} \label{eqn inclusion 2}
  \Coloc(\mathbb S(\modcat R))\subset\catK(\PInj R). \tag{$\text{i}_2$}
\end{equation}

A complex $X$ of $R$-modules is \emph{pure acyclic} if $\Hom_R(M,X)$ is acyclic for each $M\in\modcat R$. Let $\Kpac(\Modcat R)$ be the subcategory of $\catK(\Modcat R)$ consisting of pure acyclic complexes. A chain map $f\colon X \to Y$ is a \emph{pure quasi-isomorphism} if $\Cone(f)\in\Kpac(\Modcat R)$, and the \emph{pure derived category} of $R$ is the Verdier quotient \[\Dpure(R)=\frac{\catK(\Modcat R)}{\Kpac(\Modcat R)}.\] The first point of this section is that $\Dpure(R)$ may be realized both as a subcategory of $\catK(\PProj R)$ and as a subcategory of $\catK(\PInj R)$:

\begin{theorem} \label{theorem Dpure as subcat of KPProj and KPInj}
Let \( R \) be a ring.
\begin{enumerate}
\item The pair
\[ ( \Loc( \modcat R ), \Kpac(\Modcat R) ) \]
is a stable t-structure in \( \catK(\Modcat R) \). In particular \( \Dpure(R) \) is equivalent to \( \Loc( \modcat R) \), and the natural quotient functor \( \catK(\Modcat R) \to \Dpure(R) \) has a fully faithful left adjoint.
\item The pair
\[ ( \Kpac(\Modcat R), \Coloc(\mathbb S(\modcat R)) ), \]
is a stable t-structure in \( \catK(\Modcat R) \). In particular \( \Dpure(R) \) is equivalent to \( \Coloc( \mathbb S( \modcat R)) \), and the natural quotient functor \( \catK(\Modcat R) \to \Dpure(R) \) has a fully faithful right adjoint.
\end{enumerate}
\end{theorem}

\begin{proof}
Observe that \( \Kpac(\Modcat R) = (\modcat R)^{\perp} \) by definition, and thus it follows that we also have \( \Kpac(\Modcat R) = {}^{\perp}\mathbb S(\modcat R) \). Moreover, by Theorem~\ref{theorem partial Serre implies X compact and Y 0-cocompact} we know that \( \modcat R \) consists of compact objects while \( {}^{\perp}\mathbb S(\modcat R) \) consists of \( 0 \)-cocompact objects.

Now the stable t-structures exist by Theorem~\ref{theorem the two t-structures}. For the final claims, note that for a stable t-structure we always have that the localization by the aisle is the co-aisle and vice-versa.
\end{proof}

\begin{definition} Let $R$ be a ring. An object of $\catK(\Modcat R)$ is called
\begin{enumerate}
  \item \emph{homotopically pure-projective} if it belongs to ${}^\perp\Kpac(\Modcat R)$; and
  \item \emph{homotopically pure-injective} if it belongs to $\Kpac(\Modcat R)^\perp$.
\end{enumerate}
\end{definition}

\begin{corollary} \label{cor homotopically pure proj} Let $R$ be a ring.
\begin{enumerate}
\item The subcategory of homotopically pure-projectives coincides with the subcategory \( \Loc(\modcat R) \). In particular it is contained in \( \catK( \PProj R) \) (up to isomorphism).
\item The subcategory of homotopically pure-injectives coincides with the subcategory \( \Coloc(\mathbb{S}(\modcat R)) \). In particular it is contained in \( \catK( \PInj R) \) (up to isomorphism).
\end{enumerate}
\end{corollary}

\begin{proof}
By Theorem~\ref{theorem Dpure as subcat of KPProj and KPInj}, we have \( \Loc( \modcat R) = {}^{\perp}\Kpac(\Modcat R) \). The ``in particular''-statement follows with \eqref{eqn inclusion 1}. The dual argument proves the second point.
\end{proof}

\begin{remark}
By \cite[Theorem~3.6]{MR3537821} the class of homotopically pure-projectives in $\catK(\Modcat R)$ actually \emph{coincides} with $\catK(\PProj R)$. Combining this fact with our discussion it follows immediately that there is a triangle equivalence
\[\Dpure(R)\simeq\catK(\PProj R)\]
for any ring $R$.
\end{remark}

%\begin{theorem} \label{theorem emmanouil}
%    Let $R$ be a ring. An object of $\catK(\Modcat R)$ is pure acyclic if and only if it belongs to the subcategory $\catK(\PProj R)^\perp$.
%\end{theorem}

Let $X$ be a complex of $R$-modules. A \emph{pure-projective resolution} of \( X \) is a pure quasi-isomorphism $P \to X$, with \( P \) homotopically pure-projective. Dually, a \emph{pure-injective resolution} of \( X \) is a pure quasi-isomorphism \( X \to I \), with \( I \) homotopically pure-injective.

The existence of pure-projective (resp. pure-injective) resolutions was established for left (resp. right) bounded complexes in \cite{MR3473427}. We can now get rid of these restrictions:

\begin{corollary} \label{corollary existence of pure resolutions}
Let $R$ be a ring. Each complex of $R$-modules admits
\begin{enumerate}
\item a pure-projective resolution; and
\item a pure-injective resolution.
\end{enumerate}
\end{corollary}

\begin{proof}
Take $X\in\catK(\Modcat R)$. By Theorem~\ref{theorem Dpure as subcat of KPProj and KPInj}, there is a triangle
\[ P \to X \to A \to P[1] \]
with \( P \in \Loc(\modcat R) \) and \( A \in \Kpac(\Modcat R) \). By Corollary~\ref{cor homotopically pure proj}, \( P \) is homotopically projective. Since \( A \) is pure acyclic the map \( P \to X \) is a pure quasi-isomorphism.

The proof of the second point is dual.
\end{proof}

In \cite{MR3473427}, a \emph{pure-projective resolution} of $X$ is defined to be a pure quasi-isomorphism $P \to X$ where $P\in\catK(\PProj R)$ is such that $\Hom_R(P,-)\colon \catK(\Modcat R)\to\catK(\Modcat k)$ preserves pure acyclicity. Dually, a \emph{pure-injective resolution} of $X$ is a pure quasi-isomorphism $X \to I$ where $I\in\catK(\PInj R)$ is such that $\Hom_R(-,I)$ preserves pure acyclicity. We finish this section by showing that the definition of pure-projective and -injective resolutions we worked with here is in fact equivalent to the one in \cite{MR3473427}.

\begin{lemma} \label{closure prop of pure acyclics}
Let $R$ be a ring and let $X$ be a pure acyclic complex of $R$-modules.
\begin{enumerate}
\item The complex $\Hom_k(F, X)$ is pure acyclic for each $F\in\modcat k$.

\item The complex $F \otimes_k X$ is pure acyclic for each \( k \)-module \( F \).
\end{enumerate}
\end{lemma}
\begin{proof}
(1): Let \( M \in \modcat R \). By the adjunction formula we have
\[
\Hom_R(M, \Hom_k(F, X)) = \Hom_R(F \otimes_k M, X).
\]
The latter complex is acyclic, because \( F \otimes_k M \) is a finitely presented \( R \)-module again and $X$ is pure acyclic.

(2): Let \( M \in \Modcat R^{\op} \). We have
\[
(F \otimes_k X) \otimes_R M = F \otimes_k (X \otimes_R M) = (X \otimes_R M) \otimes_k F = X \otimes_R (M \otimes_k F),
\]
where the middle identity holds since \( k \) acts centrally on \( R \). The final tensor product is acyclic by definition of pure acyclicity.
\end{proof}

\begin{proposition} \label{prop.characterizations of homotopically pure proj/inj}
Let $R$ be a ring.
\begin{enumerate}
\item Let $P$ be a complex of $R$-modules. The functor \( \Hom_R(P, -) \) preserves pure acyclicity if and only if $P$ is homotopically pure-projective.

\item Let $I$ be a complex of $R$-modules. The following are equivalent
\begin{enumerate}[label=(\roman*)]
\item $I$ is homotopically pure-injective.
\item \( \Hom_R(-, I) \) preserves pure acyclicity.
\item \( \Hom_R(-, I) \) maps pure acyclic complexes to contractible complexes.
\end{enumerate}
\end{enumerate}
\end{proposition}

\begin{proof}
(1): We observe first that \( P \) is homotopically pure-projective if and only if \( \Hom_R(P, -) \) maps pure acyclic complexes to acyclic complexes. This is just because the homology of the complex $\Hom_R(P, X)$, where $X$ is pure acyclic, is the \( \Hom \)-space in the homotopy category which is zero because $P$ lies in ${}^{\perp}\Kpac(\Modcat R) $.

So it only remains to show that $\Hom_R(P, X)$ is also pure exact. Let \( F \in \modcat k \). Then we have the isomorphism
\[
\Hom_k(F, \Hom_R(P, X)) = \Hom_R(F \otimes_k P, X) =  \Hom_R(P, \Hom_k(F, X)),
\]
and the claim follows from the fact that \( \Hom_k(F, X) \) is pure acyclic from Lemma~\ref{closure prop of pure acyclics}.

(2): As in (1), we see that $(ii)$ implies $(i)$. The implication from $(iii)$ to $(ii)$ is immediate. It remains to be shown that $(i)$ implies $(iii)$. So let \( X \) be a pure acyclic complex. Then for any \( k \)-module \( F \) we have
\[
\Hom_k(F, \Hom_R(X, I)) = \Hom_R(F \otimes_k X, I)
\]
which is acyclic, because \( F \otimes_k X \) is pure acyclic by Lemma~\ref{closure prop of pure acyclics}. It follows (picking \( F \) to be cycles of the complex \( \Hom_R(X, I) \)) that \( \Hom_R(X, I) \) is contractible.
\end{proof}

\begin{corollary}
Let \( R \) be a commutative ring and let \( I\in\Modcat R \) be pure-injective. If \( 0 \to A \to B \to C \to 0 \) is a pure exact sequence of $R$-modules, then
\[ 0 \to \Hom_R(C, I) \to \Hom_R(B, I) \to \Hom_R(A, I) \to 0 \]
is split exact.
\end{corollary}

\begin{proof}
Since \( R \) is commutative, we may choose \( k = R \). The pure exact sequence is a pure acyclic complex, and \( I \) --- considered as a complex concentrated in degree \( 0 \) --- is homotopically pure-injective. Now the claim follows from the implication from (a) to (c) in Proposition~\ref{prop.characterizations of homotopically pure proj/inj}(2).
\end{proof}

\begin{remark}
In \cite[Theorem~5.4]{MR3220541} \v{S}\v{t}ov\'{\i}\v{c}ek stated a version of Corollary~\ref{corollary existence of pure resolutions} for an additive finitely accessible category $\mathcal{A}$, using the language of cotorsion pairs. More precisely, he proved that $(\catC(\PProj \mathcal{A}), \catC_{\mathsf{pac}}(\mathcal{A}))$ and $(\catC_{\mathsf{pac}}(\mathcal{A}), \catC(\PInj \mathcal{A}))$ are functorially complete hereditary cotorsion pairs in the category of complexes $\catC(\mathcal{A})$ with the induced
pure exact structure. It would be interesting to find a proof in the general context of additive finitely accessible categories using the direct approach of appealing to the stable t-structures in Theorem~\ref{theorem the two t-structures}.
\end{remark}

\section{Almost split triangles} \label{section AR-triangles}
Let $\catT$ be a triangulated category. Recall that a triangle \[A \stackrel{a}\to B \stackrel{b}\to C \to A[1]\] is called \emph{almost split} if $a$ is left almost split and $b$ is right almost split. In this case $\End_{\catT}(A)$ and $\End_{\catT}(C)$ are local rings.

\begin{theorem}[Beligiannis \cite{MR2079606}, Krause \cite{MR1803642}] \label{theorem Krause AR-theorem}
Let \( X \in \catT \) have local endomorphism ring. Denote by $I_X$ an injective envelope of the simple $\End_{\catT}(X)$-module. If the functor $\Hom_{\End_{\catT}(X)}(\catT(X,-), I_X)$ is representable, then $X$ appears in an almost split triangle
 \begin{align*} \label{align Krauses AR triangle}
 \tau X \to M \to X \to \tau X[1]. \tag{$\Delta_{\tau}$}
 \end{align*}
\end{theorem}

\begin{proof}[Idea of proof]
Let $\rho\colon \End_{\catT}(X)\doublerightarrow \End_{\catT}(X)/\!\rad \End_{\catT}(X) \hookrightarrow I_X$ be the canonical map. By assumption the functor \( \Hom_{\End_{\catT}(X)}(\catT(X,-), I_X) \) is representable, and we can choose \( \tau X \) such that \( \tau X[1] \) is a representative. In other words, there is a natural isomorphism
\[ \phi \colon \Hom_{\End_{\catT}(X)}(\catT(X,-), I_X) \to \catT(-,\tau X [1]). \]
It is routine to check that we have an almost split triangle \[\tau X \to M \to X\xto{\phi_X(\rho)} \tau X[1]. \qedhere\]
\end{proof}

In Theorem~\ref{theorem Krause AR-theorem} the computation of the object $\tau X$ relies on intrinsic properties of the ring $\End_{\catT}(X)$. Our goal now is to show that in the presence of a partial Serre functor, a more unified approach to calculating $\tau$ is sometimes available. We keep our injective cogenerator $I$ of $\Modcat k$, and $(-)^\ast = \Hom_k (-,I)$.

\begin{lemma} \label{lemma our triangle with right almost split map}
Suppose $\mathbb S\colon \catX\to\catT$ is a partial Serre functor. Then each $X\in\catX$ with local endomorphism ring appears in a triangle \begin{align*} \label{align our triangle with right almost split map}
\mathbb SX[-1]\to N \stackrel{n}\to X \to \mathbb SX \tag{$\Delta_{\mathbb S}$}
\end{align*}
with $n$ right almost split.
\end{lemma}
\begin{proof}
By assumption there is an isomorphism $\phi\colon \End_{\catT}(X)^\ast \to \catT(X,\mathbb SX)$. For each non-zero linear form $\gamma$ on $\End_{\catT}(X)$ which vanishes on $\rad \End_{\catT}(X)$, the triangle \[\mathbb SX[-1]\to N \to X \xto{\phi(\gamma)} \mathbb SX\] has the desired property. Indeed, it suffices to observe that any radical morphism $Y\to X$ composes to zero with $\phi(\gamma)$.
\end{proof}

Our first aim is to show that in the setup of Lemma~\ref{lemma our triangle with right almost split map} there is an almost split triangle ending in \( X \), and moreover that this triangle is a direct summand of (\ref{align our triangle with right almost split map}).

\begin{theorem} \label{theorem our triangle summand of AR triangle}
Assume \( \catT \) is idempotent closed. (Note that this is automatic for instance if \( \catT \) has countable products or coproducts.) Suppose $\mathbb S\colon \catX\to\catT$ is a partial Serre functor. Let \( X \in \catX \) be an object with local endomorphism ring. Then
\begin{enumerate}
\item The functor $\Hom_{\End_{\catT}(X)}(\catT(X,-), I_X)$ of Theorem~\ref{theorem Krause AR-theorem} is representable, so in particular the almost split triangle  (\ref{align Krauses AR triangle}) exists.
\item The triangle (\ref{align Krauses AR triangle}) is a direct summand of (\ref{align our triangle with right almost split map}).
\end{enumerate}
\end{theorem}

For the proof, we prepare the following two lemmas.

\begin{lemma} \label{lemma functors summand}
Let \( X \in \catT \) be an object with local endomorphism ring. Then the functor \( \Hom_{\End_{\catT}(X)}(\catT(X,-), I_X) \) is a direct summand of \( \catT(X, -)^\ast \).
\end{lemma}

\begin{proof}
We write \( E = \End_{\catT}(X) \), and denote by \( S \) its simple module.

Since \( \catT(X, -)^\ast = \Hom_{E}(\catT(X, -), E^*) \) it suffices to show that \( I_X \) is a direct summand of \( E^\ast \). Observe that since $\Hom_E(-, E^\ast) = (- \otimes_E E)^\ast$, the $E$-module \( E^\ast \) is injective. Thus, by definition of \( I_X \) it suffices to show that there is a monomorphism from \( S \) to \( E^\ast \).

Since \( I \) is a cogenerator of \( \Modcat k \) there is a non-zero map \( \rho \colon S \to I \). Now we obtain the desired injection as
\[ S \to E^\ast \colon s \longmapsto [ e \mapsto \rho(se) ]. \qedhere \]
\end{proof}

\begin{lemma} \label{lemma AR triangle summand}
Let \( \catT \) be a triangulated category, let \( \tau X \to M \stackrel{m}\to X \stackrel{s}\to \tau X [1] \) be an almost split triangle, and let \( Y \to N \stackrel{n}\to X \to Y[1] \) be a triangle with \( n \) right almost split. Then the former triangle is a direct summand of the latter.
\end{lemma}

\begin{proof}
Consider the following diagram.
\[ \begin{tikzcd}
\Delta_1\colon & \tau X \ar[r] & M \ar[r,"m"] &  X  \ar[d,equal] \ar[r,"s"] & \tau X[1]  \\
\Delta_2\colon & Y \ar[r] & N \ar[r,"n"] &  X  \ar[r] & Y[1]
\end{tikzcd} \]
Since $m$ and $n$ are both right almost split, the former factors through the latter, and vice versa. This gives rise to morphisms of triangles $\iota\colon\Delta_1 \to \Delta_2$ and $\pi\colon\Delta_2 \to \Delta_1 $. In particular, there are morphisms $i\colon \tau X \to Y$ and $p\colon Y \to \tau X$ such that $(pi)[1] \circ s=s$. But since  $\End_{\catT}(\tau X)$ is local, the non-zero morphism $s$ is left minimal. Hence $i$ is a split monomorphism, i.e.\ (\( \Delta_1 \)) is a summand of (\( \Delta_2 \)).
\end{proof}

Now the proof of Theorem~\ref{theorem our triangle summand of AR triangle} is very short.

\begin{proof}[Proof of Theorem~\ref{theorem our triangle summand of AR triangle}]
The first point follows from Lemma~\ref{lemma functors summand}: Since \( \catT \) is idempotent closed, direct summands of representable functors are representable again. Once the first point is established, the second one is an immediate application of Lemma~\ref{lemma AR triangle summand}
\end{proof}

Our next aim is to show that in certain cases, there is no difference between the triangles (\ref{align Krauses AR triangle}) and (\ref{align our triangle with right almost split map}). More precisely, we will show the following.

\begin{theorem} \label{theorem.S triangle almost split}
Assume that \( k \) is noetherian, and let \( I = \coprod_{\mathfrak{m} \in \MaxSpec k} I(k / \mathfrak{m}) \) be the direct sum of the injective envelopes of the simple \( k \)-modules.

Let $\mathbb S\colon \catX\to\catT$ be a partial Serre functor. If the endomorphism ring of $X\in\catX$ is local and finite over $k$, then the triangle (\ref{align our triangle with right almost split map}) is almost split.
\end{theorem}

Also for the proof of this theorem we prepare several lemmas.

\begin{lemma} \label{lemma.induces local map}
Let \( k \) be noetherian, and let \( E \) be a finite \( k \)-algebra which is local. Then \( \mathfrak{m} = \Ker\left( k \to E / \!\rad E\right) \) is a maximal ideal of \( k \).
\end{lemma}

\begin{proof}
Since the target of the map above is a skewfield we observe that \( \mathfrak{m} \) is a prime ideal. Moreover we note that the quotient field of \( k / \mathfrak{m} \) is a \( k \)-submodule of \( E / \rad E \). Since \( k \) is noetherian this quotient field is also finite over $k$, whence even over \( k / \mathfrak{m} \). However, no non-trivial localizations of integral domains are finite. The only remaining possibility is that \( \mathfrak{m} \) is a maximal ideal.
\end{proof}

\begin{lemma} \label{lemma no maps to other maximal}
Let \( k \) be noetherian, and let \( E \) be a finite \( k \)-algebra which is local. Let \( \mathfrak{m} \) be as in Lemma~\ref{lemma.induces local map} above. Then \( \Hom_{k}(E, \coprod_{\substack{\mathfrak{n} \in \MaxSpec k \\ \mathfrak{n} \neq \mathfrak{m}}} I(k / \mathfrak{n})) = 0. \)
\end{lemma}

\begin{proof}
For \( x \in k \setminus \mathfrak{m} \) we observe that \( x \) becomes invertible in \( E \) by Lemma~\ref{lemma.induces local map}.

Let \( \varphi \colon E \to I(k / \mathfrak{n} ) \) for some maximal ideal \( \mathfrak{n} \neq \mathfrak{m} \). Since \( E \) is finitely generated, so is \( \Imm \varphi \). It follows that \( (\Imm \varphi) \mathfrak{n}^s = 0 \) for some \( s \). Choose \( x \in \mathfrak{n} \setminus \mathfrak{m} \). Now \( x \) acts both nilpotently and invertibly on \( \Imm \varphi \), whence \( \varphi = 0 \).
\end{proof}

\begin{lemma} \label{lemma Hom preserves injective envelope}
Let $k$ be a commutative noetherian ring, and let $E$ be a finite \( k \)-algebra which is local. Let \( I \) be as in Theorem~\ref{theorem.S triangle almost split}.

Then $\Hom_k(E,I)$ is an injective envelope of $E/\!\rad E$.
\end{lemma}

\begin{proof}
We have already established, in the proof of Lemma~\ref{lemma functors summand}, that there is a monomorphism from \( E /\!\rad E \) to \( \Hom_k(E, I) \). It follows immediately from its description that this factors through $ \Hom_k(E / \!\rad E, I) \hookrightarrow \Hom_k(E, I).$

Let \( \mathfrak{m} \) be as in Lemma~\ref{lemma.induces local map}. Note that by Lemma~\ref{lemma no maps to other maximal} we may replace \( I \) by \( I(k / \mathfrak{m}) \) without affecting the \( \Hom \)-sets.

Observe that \( \Hom_k(E / \!\rad E, I(k / \mathfrak{m})) = \Hom_k(E / \!\rad E, k/\mathfrak{m}) \), since \( \mathfrak{m} \) annihilates \( E / \!\rad E \) by construction. It follows in particular that the induced monomorphism \( E / \!\rad E \to \Hom_k(E / \!\rad E, I(k/\mathfrak{m})) \) is an isomorphism, since these two objects are finite dimensional of the same dimension over \( k / \mathfrak{m} \).

Thus we need to show that \( \Hom_k(E / \!\rad E, I(k/\mathfrak{m})) \) is an essential submodule of \( \Hom_k(E, I(k/\mathfrak{m})) \). In other words, we need to show that any non-zero submodule of $\Hom_R(E,I(k/\mathfrak{m}))$ contains a morphism which vanishes on $\rad E$.

To this end, we show that for each $\phi\in\Hom_R(E,I(k/\mathfrak{m}))$ there is some $n$ such that $\phi(\rad E)^n=0$. Note that \( E / \mathfrak{m} E \) is local with radical \( \rad E / \mathfrak{m} E \), and moreover finite dimensional over $k / \mathfrak{m}$. It follows that \( \rad E / \mathfrak{m} E \) is nilpotent, that is there is \( s \) such that \( \rad E \subseteq \mathfrak{m}^s E\). Finally, observe that there is some $t$ such that $\phi \mathfrak{m}^t=0$. Indeed, since $E$ is finitely generated over $k$, so is \( \Imm \phi \), so there is a $t$ such that $(\Imm \phi)(\rad k)^t=0$.
\end{proof}

\begin{proof}[Proof of Theorem~\ref{theorem.S triangle almost split}]
We write $\End_{\catT}(X)=E$. By Lemma~\ref{lemma Hom preserves injective envelope} we know that \( \Hom_k(E,I) \) is an injective envelope of \( E/\!\rad E \). The argument in the proof of Lemma~\ref{lemma functors summand} shows that
\[ \Hom_E(\catT(X,-), I_X) \cong \catT(X, -)^\ast. \]
Thus \( \tau X [1] \) of Theorem~\ref{theorem Krause AR-theorem} coincides with \( \mathbb{S} X \). It follows that the triangles (\ref{align Krauses AR triangle}) and (\ref{align our triangle with right almost split map}) coincide (by Theorem~\ref{theorem our triangle summand of AR triangle} or directly by comparing the two constructions).
\end{proof}

Connecting back to Section~\ref{section construction  of rel serre}, we obtain the following application.

\begin{corollary} \label{corollary Kb(Lambda) has AR}
If $\Lambda$ is an Artin algebra, then $\Kb(\modcat\Lambda)$ has almost split triangles.
\end{corollary}

\begin{proof}
We found a partial Serre functor for \( \catK(\Modcat \Lambda) \) in Section~\ref{section construction  of rel serre}, and argued in Observation~\ref{observation 0-cocompacts in Kb(Lambda)} that if we choose \( I \) to be an injective envelope of the semisimple \( k / \!\rad k \), then this functor induces an auto-equivalence on \( \Kb(\modcat \Lambda) \). As the assumptions of Theorem~\ref{theorem.S triangle almost split} are satisfied, we have almost split triangles completely inside \( \Kb(\modcat \Lambda) \), starting and ending in any object with local endomorphism ring in that subcategory.
\end{proof}

\section{Non-degeneracy} \label{section non-degeneracy}

For a partial Serre functor $\mathbb S$, there is no symmetry between the objects \( X \) and \( \mathbb{S} X \). For instance we have seen in Theorem~\ref{theorem partial Serre implies X compact and Y 0-cocompact} that \( X \) is compact, while \( \mathbb{S} X \) is only \( 0 \)-cocompact. Similarly, in the construction of almost split triangles (Theorem~\ref{theorem BIKP}) the third term is required to be compact, while the first term will typically not be cocompact. However the definition of almost split triangles is completely self-dual.

In this section we study the following concept, which will serve as a weaker but symmetric version of partial Serre duality.

\begin{definition}
  Let $X,Y\in \catT$. We say that \emph{composition from \( X \) to \( Y \) is non-degenerate} if the following conditions are satisfied.
  \begin{enumerate}
    \item For each $0\neq f \colon X\to T$ there is some $g\colon T \to Y$ such that $gf\neq 0$.
    \item For each $0\neq g \colon T\to Y$ there is some $f\colon X \to T$ such that $gf\neq 0$.
  \end{enumerate}
\end{definition}

\begin{remark}
  Composition from \( X \) to \( Y \) is non-degenerate if and only if any non-zero $\catT$-submodule of $\catT(X,-)$ or $\catT(-,Y)$ contains a non-zero map $X\to Y$.
\end{remark}

Our aim is to show that composition being non-degenerate is closely linked to almost split triangles (Theorem~\ref{theorem almost split vs non deg}) and partial Serre functors (Proposition~\ref{proposition comp to T(X,SX) is non-degenerate}). Then, in Theorem~\ref{theorem non-degeneracy implies 0-cocompactness with correct def}, we will show that even this weak notion of duality implies that the two objects are \( 0 \)-compact and \( 0 \)-cocompact, respectively.

\begin{theorem} \label{theorem almost split vs non deg}
Let $X,Y\in\catT$ be such that $\End_{\catT}(X)$ and $\End_{\catT}(Y)$ are local rings. Let \( f \colon X \to Y \) be a non-zero morphism. We denote by
\[ \Delta \colon \hspace{.5em} Y[-1] \stackrel{d}\to E \stackrel{e}\to X \stackrel{f}\to Y \]
the triangle ending in \( f \). Then the following are equivalent.
\begin{enumerate}[label=(\roman*)]
\item \( \Delta \) is an almost split triangle;
\item \( d \) is left almost split;
\item \( e \) is right almost split;
\item \( gf = 0 \) whenever \( g \) is not a split monomorphism;
\item \( fh = 0 \) whenever \( h \) is not a split epimorphism;
\item for each \( 0 \neq t \colon T \to Y \) there is \( s \colon X \to T \) such that \( ts = f \);
\item for each \( 0 \neq s \colon X \to S \) there is \( t \colon S \to Y \) such that \( ts = f \);
\item composition from \( X \) to \( Y \) is non-degenerate, and \( f \cdot \rad \End_{\catT}(X) = 0 \);
\item composition from \( X \) to \( Y \) is non-degenerate, and \( \rad \End_{\catT}(Y) \cdot f = 0 \);
\item composition from \( X \) to \( Y \) is non-degenerate, and any non-zero \( \End_{\catT}(X) \)-submodule of \( \catT(X, Y) \) contains \( f \);
\item composition from \( X \) to \( Y \) is non-degenerate, and any non-zero \( \End_{\catT}(Y)^{\op} \)-submodule of \( \catT(X, Y) \) contains \( f \).
\end{enumerate}
\end{theorem}

\begin{proof}
$(iv) \iff (v)$: Suppose $(iv)$ holds and let $h\colon H\to X$ be such that $fh\neq 0$. Consider the following diagram.
\[ \begin{tikzcd}[scale=.8]
 X \ar[r,"f"] & Y \ar[d,equal]  \\
 H \ar[r,"fh"] \ar[u,swap,"h"] & Y \ar[r,"g"] &  \Cone(fh)  \ar[r] & \text{}
\end{tikzcd} \]
Then $g$ is not split mono which by assumption implies $g f = 0$. Thus there is some $s\colon X\to H$ such that $f=fhs$, which in turn implies that $hs$ is invertible in $\End_{\catT}(X)$, since this ring is local. So $h$ is a split epimorphism. A dual argument shows that $(v)\implies (iv)$.

\smallskip \noindent
\( (iii) \iff (v) \): This is just the fact that a morphism factors through \( e \) if and only if it becomes \( 0 \) when composing with \( f \) --- a basic property of triangles.

\smallskip \noindent
\( (ii) \iff (iv) \) is the dual of \( (iii) \iff (v) \).

\smallskip
Now we have seen that \( (ii) \) to \( (v) \) are equivalent. Since \( (i) \iff (ii) \wedge (iii) \) by definition, it follows that also \( (i) \) is equivalent to these statements.

\smallskip \noindent
\( (iv) \iff (vi) \): Consider the triangle \( T \stackrel{t}\to Y \stackrel{g}\to G \to T[1] \). Note that we can construct \( t \) from \( g \) and vice versa. Moreover \( t \) is non-zero if and only if \( g \) is not split mono. Now the claimed equivalence is the fact that \( f \) factors through \( t \) if and only if it becomes zero when composing with \( g \).

\smallskip \noindent
\( (v) \iff (vii) \) is the dual of \( (iv) \iff (vi) \).

\smallskip
Now we know that \( (i) \) to \( (vii) \) are equivalent. Clearly \( (vi) \) and \( (vii) \) combined imply that composition from \( X \) to \( Y \) is non-degenerate. Moreover \( f \cdot \rad \End_{\catT}(X) = 0 \) is a special case of \( (v) \), and \( \rad \End_{\catT}(Y) \cdot f = 0 \) is a special case of \( (iv) \). Thus we know that \( (i) \) through \( (vii) \) imply \( (viii) \) and \( (ix) \).

\medskip \noindent
$(viii)\implies(xi)$: It clearly suffices to consider cyclic submodules $\End_{\catT}(Y)\cdot g$ for $0\neq g \in \catT(X,Y)$. Consider the triangle $\Cone(g)[-1] \stackrel{\alpha}\to X \stackrel{g}\to Y \to \Cone(g)$. Suppose $f\alpha \neq 0$. By non-degeneracy of composition from \( X \) to \( Y \) there is some $\beta\colon X\to \Cone(g)[-1]$ such that $f\alpha \beta \neq 0$. But since $g\neq 0$, $\alpha$ is not split epi, which implies $\alpha\beta\in \rad \End_{\catT}(X)$. This contradicts $(viii)$, hence $f\alpha=0$. Thus $f$ factors through $g$, i.e.\ $f \in \End_{\catT}(Y)\cdot g$.

\smallskip \noindent
\( (ix) \implies (x) \) is the dual of \( (viii) \implies (xi) \).

\smallskip \noindent
\( (x) \implies (vi) \): Let \( 0 \neq t \colon T \to Y \). By non-degeneracy there is a map \( s_1 \colon X \to T \) such that \( t s_1 \neq 0 \). By assumption we thus have \( f \in t s_1 \End_{\catT}(X) \), i.e.\ there is \( s_2 \in \End_{\catT}(X) \) such that \( t s_1 s_2 = f \).

\smallskip \noindent
\( (xi) \implies (vii) \) is the dual of \( (x) \implies (vi) \).
\end{proof}

In particular the above theorem says that any almost split triangle gives rise to a non-degenerate composition. In case that one of the endomorphism rings is artinian, we have the following converse.

\begin{corollary} \label{corollary non-deg implies almost split triangle}
Let \( X \) and \( Y \) be objects in \( \catT \) with local endomorphism rings, and assume that at least one of these two endomorphism rings is artinian.

If composition from \( X \) to \( Y \) is non-degenerate, then there is an almost split triangle
\[ Y[-1] \to E \to X \to Y. \]
\end{corollary}

\begin{proof}
Note that \( \catT(X, Y) \neq 0 \) by definition of non-degeneracy. Assume \( \End_{\catT}(X) \) is artinian. This implies that \( \rad \End_{\catT}(X) \) is nilpotent. It follows that there is some non-zero \( f \in \catT(X, Y) \) such that \( f \cdot \rad \End_{\catT}(X) = 0 \). The claim now follows from implication \( (viii) \implies (i) \) in Theorem~\ref{theorem almost split vs non deg} above.
\end{proof}

\begin{proposition} \label{proposition comp to T(X,SX) is non-degenerate}
    Let $\mathbb S\colon\catX\to\catT$ be a partial Serre functor. Then composition from \( X \) to \( \mathbb S X \) is non-degenerate for each $X\in\catX$.
\end{proposition}

\begin{proof}
Let us start with a non-zero morphism \( f \colon X \to T \), and complete it to a triangle \( \Cone(f)[-1] \to X \stackrel{f}\to T \to \Cone(f) \). By the naturality of the isomorphism defining partial Serre duality we have the following commutative square.
\[ \begin{tikzcd}
\catT(X, \mathbb S X) \ar[r,"\cong"] \ar[d] & \catT(X, X)^\ast \ar[d] \\
\catT(\Cone(f)[-1], \mathbb S X) \ar[r,"\cong"] & \catT(X, \Cone(f)[-1])^\ast
\end{tikzcd} \]
Since \( f \) is non-zero the map \( \catT(X, \Cone(f)[-1]) \to \catT(X, X) \) is not onto, and hence its dual is not mono. It follows that the left vertical map above is not mono either, that is there is a non-zero map from \( X \) to \( \mathbb{S} X \) such that the comosition with \( \Cone(f)[-1] \to X \) vanishes. It follows that this map factors through \( f \).

  Now take a non-zero $g\colon T\to\mathbb SX$. By assumption we have a natural isomorphism \[\phi\colon\catT(-,\mathbb SX)\to\catT(X,-)^\ast.\] Let $\eta=\phi_T(g)$. Then $\eta$ is non-zero, so in particular there is some $f\colon X\to T$ such that $\eta(f)\neq 0$. We claim that $gf\neq 0$. Of course, it suffices to show that $\phi_X(gf)\neq0$. But by the commutative diagram
  \[ \begin{tikzcd}
  \catT(T,\mathbb SX) \ar[r,"\phi_T"] \ar[d]& \catT(X,T)^\ast \ar[d]  \\
  \catT(X,\mathbb SX) \ar[r,"\phi_X"] & \catT(X,X)^\ast
  \end{tikzcd} \]
  we have $\phi_X(gf)=\eta(f\circ -)$, which is non-zero since $\phi_X(gf)(\id_X)=\eta(f)\neq 0$.
  \end{proof}

  \begin{remark}
    An object $X$ may have several `non-degenerate partners'. Indeed, if $\mathbb S X$ and $\tau X$ exist, then composition from $X$ to either is non-degenerate. However, in general $\tau X$ is only a summand of $\mathbb S X$.
  \end{remark}

\begin{theorem} \label{theorem non-degeneracy implies 0-cocompactness with correct def}
Let $X,Y\in\catT$ be such that composition from \( X \) to \( Y \) is non-degenerate. Then $X$ is $0$-compact and $Y$ is $0$-cocompact.
\end{theorem}

The proof of this result relies on the following observation.

\begin{lemma} \label{lemma non-degeneracy implies covariant and contravariant ghosts coincide}
Let $X,Y\in\catT$ be such that composition from \( X \) to \( Y \) is non-degenerate. Then the following statements hold.
\begin{enumerate}
\item An object is a covariant $X$-ghost if and only if it is a contravariant $Y$-ghost.
\item A morphism $f\colon S\to T$ is a covariant $X$-ghost if and only if it is a contravariant $Y$-ghost.
\end{enumerate}
\end{lemma}
\begin{proof}
Since composition from \( X \) to \( Y \) is non-degenerate, $(1)$ is clear and
\begin{align*}
\text{$f$ is a covariant $X$-ghost} & \text{$\iff$ $f\alpha=0$ for each $\alpha\colon X\to S$} \\
&\text{$\iff$ $\beta f\alpha=0$ for each $\alpha\colon X\to S$ and $\beta\colon T\to Y$} \\
&\text{$\iff$ $\beta f=0$ for each $\beta\colon T\to Y$}\\
&\text{$\iff$ $f$ is a contravariant $Y$-ghost.} \qedhere
\end{align*}
\end{proof}

\begin{proof}[Proof of Theorem~\ref{theorem non-degeneracy implies 0-cocompactness with correct def}] We show that $Y$ is $0$-cocompact; the proof that $X$ is $0$-compact is dual.

Take a sequence \[\mathbb T \colon \cdots \to T_2\to T_1\to T_0\] in $\catT$ such that $\catT(\mathbb T[-1], Y)$ is dual ML and $\colim \catT(\mathbb T, Y)=0$. It suffices to show that $\catT(X,\holim \mathbb T)$ vanishes. As in the proof of Theorem~\ref{theorem partial Serre implies X compact and Y 0-cocompact} there is a short exact sequence \[0\to\limit^1\catT(X,\mathbb T[-1])\to\catT(X,\holim\mathbb T)\to\limit\catT(X,\mathbb T)\to 0,\] so we need only prove that the outer terms are zero.

We first show that $\limit\catT(X,\mathbb T)$ vanishes. So assume to the contrary that there is some $\left(\dots, \phi_2,\phi_1,\phi_0\right)\in\limit\catT(X,\mathbb T)$ with $\phi_i\neq0$. Then, by non-degeneracy of composition from \( X \) to \( Y \), there is some $\psi\colon T_i\to Y$ such that $\psi\phi_i\neq0$. But by assumption, the image of $\psi$ in $\colim\catT(\mathbb T,Y)$ vanishes, that is the composition $T_j\to T_i\stackrel{\psi}\to Y$ is zero for sufficiently large $j$. In particular, the non-zero $\psi\phi_i$ factors through the zero morphism $T_j\to Y$, as indicated by the following diagram, and we have a contradiction.
\[ \begin{tikzcd}[scale=.8]
 X \ar[dr,swap,"\phi_j"] \ar[r,"\phi_i"] & T_i \ar[r,"\psi"] & Y \\
 & T_j \ar[u]
\end{tikzcd} \]

Let us now show that $\limit^1\catT(X,\mathbb T[-1])=0$. It suffices to demonstrate that \[\catT(X,\mathbb T[-1])= \cdots \to\catT(X,T_2[-1])\stackrel{t_2}\to\catT(X,T_1[-1])\stackrel{t_1}\to\catT(X,T_0[-1])\] is ML. Assume to the contrary that for some $k$, the sequence of subgroups
\[\Imm t_k \supset\Imm t_k t_{k+1} \supset \cdots\]
does not stabilize. Without loss of generality, we may assume that $k=0$ and that each image is properly contained in the previous one. In other words, for each $i$ there is some $\phi_i\colon X\to T_0[-1]$ such that $\phi_i$ factors through $T_i[-1]$, say via $\psi_i$, but not through $T_{i+1}[-1]$. The following diagram, with the bottom row a triangle,

\[ \begin{tikzcd}[scale=.8]
 X \ar[r,"\psi_i"]\ar[dr,"\phi_i"]\ar[d,dashed,"\nexists"] & T_i[-1] \ar[d]  \\
 T_{i+1}[-1] \ar[r]  & T_0[-1] \ar[r] &  \Cone  \ar[r] & T_{i+1}
\end{tikzcd} \]
reveals that the composition \[X\stackrel{\psi_i}\to T_i[-1]\to T_0[-1]\to \Cone\] is non-zero. By non-degeneracy of composition from \( X \) to \( Y \), there is some non-zero \[X\stackrel{\psi_i}\to T_i[-1]\to T_0[-1]\to \Cone\to Y.\] In particular, we can find a morphism $\omega_i\colon T_0[-1]\to Y$ such that the composition $T_i[-1]\to T_0[-1] \stackrel{\omega_i}\to Y$ is non-zero, while $T_{i+1}[-1]\to T_0[-1] \stackrel{\omega_i}\to Y$ does vanish. In other words, in the commutative diagram
\[ \begin{tikzcd}
 \Ker_i \ar[r,>->] &  \catT(T_0[-1],Y)  \ar[d,equal] \ar[r,] & \catT(T_i[-1],Y) \ar[d] \\
 \Ker_{i+1} \ar[r,>->] &  \catT(T_0[-1],Y) \ar[r,] & \catT(T_{i+1}[-1],Y)  \\
\end{tikzcd} \]
with exact rows, we have $\omega_i\in\Ker_{i+1}\setminus\Ker_i$. In particular, the sequence \[\Ker_1\subsetneq\Ker_2\subsetneq\Ker_3\subsetneq\cdots\] does not stabilize, contradicting the assumption that $\catT(\mathbb T[-1], Y)$ is dual ML.
\end{proof}

\begin{corollary} \label{corollary almost split forces 0-(co)compactness}
Let \( X \to Y \to Z \to X[1] \) be an almost split triangle in a triangulated category. Then \( X \) is \( 0 \)-cocompact and \( Z \) is \( 0 \)-compact.
\end{corollary}

\begin{proof}
This is an immediate consequence of Theorem~\ref{theorem almost split vs non deg} and Theorem~\ref{theorem non-degeneracy implies 0-cocompactness with correct def}.
\end{proof}

\appendix

\section{Dual Brown representability} \label{section dual brown rep}

The aim of this appendix is to give a proof of a `constructive' version of dual Brown representability for triangulated categories with enhancements, cogenerated by a set of \( 0 \)-cocompact objects. Note that it was already pointed out by Modoi in \cite{modoi2020weight} that these categories do satisfy dual Brown representability, so our original contribution here is only the explicit description of the representing objects.

Throughout this appendix, \( \catT \) is a triangulated category at the base of a stable derivator. Moreover \( \catT \) is cogenerated by a set of \( 0 \)-cocompact objects. By Lemma~\ref{lemma products} we know that products of \( 0 \)-cocompact objects are \( 0 \)-cocompact again, whence we may assume that \( \catT \) is cogenerated by a single \( 0 \)-cocompact object \( S \), which we may moreover assume to be invariant under suspension.

With this setup, we will prove the following.

\begin{theorem} \label{thm.brown}
Let \( F \colon \catT \to \Ab \) be a homological functor commuting with products. For any commutative diagram
\[ \begin{tikzcd}
\catT(T_0, -) \ar[r,"{\catT(f_1,-)}"] \ar[drr] & \catT(T_1, -) \ar[r,"{\catT(f_2,-)}"] \ar[dr] & \catT(T_2, -) \ar[r,"{\catT(f_3,-)}"] \ar[d] & \cdots \\
&& F
\end{tikzcd} \]
such that all the induced maps \( \Imm \catT(f_i, S) \to F(S) \) are isomorphisms, we have
\[ F \cong \catT( \holim T_i, -). \]
In particular \( \catT \) satisfies dual Brown representability.
\end{theorem}

The proof of this theorem is less direct than one might imagine: We first show, in Proposition~\ref{prop representing object for brown}, that if \( F \) already is representable then we do get the desired isomorphism of functors. Then we show, in Proposition~\ref{prop epi for brown}, that in general \( F \) is at least an epimorphic image of \( \catT( \holim T_i, -) \). The argument for this part comes from \cite{MR3085026}. Finally we complete the proof by employing a trick of Neeman's \cite{MR2529296}.

While large parts of the argument are available in the literature, we found that the varying notation made it slightly challenging to read the entire proof. Therefore we believe it might be worthwhile to give a complete account here.

\medskip

Let us start with two brief observations translating our assumptions on \( F \).

\begin{observation} \label{obs prod coprod for nat to functor}
For any set of objects \( T_i \), any natural transformation from \( \coprod \catT(T_i, -) \) to \( F \) factors uniquely through \( \coprod \catT(T_i, -) \to \catT( \prod T_i, -) \).

Indeed  we have the following commutative square
\[ \begin{tikzcd}
\{ \coprod \catT(T_i, -)  \to F \} \ar[r] \ar[d,"\cong"] & \{ \catT(\prod  T_i, -) \to F \} \ar[d, "\cong"] \\
\prod F(T_i) \ar[r] & F( \prod T_i)
\end{tikzcd} \]
where the lower horizontal map is an isomorphism by assumption.
\end{observation}

\begin{observation} \label{obs homological weak cockernel}
For any triangle \( T_1 \to T_2 \to T_3 \to T_1[1] \), and any natural transformation \( \catT(T_2,-) \to F \) such that composition \( \catT(T_3, -) \to \catT(T_2, -) \to F \) vanishes, there is a factorization as indicated by the following diagram.
\[ \begin{tikzcd}
\catT(T_1, -) \ar[rd,dashed,swap,"\exists"] & \catT(T_2, -) \ar[l] \ar[d] & \catT(T_3, -) \ar[l] \ar[dl,-,dotted,bend right,xshift=-.3em,yshift=-.3em,shorten >=.6em] \\
& F
\end{tikzcd} \]
This follows from the fact that \( F \) is cohomological by employing the Yoneda lemma.
\end{observation}

Of course the ``in particular'' part of Theorem~\ref{thm.brown} above only follows if we can find at least one such diagram, given \( F \). There is a straight-forward way of doing so:

\begin{construction} \label{const functor sequence for Brown}
Pick an epimorphism of \( \End(S) \)-modules \( \End(S)^{(I_0)} \to F(S) \). Equivalently, we have a natural transformation \( \catT(S, -)^{(I_0)} \to F \) which induces an epimorphism on \( S \). By Observation~\ref{obs prod coprod for nat to functor} this gives rise to a natural transformation \( \catT(S^{I_0}, -) \to F \) which induces an epimorphism on \( S \). We pick \( T_0 = S^{I_0} \) and this natural transfomation.

Now assume a natural transformation \( \psi_i \colon \catT(T_i, - ) \to F \) inducing an epimorphism \( \psi_i^S \colon \catT(T_i, S) \to F(S) \) is already constructed. Pick an epimorphism of \( \End(S) \)-modules \( \End(S)^{(I_{i+1})} \to \Ker \psi_i^S \). Similarly to the first step, this gives rise to a natural transformation \( \catT( S^{I_{i+1}}, -) \to \catT(T_i, -) \) such that the sequence
\[ \catT(S^{I_{i+1}}, -) \to \catT(T_i, -) \xto{\psi_i} F \]
is exact on \( S \). In particular the composition vanishes. Picking  \( T_{i+1} \) to be the object fitting in the triangle \( T_{i+1} \to T_i \to S^{I_{i+1}}) \to T_{i+1}[1] \) we can employ Observation~\ref{obs homological weak cockernel} and obtain a natural transformation \( \catT(T_{i+1}, -) \to F \). Moreover, evaluating at \( S \) we can use the exactness observed above to conclude that we have the desired image.
\end{construction}

As an intermediate step towards Theorem~\ref{thm.brown} we will prove the following.

\begin{proposition} \label{prop representing object for brown}
Theorem~\ref{thm.brown} holds under the additional assumption that \( F \) is representable.
\end{proposition}

The proof is based on the following result due to Keller and Nicol\'as \cite{MR3031826}.

\begin{theorem} \label{theorem Keller-Nicolas}
Let \( \catT \) be a triangulated category at the base of a stable derivator. Given a commutative diagram
\[ \begin{tikzcd}
\cdots \ar[r] & X_2 \ar[r] \ar[d,"\varphi_2"] & X_1 \ar[r] \ar[d,"\varphi_1"] & X_0 \ar[d,"\varphi_0"] \\
\cdots \ar[r] & Y_2 \ar[r] & Y_1 \ar[r] & Y_0
\end{tikzcd} \]
there is a choice of cone morphisms
\[ \cdots \to \Cone(\varphi_2) \to \Cone(\varphi_1) \to \Cone(\varphi_1) \]
such that there is a triangle
\[ \holim X_i \to \holim Y_i \to \holim \Cone(\varphi_i) \to \holim X_i[1]. \]
\end{theorem}

Moreover, we will utilize the following observation.

\begin{lemma}\label{lemma 0-cocompacts implies ghosts closed under holim}
Let \( \cdots \to X_2 \to X_1 \to X_0 \) be a sequence of contravariant \( S \)-ghosts. Then \( \holim X_i = 0 \).
\end{lemma}

\begin{proof}
The vanishing of all maps implies that the sequence
\[ \cdots \longleftarrow \catT(X_2, S) \longleftarrow \catT(X_1, S) \longleftarrow \catT(X_0, S) \]
has vanishing colimit and is dual ML. It follows, since \( S \) is \( 0 \)-cocomapact and invariant under suspension, that \( \holim X_i \) is an \( S \)-ghost. But since \( S \) is a cogenerator, this means that \( \holim X_i = 0 \).
\end{proof}

\begin{proof}[Proof of Proposition~\ref{prop representing object for brown}]
Let \( F = \catT(X, -) \). Thus, via the inverse of the Yoneda functor, we have the commutative diagram
\[ \begin{tikzcd}
\cdots \ar[r,equal] & X \ar[r,equal] \ar[d,"\psi_2"] & X \ar[r,equal] \ar[d,"\psi_1"] & X \ar[d,"\psi_0"] \\
\cdots \ar[r] & T_2 \ar[r] & T_1 \ar[r] & T_0
\end{tikzcd} \]
Applying Theorem~\ref{theorem Keller-Nicolas}, we obtain a triangle
\[ \holim X \to \holim T_i \to \holim \Cone(\psi_i) \to \holim X[1] \]
for suitable cone morphisms. Clearly \( \holim X = X \), so it remains to show that \( \holim \Cone(\psi_i) = 0 \).

By definition, a cone morphism makes the following diagram commutative.
\[ \begin{tikzcd}
X \ar[r] & T_i \ar[r] & \Cone \psi_i \ar[r] & X[1] \\
X \ar[r] \ar[u,equal] & T_{i+1} \ar[r] \ar[u] & \Cone \psi_{i+1} \ar[r] \ar[u] & X[1] \ar[u,equal]
\end{tikzcd} \]
Applying \( \catT(-, S) \) this turns into
\[ \begin{tikzcd}
\catT( \Cone \psi_i, S) \ar[r,>->] \ar[d] & \catT(T_i, S) \ar[r,->>] \ar[d] & \catT(X, S) \ar[d,equal] \\
\catT( \Cone \psi_{i+1}, S) \ar[r,>->] & \catT(T_{i+1}, S) \ar[r,->>] & \catT(X, S).
\end{tikzcd} \]
To see this, note that the epimorphisms follow from the fact that the image of the middle vertical map is \( \catT(X, S) \) by assumption. Since \( S \) is assumed to be invariant under suspension, it follows that we also get the claimed monomorphisms.

Again invoking the fact that the image of the middle vertical map is \( \catT(X, S) \), we see that the left vertical map needs to vanish. In other words, all our cone morphisms are contravariant \( S \)-ghosts. Now the claim follows from Lemma~\ref{lemma 0-cocompacts implies ghosts closed under holim}.
\end{proof}

Now we return to the general situation, where \( F \) is not assumed to be representable a priori.

\begin{proposition} \label{prop epi for brown}
In the situation of Theorem~\ref{thm.brown}, there is a natural epimorphism \( \catT(\holim T_i, -) \to F \).
\end{proposition}

This result, as well as the argument here are based on \cite[Theorem~8]{MR3085026}.

\medskip
For the proof we will need to show that, for any given \( X \in \catT \), the induced map \( \catT( \holim T_i, X) \to F(X) \) is surjective.

Given \( X \), we consider the functor \( \catT(X, -) \), and construct the diagram of functors
\[ \begin{tikzcd}
\catT(X_0, -) \ar[r,"{\catT(g_1,-)}"] \ar[drr] & \catT(X_1, -) \ar[r,"{\catT(g_2,-)}"] \ar[dr] & \catT(X_2, -) \ar[r,"{\catT(g_3,-)}"] \ar[d] & \cdots \\
&& \catT(X, -)
\end{tikzcd} \]
as described in Construction~\ref{const functor sequence for Brown}. In particular \( X_0 = S^{I_0} \), and \( \Cone g_i = S^{I_i} \) for suitable sets \( I_i \). By Proposition~\ref{prop representing object for brown} we know that \( X = \holim X_i \).

With this setup, the key step in the proof of Proposition~\ref{prop epi for brown} is the following.

\begin{lemma} \label{lemma for brown epi}
Given a sequence of maps \( \catT(X_i, -) \to F \) making the solid part of the following diagram commutative, we can find the dashed arrows making the entire diagram commutative.
\[ \begin{tikzcd}
\catT(X_0, -) \ar[r,"{\catT(g_1,-)}"] \ar[drrrr] \ar[dd,"{\catT(\varphi_0, -)}",dashed] & \catT(X_1, -) \ar[r,"{\catT(g_2,-)}"] \ar[drrr] \ar[dd,"{\catT(\varphi_1,-)}",dashed]& \catT(X_2, -) \ar[r,"{\catT(g_3,-)}"] \ar[drr] \ar[dd,"{\catT(\varphi_3,-)}",dashed]& \cdots \\
&&&& F \\
\catT(T_0, -) \ar[r,swap,"{\catT(f_1,-)}"] \ar[urrrr] & \catT(T_1, -) \ar[r,swap,"{\catT(f_2,-)}"] \ar[urrr] & \catT(T_2, -) \ar[r,swap,"{\catT(f_3,-)}"] \ar[urr] & \cdots
\end{tikzcd} \]
\end{lemma}

\begin{proof}
Constructing from left to right, observe first that we can find \( \varphi_0 \) since \( X_0 = S^{I_0} \) and \( \catT(T_0, S^{I_0}) \to F(S^{I_0}) \) is a surjection.

Next, assume we have already constructed \( \varphi_i \). Note that we have a triangle \( X_{i+1} \to X_i \to S^{I_{i+1}} \to X_{i+1} \). Thus, in order to obtain a map \( \varphi_{i+1} \) making the square between \( \catT(\varphi_i, -) \) and \( \catT(\varphi_{i+1}, -) \) commutative, it suffices to show that the composition
\[ T_{i+1} \xto{f_{i+1}} T_i \xto{\varphi_i} X_i \to S^{I_{i+1}} \]
vanishes. Equivalently we may consider the sequence
\[ \catT(S^{I_{i+1}}, -) \to \catT(X_i, -) \xto{\catT(\varphi_i, -)} \catT(T_i, -) \xto{\catT(f_{i+1}, -)} \catT(T_{i+1}, -), \]
and moreover it suffices to consider the evaluation at \( S \). Now note that, by assumption, \( \Imm \catT(f_{i+1}, S) = F(S) \), so the above vanishing is implied by the vanishing of the composition
\begin{align*}
& [ \catT(S^{I_{i+1}}, S) \to \catT(X_i, S) \xto{\catT(\varphi_i, S)} \catT(T_i, S) \to F(S) ] \\
& = [ \catT(S^{I_{i+1}}, S) \to \catT(X_i, S) \xto{\catT(g_{i+1}, S)} \catT(X_{i+1}, S) \to F(S) ],
\end{align*}
which holds since the first two maps come from consecutive maps in a triangle.

Note however that at this point we cannot be sure that the triangle involving \( \catT(\varphi_{i+1}, -) \) and \( F \) commutes. Let
\[ \delta = [ \catT(X_{i+1}, -) \to F] - [\catT(T_{i+1}, -) \to F] \circ \catT( \varphi_{i+1}, -) \]
be the obstruction to the triangle commuting. Of course \( \delta \circ \catT(g_{i+1}, -) = 0 \). Since \( F \) is homological, it follows that \( \delta \) factors through the map \[ \catT(X_{i+1}, -) \to \catT(S^{I_{i+1}}[-1], -), \] say via \( \delta' \). Since \( S \) is assumed to be invariant under suspension we may disregard the shift. Now recall that \( \catT(T_{i+1}, S) \to F(S) \) is a surjection. It follows that any map \( \catT(S^{I_{i+1}}[-1], - ) \to F \), in particular  \( \delta' \), factors through \( \catT(T_{i+1}, -) \to F \). Thus we find a map \( \delta'' \colon T_{i+1} \to S^{I_{i+1}}[-1] \) making the following diagram commutative.
\[ \begin{tikzcd}
\catT(X_{i+1}, -) \ar[rr] && \catT(S^{I_{i+1}}[-1], -) \ar[dd,"{\delta'}"] \ar[lldd,swap,very near end,"{\catT(\delta'',-)}"] \\
&& \\
\catT(T_{i+1}, -) \ar[rr] && F \ar[from=uull,near start,"\delta",crossing over]
\end{tikzcd} \]
It follows that we can replace \( \varphi_{i+1} \) by \( \varphi_{i+1} + [S^{I_{i+1}}[-1] \to X_{i+1}] \circ \delta'' \), fixing the commutativity of the triangle involving \( \catT(\varphi_{i+1}, -) \) and \( F \), while not affecting the previously commutative square involving \( \catT(\varphi_i, -) \) and \( \catT(\varphi_{i+1}, -) \).
\end{proof}

\begin{remark} \label{rem T_i is approx}
The iterative construction in the proof of Lemma~\ref{lemma for brown epi} does not actually require the entire diagram. In particular any map \( \catT(X_i, -) \to F \) factors through \( \catT(\varphi_i, -) \).
\end{remark}

Now we are ready to prove the proposition.

\begin{proof}[Proof of Proposition~\ref{prop epi for brown}]
Let \( X \in \catT \). We have \( X = \holim X_i \), with the sequence \( X_i \) as discussed directly below the proposition. Recall that \( F \) is a homological functor commuting with products. Thus the triangle \[ \holim X_i \to \prod X_i \to \prod X_i \to \holim X_i [1] \] gives rise to the exact sequence
\[ \prod F(X_i[-1]) \to \prod F(X_i[-1]) \to F(\holim X_i) \to \prod F(X_i) \to \prod F(X_i). \]
By definition, the kernel of the last map is \( \limit F(X_i) \), while the cokernel of the first map is \( \limit^1 F(X_i[-1]) \). Thus we have the lower sequence in the following diagram. The upper sequence exists by the same argument applied to the functor \( \catT(\holim T_i, -) \).
\[ \begin{tikzcd}
\limit^1 \catT(\holim T_i, X_i[-1]) \ar[r,>->] \ar[d] & \catT(\holim T_i, \holim X_i) \ar[r,->>] \ar[d] & \limit \catT(\holim T_i, X_i) \ar[d] \\
\limit^1 F(X_i[-1]) \ar[r,>->] & F( \holim X_i) \ar[r,->>] & \limit F(X_i)
\end{tikzcd} \]
In order to show that the middle vertical map is surjective it suffices to show that the outer two are surjective.

We first consider the left hand side. Note that Remark~\ref{rem T_i is approx} holds analogously for \( X_i[-1] \), so all the maps \( \catT(T_i, X_i) \to F(X_i) \) are surjective. In particular the same holds for the maps \( \catT(\holim T_i, X_i) \to F(X_i) \). Now the left vertical map is onto by right exactness of \( \limit^1 \).

Next we look at the right hand side. Note that an element of \( \limit F(X_i) \) is a sequence of elements \( x_i \in F(X_i) \) such that \( F(g_i)(x_i) = x_{i-1} \). Translating via the Yoneda lemma these are maps \( x_i \colon \catT(X_i, -) \to F \) such that assumptions of Lemma~\ref{lemma for brown epi} are satisfied. By that lemma we obtain maps \( \varphi_i \colon T_i \to X_i \). Now \( (\varphi_i \circ [\holim T_i \to T_i] ) \) is an element of \( \limit \catT( \holim T_i, X_i ) \), and moreover a preimage of \( (x_i) \).
\end{proof}

Now we are ready to complete the proof of Theorem~\ref{thm.brown}. The missing piece is \cite[Theorem 1.3]{MR2529296}.

\begin{proof}[Proof of Theorem~\ref{thm.brown}]
By Proposition~\ref{prop epi for brown} there is a natural epimorphism \[ \catT( \holim T_i, -) \to F. \] Note that its kernel again satisfies the assumptions of Proposition~\ref{prop epi for brown}, and thus is an epimorphic image of some \( \catT( \holim T_i', -) \). In particular \( F \) has a projective presentation as indicated in the left half of the following diagram.
\[ \begin{tikzcd}
\catT( \holim T_i', -) \ar[rr] && \catT(\holim T_i, -) \ar[rd,->>] \ar[rr] && \catT(C, -) \ar[ld,dashed,bend left] \\
&&& F \ar[ru,>->]
\end{tikzcd} \]
Let \( C \) denote the cocone of the map \( \colim T_i \to \colim T_i' \). We have the exact sequence as in the diagram above from this triangle. Finally, note that since \( F \) is homological we get the dashed morphism making the triangle to its left commutative. It follows that \( F \) is a direct summand of \( \catT(C, -) \), hence is representable.

Now the claim of Theorem~\ref{thm.brown} follows from Proposition~\ref{prop representing object for brown}.
\end{proof}

\section{Exactness of partial Serre functors} \label{section exactness of serre}

The aim of this appendix is to prove the following theorem, summing up functorial properties of partial Serre functors.

\begin{theoremappendix}
  Suppose \( \catT \) is a triangulated category, and let \( \catX \) be the full subcategory of all objects \( X \) such that \( \catT(X, -)^\ast \) is representable.

  Then \( \catX \) is a triangulated subcategory of \( \catT \), and there is a partial Serre functor \( \mathbb S \colon \catX \to \catT \). Moreover, \( \mathbb S \) is a triangle functor.
\end{theoremappendix}

Most of this theorem is actually fairly easily seen. In fact, the existence of a partial Serre functor \( \mathbb S \colon \catX \to \catT \) follows easily from our assumption on representability --- see Observation~\ref{obs.dualizable_gives_functor} below. The main technical challenges in proving the theorem are showing that \( \catX \) is triangulated, which we will show in Corollary~\ref{cor.dualizable_is_triang} for the case that \( \catT \) is idempotent closed and in the very last subsection in general, and that \( \mathbb S \) is a triangle functor---see Theorem~\ref{thm.Serre_is_triangle}.

\subsection*{General observations on partial Serre functors}

\begin{observation} \label{obs.dualizable_gives_functor}
Let \( \catX \) be a subcategory of \( \catT \) such that for any object \( X \in \catX \) the functor \( \catT(X, -)^\ast \) is representable. Fix for all \( X \) a representing object \( \mathbb S X \) and a natural isomorphism
\[ \eta_X \colon \catT(X, -)^\ast \to \catT(-, \mathbb S X ). \]

Then \( \mathbb S \) defines a functor \( \catX \to \catT \) by requiring the following square of functors to be commutative for any morphism \( f \colon X_1 \to X_2 \) in \( \catX \).
\[ \begin{tikzcd}
\catT(X_1, -)^\ast \ar[r,"f \cdot -"] \ar[d,"\eta_{X_1}"] & \catT(X_2, -)^\ast \ar[d,"\eta_{X_2}"] \\
\catT(-, \mathbb S X_1) \ar[r,"\mathbb S f \circ -"] & \catT(-, \mathbb S X_2)
\end{tikzcd} \]
Note that the lower natural transformation exists and is unique since \( \eta_{X_1} \) is an isomorphism, and that it is uniquely representable as composition with some map by the Yoneda Lemma.

It follows directly from the commutative square defining \( \mathbb S f \) that \( \eta \) becomes a natural isomorphism of functors on \( \catX \times \catT^{\op} \), i.e.\ that \( \mathbb S \) is a partial Serre functor.
\end{observation}

\begin{remark}
One may see that a Serre functor on \( \catX \) is unique up to unique natural isomorphism. Indeed, if \( \widetilde{\mathbb S} \) is a different choice of a partial Serre functor on \( \catX \), with corresponding natural isomorphism \( \widetilde{\eta} \colon \catT(-, -)^\ast \to \catT(- , \widetilde{\mathbb S} -) \), then \( \widetilde{\eta} \circ \eta^{-1} \) is a natural isomorphism \( \catT(-, \mathbb S - ) \to \catT(-, \widetilde{\mathbb S} - ) \), which, by the Yoneda lemma, comes from a natural isomorphism \( \mathbb S \to \widetilde{ \mathbb S } \).
\end{remark}

\begin{observation} \label{obs.dualizable_closed_automorphism}
Let \( \catT \) be a \( k \)-category, and \( [1] \) be an automorphism of \( \catT \). If \( \catT(X, -)^\ast \) is representable for some object \( X \), then so is \( \catT(X[1], -)^\ast \):
\[ \catT(X[1], -)^\ast \cong \catT(X, - )^\ast \circ [-1] \cong \catT(-, \mathbb S X) \circ [-1] \cong \catT(-, (\mathbb S X) [1]). \] In particular, the subcategory of all objects \( X \) such that \( \catT(X, -)^\ast \) is representable, is invariant under all automorphisms of \( \catT \).
\end{observation}

Note however that we may not be able to choose \( \eta_{X[1]} \) to be induced by \( \eta_X \) consistently over the entire subcategory \( \catX \). To account for this, we observe that at least there is a natural isomorphism controlling the difference between the two.

\begin{observation} \label{obs.zeta}
Let \( \mathbb S \) be a partial Serre functor on a subcategory \( \catX \) of \( \catT \). Assume \( \catT \) has an automorphism \( [1] \), and that \( \catX \) is invariant under \( [1] \). Then there is a unique natural isomorphism \( \zeta \colon \mathbb S \circ [1] \to [1] \circ \mathbb S \) such that the following diagram of isomorphisms commutes functorially in \( X \) and in \( T \).
\[ \begin{tikzcd}
\catT(X, T)^\ast \ar[d,"\eta_{X, T}"] && \catT(X[1], T[1])^\ast  \ar[ll,swap,"{[1]}"] \ar[d,"\eta_{X[1],T[1]}"] \\
\catT(T, \mathbb S X) \ar[rd,"{[1]}"] && \catT(T[1], \mathbb S(X[1])) \ar[ld,swap,"{\catT(T[1], \zeta_X)}"]\\
& \catT(T[1], (\mathbb S X)[1])
\end{tikzcd} \]
\end{observation}

\subsection*{\( \catX \) is triangulated}

Denote by \( \Modcat \catT \) the category of functors \( \catT^{\op} \to \Modcat k \), and by \( \modcat \catT \) the subcategory of finitely presented functors. Recall that by \cite[Theorem~3.1]{MR0211399} the category \( \modcat \catT \) is Frobenius abelian, and its injectives are precisely the direct summands of representable functors. (Note that in \cite{MR0211399} the objects of \( \modcat \catT \) are described as images of morphisms of representable functors, rather than as cokernels of such morphisms. However, since we can always complete triangles there is no difference between these categories.) Therefore we are particularly interested in injective functors.

\begin{lemma}
Let \( \catT \) be an additive category, and let \( T \) be  an object in $\catT$. Then \( \catT(T, -)^\ast \) is injective in \( \Modcat \catT \).
\end{lemma}

\begin{proof}
We observe that for \( F \in \Modcat \catT \) we have the natural isomorphism
\[ (\Modcat \catT)(F, \catT(T, -)^\ast) \stackrel{\cong}\longleftrightarrow F(T)^\ast. \]
The map from left to right is given by sending a natural transformation \( \eta \) to the composition \( \left( F(T) \stackrel{\eta_T}\to \catT(T, T)^\ast \xrightarrow{\ev_{\id}} I \right) \). The map from right to left sends a linear form \( \phi \) to the natural transformation
\begin{align*}
F(X) & \to \catT(T, X)^\ast \\
f & \longmapsto [ t \mapsto (\phi \circ F(t))(f)].
\end{align*}
Since \( F(T)^\ast \) is exact in \( F \) it follows that \( \catT(T, -)^\ast \) is injective.
\end{proof}

\begin{proposition} \label{prop.representable=fp}
Suppose \( \catT \) is triangulated and idempotent complete, and let \( T \in \catT \). Then the functor \( \catT(T, -)^\ast \) is representable if and only if it is finitely presented.
\end{proposition}

\begin{proof}
Clearly any representable functor is finitely presented.

Assume conversely that \( \catT(T, -)^\ast \) is finitely presented. Since it is injective even in \( \Modcat \catT \), it is also injective in \( \modcat \catT \). Now, since $\catT$ is assumed to be idempotent complete, the claim follows from \cite[Proposition~15.1]{MR224090}.
\end{proof}

\begin{theorem} \label{thm.cone_dualizable}
Let \( \catT \) be triangulated and idempotent complete, and take a triangle \( X \to Y \to Z \to X[1] \) in $\catT$.

If both \( \catT(X, -)^\ast \) and \( \catT(Y, -)^\ast \) are representable, then \( \catT(Z, -)^\ast \) is representable.
\end{theorem}

\begin{proof}
By Proposition~\ref{prop.representable=fp}, it suffices to show that \( \catT(Z, -)^\ast \) is finitely presented.

Consider the exact sequence
\[ \catT(X, -)^\ast \to \catT(Y, -)^\ast \to \catT(Z, -)^\ast \to \catT(X[1], -)^\ast \to \catT(Y[1],-)^\ast. \]
Since the leftmost two terms are finitely presented---in fact they are representable---it follows that the image of the map \( \catT(Y, -)^\ast \to \catT(Z, -)^\ast \) is finitely presented. Since the rightmost two terms are finitely presented by Observation~\ref{obs.dualizable_closed_automorphism}, and \( \modcat \catT \) is abelian, we also have that the image of the map \( \catT(Z, -)^\ast \to \catT(X[1], -)^\ast \) is finitely presented. Now the claim follows, since extensions of finitely presented functors are finitely presented.
\end{proof}

\begin{corollary} \label{cor.dualizable_is_triang}
Let \( \catT \) be triangulated and idempotent complete. Then the collection of objects \( X \) such that \( \catT(X, -)^\ast \) is representable, is a triangulated subcategory of \( \catT \).
\end{corollary}

\subsection*{\( \mathbb S \) is a triangle functor}

\begin{proposition} \label{prop.dualize_TR3}
Let \( \catT \) be a category with an automorphism \( [1] \), and let \( \catX \subset \catT \) be a \( [1] \)-invariant subcategory admitting a partial Serre functor \( \mathbb S \). Let

\[ X \stackrel{x}\to Y \stackrel{y}\to Z \stackrel{z}\to X[1] \text{ \,and \, } T \stackrel{t}\to U \stackrel{u}\to V \stackrel{v}\to T[1] \] be sequences of objects and morphisms in \( \catX \) and \( \catT \), respectively.

Assume that for any \( f \) and \( g \) in the following diagram, there is a morphism \( h \) making the diagram commutative.
\[ \begin{tikzcd}
X \ar[r,"x"] \ar[d,"\forall f"] & Y \ar[r,"y"] \ar[d,"\forall g"] & Z \ar[r,"z"] \ar[d,dashed,"\exists h"] & X[1] \ar[d,"{f[1]}"] \\
T \ar[r,"t"] & U \ar[r,"u"] & V \ar[r,"v"] & T[1]
\end{tikzcd} \] Then also in the following diagram we have that for any \( f \) and \( g \) there is \( h \) making it commutative.
\[ \begin{tikzcd}
U \ar[r,"u"] \ar[d,dashed,"\exists h"] & V \ar[r,"v"] \ar[d,"\forall f"] & T[1] \ar[r,"{t[1]}"] \ar[d,"\forall g"] & U[1] \ar[d,dashed,"{h[1]}"] \\
\mathbb S X \ar[r,"\mathbb S x"] & \mathbb S Y \ar[r,"\mathbb S y"] & \mathbb S Z \ar[r,"\zeta_X \circ \mathbb S z"] & (\mathbb S X)[1]
\end{tikzcd} \]
\end{proposition}

\begin{proof}
The assumption may be reformulated into the statement that the morphism between the kernels in the following commutative diagram, is onto.
\[ \begin{tikzcd}
 \Ker_1 \ar[r,>->] \ar[d,"\text{epi}"] & \catT(X, T) \oplus \catT(Y, U) \oplus \catT(Z, V) \ar[r] \ar[d,->>] & \catT(X,U) \oplus \catT(Y,V) \oplus \catT(Z, T[1]) \ar[d,->>] \\
 \Ker_2 \ar[r,>->] & \catT(X, T) \oplus \catT(Y, U) \ar[r] & \catT(X,U)
\end{tikzcd} \]
Here the vertical maps are projections to the respective summands, and the horizontal maps are given by
\[ \left[ \begin{matrix} t \circ \star & - (\star \circ x) & 0 \\ 0 & u \circ \star & - (\star \circ y) \\ - (\star [1] \circ z) & 0 & v \circ \star \end{matrix} \right] \text{ and \,} [ t \circ \star \; - (\star \circ x) ], \]
respectively. Here \( t \circ \star \) stands for \( f \longmapsto t \circ f \), and similar.

We may consider the kernels of the vertical projections, and a cokernel morphism as indicated in the following diagram.
\[ \begin{tikzcd}
& \catT(Z, V) \ar[r] \ar[d,>->] & \catT(Y, V) \oplus \catT(Z, T[1]) \ar[r,->>] \ar[d,>->] & \Cok_0 \ar[d,"\text{mono}" swap] \\
 \Ker_1 \ar[r,>->] \ar[d,"\text{epi}"] & \begin{smallmatrix} \catT(X, T) \oplus \catT(Y, U) \\ \oplus \catT(Z, V) \end{smallmatrix} \ar[r] \ar[d,->>] & \begin{smallmatrix} \catT(X, U) \oplus \catT(Y, V) \\ \oplus \catT(Z, T[1]) \end{smallmatrix} \ar[d,->>] \ar[r,->>] & \Cok_1 \\
 \Ker_2 \ar[r,>->] & \catT(X, T) \oplus \catT(Y, U) \ar[r] & \catT(X, U)
\end{tikzcd} \]
By the Snake Lemma we observe that the map \( \Ker_1 \to \Ker_2 \) being epi is equivalent to the map \( \Cok_0 \to \Cok_1 \) being mono.

Dualizing the upper two rows and the rightmost three columns of this diagram, and identifying via the natural isomorphism \( \eta \) defining the partial Serre functor, we obtain the diagram
\[ \begin{tikzcd}
 \Cok^\ast_1 \ar[r,>->] \ar[d,"\text{epi}"] & \begin{smallmatrix} \catT(U, \mathbb S X) \oplus \catT(V, \mathbb S Y) \\ \oplus \catT(T[1], \mathbb S Z)
\end{smallmatrix} \ar[r] \ar[d] & \begin{smallmatrix}\catT(T, \mathbb S X) \oplus \catT(U, \mathbb S Y) \\ \oplus \catT(V, \mathbb S Z) \end{smallmatrix} \ar[d] \\
 \Cok^\ast_0 \ar[r,>->] & \catT(V,\mathbb SY) \oplus \catT(T[1], \mathbb S Z) \ar[r] & \catT(V, \mathbb S Z)
\end{tikzcd} \]
with the horizontal maps given by
\[  \left[ \begin{matrix} \star \circ t & 0 & - (\zeta_X \circ \mathbb S z \circ \star)[-1] \\ - (\mathbb S x \circ \star) & \star \circ u & 0 \\ 0 & -(\mathbb S y \circ \star) & \star \circ v \end{matrix} \right] \text{ and } [- \mathbb S y \circ \star \; \star \circ v] . \]
Most of these entries are immediate from the naturality of \( \eta \), only the term in the right upper corner warrants further explanation. Note that the map in question is induced via \( \eta \) by the composition along the upper row of the following diagram. Hence it is precisely the composition appearing in the lower row of that diagram.
\[ \begin{tikzcd}[cramped]
\catT(Z, T[1])^\ast \ar[r,"(\star \circ z)^\ast"] \ar[d,"\eta_{Z,T[1]}"] & \catT(X[1], T[1])^\ast \ar[rr,"{[1]^{\ast}}"] \ar[d,"\eta_{X[1],T[1]}"] && \catT(X, T)^\ast \ar[d,"\eta_{X,T}"] \\
\catT(T[1], \mathbb S Z) \ar[r,"\mathbb{S} z \circ \star"] & \catT(T[1], \mathbb S (X[1])) \ar[rd,"\zeta_X \circ \star"] && \catT(T, \mathbb S X)  \ar[ld,swap,"{[1]}"] \\
&& \catT(T[1], (\mathbb S X)[1])
\end{tikzcd} \]
Here the left square commutes by the naturality of \( \eta \), and the pentagon on the right comes from Observation~\ref{obs.zeta}.

Now we can translate the statement that the morphism between kernels is surjective, back to a commutative diagram: It means that for any \( f \in \catT(V, \mathbb S Y) \) and \( g \in \catT(T[1], \mathbb S Z) \) such that \( \mathbb{S} y \circ f = g \circ v \)---i.e.\ for each element of the lower kernel---there is \( h \in \catT(U, \mathbb S X) \) such that \( h \circ t = (\zeta_X \circ \mathbb{S} z \circ g)[-1] \) and \( \mathbb{S} x \circ h = f \circ u \) --- i.e.\ a preimage in the upper kernel. In diagrammatic language we thus have
\[ \begin{tikzcd}
U \ar[r,"u"] \ar[d,dashed,"\exists h"] & V \ar[r,"v"] \ar[d,"\forall f"] & T[1] \ar[r,"{t[1]}"] \ar[d,"\forall g"] & U[1] \ar[d,dashed,"{h[1]}"] \\
\mathbb S X \ar[r,"\mathbb S x"] & \mathbb S Y \ar[r,"\mathbb S y"] & \mathbb S Z \ar[r,"\zeta_X \circ \mathbb S z"] & (\mathbb S X)[1]
\end{tikzcd} \]
as desired.
\end{proof}

\begin{theorem} \label{thm.Serre_is_triangle}
Let \( \catT \) be a triangulated category, and let \( \catX \) be a triangulated subcategory admitting a partial Serre functor \( \mathbb S \).

Then \( \mathbb S \colon \catX \to \catT \) is a triangle functor.
\end{theorem}

\begin{proof}
The first ingredient to a triangle functor is a natural isomorphism \[ \mathbb{S} \circ [1] \to [1]\circ \mathbb S. \] We have already constructed such a natural isomorphism \( \zeta \) in Observation~\ref{obs.zeta}. However, here we will choose the natural isomorphism \( - \zeta \).

It remains for us to show that for any triangle \( X \stackrel{x}\to Y \stackrel{y}\to Z \stackrel{z}\to X[1] \) in \( \catX \), the sequence
\[ \mathbb S X \stackrel{\mathbb S x}\to \mathbb S Y \stackrel{\mathbb S y}\to \mathbb S Z \xrightarrow{(- \zeta_X) \circ \mathbb S z} (\mathbb S X)[1] \]
is a triangle in $\catT$.

Note that there exists a triangle \(  (\mathbb S Z)[-1] \to U \to \mathbb S Y \stackrel{\mathbb S y}\to \mathbb S Z \). By the axioms of triangulated categories, any two triangles satisfy the assumptions of Proposition~\ref{prop.dualize_TR3}. We will apply that proposition to our original triangle and the one involving \( U \). Choosing the two free vertical maps in the conclusion to be identities, we obtain the commutative diagram
\[ \begin{tikzcd}
U \ar[r,"u"] \ar[d,dashed,"\exists h"] & \mathbb S Y \ar[r,"\mathbb S y"] \ar[d,"\id"] & \mathbb S Z \ar[r,"{t[1]}"] \ar[d,"\id"] & U[1] \ar[d,dashed,"{h[1]}"] \\
\mathbb S X \ar[r,"\mathbb S x"] & \mathbb S Y \ar[r,"\mathbb S y"] & \mathbb S Z \ar[r,"\zeta_X \circ \mathbb S z"] & (\mathbb S X)[1]
\end{tikzcd} \]

Let \( T \in \catT \). Applying \( \catT(T, -) \) to the above diagram we obtain
\[\begin{tikzpicture}
  \node (tt1) at (0,8) {$\catT(T, (\mathbb S Y)[-1])$};
  \node (tt2) at (5,8) {$\catT(T, \mathbb S (Y[-1]))$};
  \node (tt3) at (9,8) {$\catT(Y[-1], T)^\ast$};
  \node (t1) at (0,6.5) {$\catT(T, (\mathbb S Z)[-1])$};
  \node (t2) at (5,6.5) {$\catT(T, \mathbb S (Z[-1]))$};
  \node (t3) at (9,6.5) {$\catT(Z[-1], T)^\ast$};
  \node (c1) at (0,5) {$\catT(T, U)$};
  \node (c2) at (5,5) {$\catT(T, \mathbb S X)$};
  \node (c3) at (9,5) {$\catT(X, T)^\ast$};
  \node (b1) at (0,3.5) {$\catT(T, \mathbb S Y)$};
  \node (b2) at (5,3.5) {$\catT(T, \mathbb S Y)$};
  \node (b3) at (9,3.5) {$\catT(Y, T)^\ast$};
  \node (bb1) at (0,2) {$\catT(T, \mathbb S Z)$};
  \node (bb2) at (5,2) {$\catT(T, \mathbb S Z)$};
  \node (bb3) at (9,2) {$\catT(Z, T)^\ast$.};
  \draw[->] (tt1) to node[right]{\scriptsize $(\mathbb Sy)[-1]\circ\star$}(t1);
  \draw[->] (t1) to node[right]{\scriptsize $ t\circ\star$}(c1);
  \draw[->] (c1) to node[right]{\scriptsize $ u\circ\star$}(b1);
  \draw[->] (b1) to node[right]{\scriptsize $ \mathbb Sy\circ\star$}(bb1);
  \draw[->] (tt2) to node[right]{\scriptsize $\mathbb S(y[-1])\circ\star$}(t2);
  \draw[->] (t2) to node[right]{\scriptsize $ \mathbb S(z[-1])\circ\star$}(c2);
  \draw[->] (c2) to node[right]{\scriptsize $ \mathbb Sx\circ\star$}(b2);
  \draw[->] (b2) to node[right]{\scriptsize $ \mathbb Sy\circ\star$}(bb2);
  \draw[->] (tt3) to node[right]{\scriptsize $y[-1]\cdot\star$}(t3);
  \draw[->] (t3) to node[right]{\scriptsize $ z[-1]\cdot\star$}(c3);
  \draw[->] (c3) to node[right]{\scriptsize $ x\cdot\star$}(b3);
  \draw[->] (b3) to node[right]{\scriptsize $ y\cdot\star$}(bb3);
  \draw[->] (tt1) to node[above]{\scriptsize $\zeta_{Y[-1]}[-1]\circ\star$}(tt2);
  \draw[->] (tt3) to node[above]{\scriptsize $ \eta_{Y[-1],T}$}  node[below]{\scriptsize$\cong$} (tt2);
  \draw[->] (t1) to node[above]{\scriptsize $\zeta_{Z[-1]}[-1]\circ\star$}(t2);
  \draw[->] (t3) to node[above]{\scriptsize $ \eta_{Z[-1],T}$}  node[below]{\scriptsize$\cong$} (t2);
  \draw[->] (c1) to node[above]{\scriptsize $h\circ\star$}(c2);
  \draw[->] (c3) to node[above]{\scriptsize $ \eta_{X,T}$}  node[below]{\scriptsize$\cong$} (c2);
  \draw[->] (b1) to node[above]{\scriptsize $\id$}(b2);
  \draw[->] (b3) to node[above]{\scriptsize $ \eta_{Y,T}$} (b2);
  \draw[->] (bb1) to node[above]{\scriptsize $\id$}(bb2);
  \draw[->] (bb3) to node[above]{\scriptsize $ \eta_{Z,T}$} (bb2);
\end{tikzpicture}\]
%\[ \begin{tikzcd}
%\catT(T, \mathbb (S Y)[-1]) \ar[r,"{(\mathbb S y)[-1] \circ \star}"] \ar[d,"{\zeta_{Y[-1]}[-1]} \circ \star"] & \catT(T, \mathbb (S Z)[-1]) \ar[r,"t \circ \star"] \ar[d,"{\zeta_{Z[-1]}[-1]} \circ \star"] & \catT(T, U) \ar[r,"u \circ \star"] \ar[d,"h \circ \star"] & \catT(T, \mathbb S Y) \ar[r,"\mathbb S y \circ \star"] \ar[d,"\id"] & \catT(T, \mathbb S Z) \ar[d,"\id"] \\
%\catT(T, \mathbb S (Y[-1])) \ar[r,"{\mathbb S(y[-1]) \circ \star}"] & \catT(T, \mathbb S(Z[-1])) \ar[r,"{\mathbb S(z[-1]) \circ \star}"] & \catT(T, \mathbb S X) \ar[r,"\mathbb S x \circ \star"] & \catT(T, \mathbb S Y) \ar[r,"\mathbb S y \circ \star"] & \catT(T, \mathbb S Z) \\
%\catT(Y[-1], T)^\ast \ar[r,"{y[-1] \cdot \star}"] \ar[u,"\eta_{Y[-1],T}"',"\cong"] & \catT(Z[-1], T)^\ast \ar[r,"{z[-1] \cdot \star}"] \ar[u,"\eta_{Z[-1],T}"',"\cong"] & \catT(X, T)^\ast \ar[r,"x \cdot \star"] \ar[u,"\eta_{X,T}"',"\cong"] & \catT(Y, T)^\ast \ar[r,"y \cdot \star"] \ar[u,"\eta_{Y,T}"',"\cong"] & \catT(Z, T)^\ast \ar[u,"\eta_{Z,T}"',"\cong"]
%\end{tikzcd} \]
The leftmost upper square commutes by naturality of \( \zeta \). For the second left square from the top, note that \( \mathbb S( z [-1] ) \circ \zeta_{Z[-1]}[-1] = \zeta_X[-1] \circ (\mathbb S z)[-1] \) by naturality of \( \zeta \), and this in turn is equal to  \( h \circ t \).

The leftmost column of this diagram is exact, and so is the rightmost column, since it is the dual of an exact sequence. Thus the middle column is also exact, so the Five lemma applies, and tells us that \( h \circ \star \) is an isomorphism. By the Yoneda lemma this implies that \( h \) is an isomorphism.

%The leftmost upper square commutes by naturality of \( \zeta \). For the second upper square note that \( \mathbb S( z [-1] ) \circ \zeta_{Z[-1]}[-1] = \zeta_X[-1] \circ (\mathbb S z)[-1] \) by naturality of \( \zeta \), and the right hand side in turn is equal to  \( h \circ t \).

%The top row of this diagram is exact, and so is the bottom row, since it it the dual of an exact sequence. Thus the midle row is also exact, so the Five Lemma applies, and tells us that \( h \circ \star \) is an isomorphism. By the Yoneda Lemma this implies that \( h \) is an isomorphism.

Now
\[ \mathbb S X \stackrel{ \mathbb S x }\to \mathbb S Y \stackrel{\mathbb S y}\to \mathbb S Z \xrightarrow{(- \zeta_X) \circ \mathbb S z} (\mathbb S X)[1] \]
is isomorphic to the triangle \( U \stackrel{u}\to \mathbb S Y \stackrel{\mathbb S y}\to \mathbb S Y \xrightarrow{-t[1]} U[1] \), hence it is a triangle itself.
\end{proof}

\subsection*{The case that \( \catT \) in not idempotent closed} \label{sect.idempotent_closure}

By \cite{MR1813503}, any triangulated category \( \catT \) is canonically embedded into an idempotent closed triangulated category \( \widehat \catT \). More precisely, \( \widehat \catT \) is the category of injective objects in \( \modcat \catT \).

We denote by \( \catX \) and \( \widehat \catX \) the subcategories of objects $X$ and $\widehat X$ such that \( \catT(X, -)^\ast \) and \( \widehat \catT (\widehat X, -)^\ast \) are representable, respectively. One easily observes, using the fact that any object in \( \widehat \catT \) is a direct summand of an object in \( \catT \), that \( \catX \subset \widehat \catX \). Note that by Theorem~\ref{thm.cone_dualizable} we know that \( \widehat \catX \) is a triangulated subcategory of \( \widehat \catT \).

Now let \( X \to Y \to Z \to X[1] \) be a triangle in \( \catT \), such that \( \catT(X, -)^\ast \) and \( \catT(Y, -)^\ast \) are representable, say by \( \mathbb S X \) and \( \mathbb S Y \). By the above comment we know that \( \widehat \catT(Z, -)^\ast \) is representable, say by \( \widehat {\mathbb S Z} \in \widehat \catT \). Finally we apply Theorem~\ref{thm.Serre_is_triangle}, which tells us that \( \mathbb S X \to \mathbb S Y \to \widehat{ \mathbb S Z } \to \mathbb S X [1] \) is a triangle in \( \widehat \catT \). But since \( \catT \) is a triangulated subcategory of \( \widehat \catT \), and two of the terms of the triangle lie in \( \catT \), so does \( \widehat{ \mathbb S Z} \). Thus we have shown that \( \catT(Z, -)^\ast \) is representable.

\bibliographystyle{amsplain}
\bibliography{main}

\providecommand{\bysame}{\leavevmode\hbox to3em{\hrulefill}\thinspace}
\providecommand{\MR}{\relax\ifhmode\unskip\space\fi MR }
% \MRhref is called by the amsart/book/proc definition of \MR.
\providecommand{\MRhref}[2]{%
  \href{http://www.ams.org/mathscinet-getitem?mr=#1}{#2}
}
\providecommand{\href}[2]{#2}
\begin{thebibliography}{10}

\bibitem{MR2927802}
Takuma Aihara and Osamu Iyama, \emph{Silting mutation in triangulated
  categories}, J. Lond. Math. Soc. (2) \textbf{85} (2012), no.~3, 633--668.

\bibitem{MR1044344}
Maurice Auslander and Ragnar-Olaf Buchweitz, \emph{The homological theory of
  maximal {C}ohen-{M}acaulay approximations}, no.~38, 1989, Colloque en
  l'honneur de Pierre Samuel (Orsay, 1987), pp.~5--37.

\bibitem{Ballard}
Matthew Ballard, \emph{Derived categories of singular schemes with an
  application to reconstruction}, Adv. Math. \textbf{227} (2011), no.~2,
  895--919.

\bibitem{MR1813503}
Paul Balmer and Marco Schlichting, \emph{Idempotent completion of triangulated
  categories}, J. Algebra \textbf{236} (2001), no.~2, 819--834.

\bibitem{MR2079606}
Apostolos Beligiannis, \emph{Auslander-{R}eiten triangles, {Z}iegler spectra
  and {G}orenstein rings}, $K$-Theory \textbf{32} (2004), no.~1, 1--82.

\bibitem{MR2327478}
Apostolos Beligiannis and Idun Reiten, \emph{Homological and homotopical
  aspects of torsion theories}, Mem. Amer. Math. Soc. \textbf{188} (2007),
  no.~883, viii+207.

\bibitem{MR3914144}
Dave Benson, Srikanth~B. Iyengar, Henning Krause, and Julia Pevtsova,
  \emph{Local duality for representations of finite group schemes}, Compos.
  Math. \textbf{155} (2019), no.~2, 424--453.

\bibitem{BIKP}
\bysame, \emph{Local duality for the singularity category of a finite
  dimensional gorenstein algebra}, Nagoya Mathematical Journal (2020), 1–24.

\bibitem{MR1039961}
A.~I. Bondal and M.~M. Kapranov, \emph{Representable functors, {S}erre
  functors, and reconstructions}, Izv. Akad. Nauk SSSR Ser. Mat. \textbf{53}
  (1989), no.~6, 1183--1205, 1337.

\bibitem{MR1818984}
Alexei Bondal and Dmitri Orlov, \emph{Reconstruction of a variety from the
  derived category and groups of autoequivalences}, Compositio Math.
  \textbf{125} (2001), no.~3, 327--344.

\bibitem{MR405403}
Edgar~H. Brown, Jr. and Michael Comenetz, \emph{Pontrjagin duality for
  generalized homology and cohomology theories}, Amer. J. Math. \textbf{98}
  (1976), no.~1, 1--27.

\bibitem{MR2249625}
Aslak~Bakke Buan, Robert Marsh, Markus Reineke, Idun Reiten, and Gordana
  Todorov, \emph{Tilting theory and cluster combinatorics}, Adv. Math.
  \textbf{204} (2006), no.~2, 572--618.

\bibitem{buchweitz1986maximal}
Ragnar-Olaf Buchweitz, \emph{Maximal {C}ohen--{M}acaulay modules and
  {T}ate-cohomology over {G}orenstein rings}, unpublished (1986), available at
  http://hdl.handle.net/1807/16682.

\bibitem{MR106918}
P.~M. Cohn, \emph{On the free product of associative rings}, Math. Z.
  \textbf{71} (1959), 380--398.

\bibitem{MR0150183}
Albrecht Dold and Dieter Puppe, \emph{Homologie nicht-additiver {F}unktoren.
  {A}nwendungen}, Ann. Inst. Fourier Grenoble \textbf{11} (1961), 201--312.

\bibitem{MR1914985}
Paul~C. Eklof and Alan~H. Mekler, \emph{Almost free modules}, revised ed.,
  North-Holland Mathematical Library, vol.~65, North-Holland Publishing Co.,
  Amsterdam, 2002, Set-theoretic methods.

\bibitem{MR3537821}
Ioannis Emmanouil, \emph{On pure acyclic complexes}, J. Algebra \textbf{465}
  (2016), 190--213.

\bibitem{MR0211399}
Peter Freyd, \emph{Stable homotopy}, Proc. {C}onf. {C}ategorical {A}lgebra
  ({L}a {J}olla, {C}alif., 1965), Springer, New York, 1966, pp.~121--172.

\bibitem{MR910167}
Dieter Happel, \emph{On the derived category of a finite-dimensional algebra},
  Comment. Math. Helv. \textbf{62} (1987), no.~3, 339--389.

\bibitem{MR1112170}
\bysame, \emph{On {G}orenstein algebras}, Representation theory of finite
  groups and finite-dimensional algebras ({B}ielefeld, 1991), Progr. Math.,
  vol.~95, Birkh\"auser, Basel, 1991, pp.~389--404.

\bibitem{MR224090}
Alex Heller, \emph{Stable homotopy categories}, Bull. Amer. Math. Soc.
  \textbf{74} (1968), 28--63.

\bibitem{MR597688}
Yasuo Iwanaga, \emph{On rings with finite self-injective dimension. {II}},
  Tsukuba J. Math. \textbf{4} (1980), no.~1, 107--113.

\bibitem{MR3031826}
Bernhard Keller and Pedro Nicol\'as, \emph{Weight structures and simple dg
  modules for positive dg algebras}, Int. Math. Res. Not. IMRN (2013), no.~5,
  1028--1078.

\bibitem{MR907948}
Bernhard Keller and Dieter Vossieck, \emph{Sous les cat\'egories d\'eriv\'ees},
  C. R. Acad. Sci. Paris S\'er. I Math. \textbf{305} (1987), no.~6, 225--228.

\bibitem{MR1803642}
Henning Krause, \emph{Auslander-{R}eiten theory via {B}rown representability},
  vol.~20, 2000, Special issues dedicated to Daniel Quillen on the occasion of
  his sixtieth birthday, Part IV, pp.~331--344.

\bibitem{MR2157133}
\bysame, \emph{The stable derived category of a {N}oetherian scheme}, Compos.
  Math. \textbf{141} (2005), no.~5, 1128--1162.

\bibitem{MR2923949}
\bysame, \emph{Approximations and adjoints in homotopy categories}, Math. Ann.
  \textbf{353} (2012), no.~3, 765--781.

\bibitem{MR3085026}
George~Ciprian Modoi, \emph{The dual of {B}rown representability for homotopy
  categories of complexes}, J. Algebra \textbf{392} (2013), 115--124.

\bibitem{modoi2020weight}
\bysame, \emph{Weight structures cogenerated by weak cocompact objects}, 2020.

\bibitem{MR1191736}
Amnon Neeman, \emph{The connection between the {$K$}-theory localization
  theorem of {T}homason, {T}robaugh and {Y}ao and the smashing subcategories of
  {B}ousfield and {R}avenel}, Ann. Sci. \'Ecole Norm. Sup. (4) \textbf{25}
  (1992), no.~5, 547--566.

\bibitem{MR1308405}
\bysame, \emph{The {G}rothendieck duality theorem via {B}ousfield's techniques
  and {B}rown representability}, J. Amer. Math. Soc. \textbf{9} (1996), no.~1,
  205--236.

\bibitem{MR1812507}
\bysame, \emph{Triangulated categories}, Annals of Mathematics Studies, vol.
  148, Princeton University Press, Princeton, NJ, 2001.

\bibitem{MR2439608}
\bysame, \emph{The homotopy category of flat modules, and {G}rothendieck
  duality}, Invent. Math. \textbf{174} (2008), no.~2, 255--308.

\bibitem{MR2529296}
\bysame, \emph{Brown representability follows from {R}osick\'{y}'s theorem}, J.
  Topol. \textbf{2} (2009), no.~2, 262--276.

\bibitem{MR2680406}
\bysame, \emph{Some adjoints in homotopy categories}, Ann. of Math. (2)
  \textbf{171} (2010), no.~3, 2143--2155.

\bibitem{MR2875857}
\bysame, \emph{Non-left-complete derived categories}, Math. Res. Lett.
  \textbf{18} (2011), no.~5, 827--832. \MR{2875857}

\bibitem{MR3212862}
\bysame, \emph{The homotopy category of injectives}, Algebra Number Theory
  \textbf{8} (2014), no.~2, 429--456.

\bibitem{MR3946864}
Steffen Oppermann, Chrysostomos Psaroudakis, and Torkil Stai, \emph{Change of
  rings and singularity categories}, Adv. Math. \textbf{350} (2019), 190--241.

\bibitem{Rouquier}
Raphael Rouquier, \emph{Dimensions of triangulated categories}, J.\ K-theory
  \textbf{1} (2008), 193--256.

\bibitem{MR1453167}
Jean-Louis Verdier, \emph{Des cat\'egories d\'eriv\'ees des cat\'egories
  ab\'eliennes}, Ast\'erisque (1996), no.~239, xii+253 pp. (1997), With a
  preface by Luc Illusie, Edited and with a note by Georges Maltsiniotis.

\bibitem{MR3220541}
Jan \v{S}\v{t}ov\'{\i}\v{c}ek, \emph{On purity and applications to coderived
  and singularity categories}, preprint, arXiv:1412.1615.

\bibitem{MR1269324}
Charles~A. Weibel, \emph{An introduction to homological algebra}, Cambridge
  Studies in Advanced Mathematics, vol.~38, Cambridge University Press,
  Cambridge, 1994.

\bibitem{MR3473427}
Yuefei Zheng and Zhaoyong Huang, \emph{On pure derived categories}, J. Algebra
  \textbf{454} (2016), 252--272.

\end{thebibliography}

\end{document}